%% file: Doktorarbeit_Verbesserungen.tex
\documentclass[10pt,a4paper,plainheadsepline, headsepline,abstracton,twoside,bibliography=totoc]{scrreprt}
\pdfoutput=1

\usepackage[english]{babel}
\usepackage[utf8]{inputenc}
\usepackage{textcomp}  
\usepackage[T1]{fontenc}
\usepackage{lmodern}
\usepackage{url}

\usepackage{graphicx}
\usepackage[automark]{scrpage2}
\usepackage[a4paper,left=3.5cm, right=2.5cm,top=3cm, bottom=3cm]{geometry}

\usepackage{tikz}

\usepackage{color}

\usepackage{amsthm,amsmath,amssymb,amsfonts}
\usepackage{mathrsfs}

\newtheorem{Theorem}{Theorem}[chapter]
\newtheorem*{thma}{Theorem}
\newtheorem{Lemma}[Theorem]{Lemma}
\newtheorem{Proposition}[Theorem]{Proposition}
\newtheorem{Corollary}[Theorem]{Corollary}

\theoremstyle{definition} 

\newtheorem{Remark}[Theorem]{Remark}
\newtheorem{Definition}[Theorem]{Definition}

\newtheorem{Conjecture}[Theorem]{Conjecture}

\newenvironment{Proof}{\begin{proof}[Proof]}{\end{proof}}

\newcommand{\R}{\mathbb{R}}
\newcommand{\C}{\mathbb{C}}
\newcommand{\N}{\mathbb{N}}
\newcommand{\Z}{\mathbb{Z}}
\newcommand{\hyp}{\R^n \setminus \R^l}

\newcommand{\ph}{\varphi}
\newcommand{\Om}{\Omega}
\newcommand{\eps}{\varepsilon}
\newcommand{\Sc}{{\cal S}}

\DeclareMathOperator{\rloc}{rloc}
\DeclareMathOperator{\rinf}{rinf}
\DeclareMathOperator{\tr}{tr}
\DeclareMathOperator{\Ext}{Ext}
\DeclareMathOperator{\dist}{dist}

\newcommand{\AR}[1][s]{A_{p,q}^{#1}(\R^n)}
\newcommand{\BR}[1][s]{B_{p,q}^{#1}(\R^n)}
\newcommand{\FR}[1][s]{F_{p,q}^{#1}(\R^n)}
\newcommand{\Fpp}[2]{F_{p,p}^{#1}(#2)}

\newcommand{\bspq}{\BR}
\newcommand{\fspq}{\FR}

\newcommand{\fO}[1][\Z^{\Om}]{f_{p,q}^{s}(#1)}
\newcommand{\AO}[1][\Om]{A_{p,q}^s(#1)}
\newcommand{\BO}[1][\Om]{B_{p,q}^s(#1)}
\newcommand{\FO}[1][\Om]{F_{p,q}^s(#1)}

\newcommand{\FtBar}[1][\Om]{\tilde{F\,}\!_{p,q}^s(\bar{#1})}
\newcommand{\At}[1][\Om]{\tilde{A\,\,}\!\!_{p,q}^s(#1)}
\newcommand{\Bt}[1][\Om]{\tilde{B\,}\!_{p,q}^s(#1)}
\newcommand{\Ft}[1][\Om]{\tilde{F\,}\!_{p,q}^s(#1)}

\newcommand{\Fo}[1][\Om]{\mathring{F\,\,}\!\!_{p,q}^s(#1)}

\newcommand{\Frinf}[1][\Om]{F_{p,q}^{s,\rinf}(#1)}

\newcommand{\Frloc}[1][\Om]{F_{p,q}^{s,\rloc}(#1)}

 \usetikzlibrary{patterns}

\newcommand{\sint}{\lfloor s\rfloor}
\newcommand{\srest}{\{s\}}

\newcommand{\Lint}{\lfloor L\rfloor}
\newcommand{\Lrest}{\{L\}}
\newcommand{\rint}{\lfloor \rho\rfloor}

\newcommand{\lip}[1][\sigma]{lip^{#1}(\R^n)}
\newcommand{\Kap}[1][L]{\varkappa_{#1}}
\newcommand{\hold}[1][s]{{\cal C}^{#1}(\R^n)}

\begin{document}

\pagestyle{plain}
	\ofoot[\pagemark]{\pagemark}
\automark[section]{chapter} 

\pagestyle{empty}
 	\input{Titelseite}
	\newpage

\input{Acknowledgements}

	\newpage
	\tableofcontents
	\newpage

\chapter*{Zusammenfassung}
\addcontentsline{toc}{chapter}{Zusammenfassung}
 	\input{Einfuehrung}

\chapter*{Introduction}
\addcontentsline{toc}{chapter}{Introduction}
 	\input{Introduction}

\newpage
	\pagestyle{scrheadings}
	\ofoot[\pagemark]{\pagemark}
\automark[section]{chapter} 
	\ohead[\headmark]{\headmark}

\chapter{Preliminaries}
	\input{Grundlagen}

	\newpage
 	\chapter[Atomic representations]{Atomic representations in function spaces and applications to pointwise multipliers and diffeomorphisms, a new approach}
	\input{Pointwise}

	\newpage
 	\chapter{Decomposition theorems for function spaces on domains}
	\input{Zerlegungstheorem}

	\newpage
 	\chapter{Wavelets for reinforced function spaces on cubes}
	\input{Wavelets_Wuerfel}

	\newpage
 	\chapter{Discussion and open problems}
	\input{Open_Problems}

\bibliographystyle{amsalpha}
\bibliography{ben}

\newpage
\chapter*{Ehrenwörtliche Erklärung}
\thispagestyle{empty}
\pagestyle{empty}
\input{Erklaerung}
\newpage
\thispagestyle{empty}
\input{CV_de_a}

\end{document}

%% file: Titelseite.tex
\titlepage
\begin{center}
\includegraphics[scale=0.5]{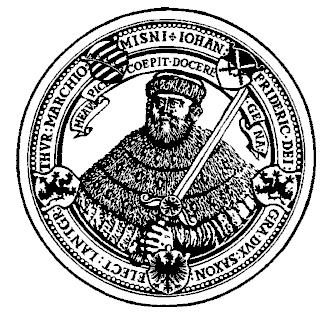}
\vspace{2.0cm}

{\huge \textbf{Wavelets in function spaces on \\[1ex] cellular domains}}\\ [15ex]

{\LARGE \sc Dissertation}\\[5ex] 
zur Erlangung des akademischen Grades\\[2ex]
{\Large doctor rerum naturalium (Dr. rer. nat.)}\\[5ex]

\vspace{2.5cm}

{\large
	vorgelegt dem Rat der\\[2ex]
	Fakultät für Mathematik und Informatik\\[1ex]
	der Friedrich-Schiller-Universit\"at Jena
	\bigskip
	\bigskip
	
	von Dipl.-Math.\ Benjamin Scharf \\
	\smallskip
	geb.\ am 29.11.1985 in Jena\\[1ex]
	
}

\end{center}

\newpage
\begin{center}
\end{center}
\vfill
\large \textbf{Gutachter \\[3ex]
1. Prof. Dr. Hans-Jürgen Schmeißer, Jena \\[2.5ex]
2. Prof. Dr. Dr. h.c. Hans Triebel, Jena	\\[2.5ex]
3. Prof. Dr. Leszek Skrzypczak, Poznan, Polen \\[4ex]
Tag der öffentlichen Verteidigung: 14.02.2013 }

%% file: Acknowledgements.tex
\newenvironment{acknowledgements}{%
  \renewcommand*{\abstractname}{Acknowledgements} \abstract}{%
  \endabstract
}

\begin{acknowledgements}
At first I would like to express my deepest gratitude to my two supervisors, Prof. Hans-J\"urgen Schmei\ss er and Prof. Hans Triebel. They both encouraged me to deal with this topic and always showed their faith in me. Whenever I had something on my mind they were there to listen and discuss with me. They also gave me the freedom to look left and right for different results than the ones we had in mind at the beginning of my PhD-education. I also thank them for helpful discussions, remarks, suggestions and the possibility to get to know a lot of people dealing with function spaces and their applications.

Furthermore, I like to thank Prof.\ Dorothee Haroske, Dr.\ Henning Kempka, Dr.\ Cornelia Schneider, Prof.\ Winfried Sickel, Prof. Leszek Skrzypczak, Dr.\ Tino Ullrich and Dr.\ Jan Vybiral for helpful comments and discussions dealing mainly with pointwise multipliers, local means and non-smooth atomic decompositions.

I would like to give thanks to my colleagues Dr.\ Markus Hansen, Prof.\ Aicke Hinrichs, Dr.\	 Maryia Kabanava, Lev Markhasin, Therese Mieth, Marcel Rosenthal, Philipp Skandera, Paolo Di Tella, Mario Ullrich and Markus Weimar for the conversations, discussions and also a lot of nice trips to conferences and summer schools. 

Moreover, I like to thank Studienstiftung des Deutschen Volkes for financial support and miscellaneous opportunities as well as the people from Wurzel e.\,V. Jena.

Last but not least I thank my family, especially my parents, and my girlfriend Janin for their unbreakable support in every situation of life. Behind the curtains of every such project there are a lot of things where their support was essential.
\end{acknowledgements}

%% file: Einfuehrung.tex
Heutzutage spielt die Theorie und Anwendung von Wavelet-Zerlegungen nicht nur in der Untersuchung von Funktionenräumen (von Lebesgue-, Hardy-, Sobolev-, Besov-, Triebel-Lizorkin-Typ), sondern auch in ihren Anwendungen in Signal- und numerischer Analysis, partiellen Differentialgleichungen und Bildverarbeitung eine wichtige Rolle.

Eine Wavelet-Basis auf $\R^n$ ist im Kern eine Orthonormalbasis von $L_2(\R^n)$, die aus leicht zu beschreibenden Funktionen $\Phi=\{\Phi_r^j\}$ besteht. Dies bedeutet, dass wir eine Zerlegung
\begin{align*}
 f=\sum_{j,r} \lambda_r^j(f) \cdot \Phi_r^j
\end{align*}
mit Koeffizienten aus dem Folgenraum $l_2(\Z^n)$ finden können. Die Koeffizienten können mit Hilfe des Skalarprodukts in $L_2(\R^n)$ ausgerechnet werden mittels
\begin{align}
\label{decointrD}
 \lambda_r^j(f)=\left( f,\Phi_r^j\right).
\end{align}
Gute Einführungen in die Theorie der Wavelets kann man in \cite{Woj97} und \cite{Mal99} finden.

Allerdings haben Wavelets auf $\R^n$ noch weitere bemerkenswerte Eigenschaften in Verbindung mit Funktionenräumen vom Besov- und Triebel-Lizorkin-Typ. In \cite[Theorem 3.5]{Tri06} zeigte Triebel, dass die Daubechies Wavelet-Basis von $L_2(\R^n)$ auch eine Basis für die Besov- und Triebel-Lizorkin-Räume ist -- mit geeigneten Folgenräumen für die Koeffizienten $\lambda_r^j(f)$. Triebel wies auch auf einige Vorläufer in  \cite[Bemerkung 1.66]{Tri06} hin. Einführungen in Daubechies Wavelets findet man in \cite{Dau92} und \cite{Woj97}. Triebels Erkenntnisse über Wavelets basieren auf der Theorie der lokalen Mittel und der atomaren Zerlegungen von Funktionenräumen aus den späten 1980igern und 1990igern. In diesem Zusammenhang sollte man einen Blick in \cite[Kapitel 2]{Tri92}, \cite{Ryc99} bzw. \cite[Abschnitt 13]{Tri97}, \cite{FrJ85} und \cite{FrJ90} werfen.

Ein eher schwieriges Problem ist die Konstruktion von Wavelet-Basen auf Ge\-bie\-ten $\Om \subset \R^n$ für passende Funktionenräume. Ein Ausgangspunkt sind die Arbeiten von Ciesielski und Figiel \cite{CiF83A}, \cite{CiF83B} und \cite{Cie84}, die sich mit Spline-Basen für Räume differenzierbarer Funktionen wie auch mit klassischen Sobolev-Räumen und Besov-Räumen auf kompakten $C^{\infty}$-Mannigfaltigkeiten beschäftigen. Die grund\-sätz\-liche Idee ist die Zerlegung beliebiger Gebiete in einfache Standard-Gebiete -- das so genannte Problem der Gebietszerlegung. Allerdings ist diese Aufgabe ziemlich kompliziert. Verwandte Zugänge und Verallgemeinerungen in dieser Richtung kann man in \cite{Dah97}, \cite{DaS99}, \cite{Dah01}, \cite{Coh03}, \cite{HaS04}, \cite{JoK07} und \cite{FoG08} finden. 

Eine spezielle Zerlegung wird genau durch die Einführung von zellulären Gebieten in \cite[Definition 5.40]{Tri08} angestoßen. Ähnliche Ideen findet man in \cite[Abschnitt 9.1]{Dah97}, \cite{Dah01}, \cite{Mal99} und \cite{CDV00}. Ein zelluläres Gebiet ist eine disjunkte Vereinigung von diffeomorphen Bildern des Würfels. Das bekannteste Beispiel ist der Einheitswürfel $Q$ in $\R^n$, der eine wichtige Rolle in dieser Arbeit spielen wird. Weiterhin sind alle $C^{\infty}$-Gebiete zellulär. 

Wenn man sich Wavelets auf Gebieten anschaut, dann steht man vor einem neuen Problem -- Triebel musste die sogenannten kritischen Werte in seinen Ausführungen ausschließen. Auf der einen Seite konstruierte er (Riesz)-Wavelet-Frames für Triebel-Lizorkin-Räume $\FO[\Om]$ auf $C^{\infty}$-Gebieten $\Om$ mit den natürlichen Ausnahmewerten $s-\frac{1}{p} \in \N_0$ in \cite[Theorem 5.27]{Tri08}. Im Gegensatz zu einer Wavelet-Basis ist eine Zerlegung wie in \eqref{decointrD} für Wavelet-Frames nicht eindeutig. Triebel konnte nicht zeigen, dass es eine (Riesz)-Wavelet-Basis für allgemeine Dimensionen oder allgemeinen Glattheitsparameter $s$ für $C^{\infty}$-Gebiete gibt.

Auf der anderen Seite konstruierte er (Riesz)-Wavelet-Basen für $\FO[\Om]$, wobei $\Om$ ein $n$-dimensionales zelluläres Gebiet ist. Allerdings musste er dabei die Ausnahmewerte $s-\frac{k}{p} \notin \N_0$ für $k\in \{1,\ldots,n\}$ ausschließen, siehe \cite[Theorem 6.30]{Tri08}. Zum Beispiel ist der berühmteste Sobolevraum $W_2^1(Q)$ ein Ausnahmeraum und bis heute scheint es keine Konstruktion für eine Wavelet-Basis im Sinne Triebels auf $W_2^1(Q)$ zu geben. Einen Überblick über die Situation für klassische Räume wird in \cite[Abschnitt 5.3.1, Bemerkung 5.50]{Tri08} gegeben, mit Verweisen auf \cite{Gri85} und \cite{Gri92}.

Das Hauptziel dieser Arbeit ist die Eingliederung der Ausnahmewerte. Einen Vorschlag, wie man dies machen könnten, ist in \cite[Abschnitt 6.2.4]{Tri08} gegeben. Als Erstes betrachten wir genau wie in \cite{Tri08} die Situation für den Einheitswürfel $Q$ als das Standardbeispiel eines zellulären Gebiets. Die Idee ist, die Räume $\FO[Q]$ zu modifizieren und zusätzliche Bedingungen zu fordern -- die Funktionenräume einzuschränken. Man definiert die eingeschränkten Funktionenräume $\Frinf[Q]$ wie folgt: Man wähle ein $f \in \FO[Q]$ und für jeden kritischen Wert $l \in \{0,\ldots,n-1\}$, das heißt für jeden Wert, für den 
\begin{align*}
 s-\frac{n-l}{p} \in \N_0
\end{align*}
gilt, verlangt man von $f$, das es die Eigenschaft $R_l^{r,p}$ erfüllt. Grob gesagt sichert diese Eigenschaft einen gewissen Abfall von den Ableitungen von $f$ an den Rändern (Kanten, Ecken) der Dimension $l<n$ des Einheitswürfels. Die Konstruktion des eingeschränkten Triebel-Lizorkin Funktionenraums $\Frinf[Q]$ stellt sicher, dass im nicht-kritischen Fall, der in \cite[Theorem 6.30]{Tri08} untersucht wurde, die beiden Räume $\Frinf[Q]$ und $\FO[Q]$ die gleichen sind. Das Hauptziel dieser Dissertation ist die Konstruktion von (Riesz)-Wavelet-Basen für $\Frinf[Q]$ ohne irgendwelche Ausnahmewerte. Wir werden in Kapitel \ref{Wavelets} das folgende Theorem zeigen:
\begin{thma}
 Sei $Q$ der Einheitswürfel in $\R^n$ für $n\geq 2$. Sei
 \begin{align*}
  s>0, \quad 1\leq p <\infty \quad \text{und} \quad 1\leq q<\infty.
 \end{align*}
 Dann existiert eine oszillierende $u$-Riesz Basis für $\Frinf[Q]$  für jedes $u \in \N_0$ mit $u>s$. Der dazugehörige Folgenraum ist $\fO[\overline{Q}]$. 
\end{thma}
Dieses Theorem ist eine Verallgemeinerung von Theorem 6.30 in \cite{Tri08} für die $F$-Räume, da für die nicht-kritischen Werte $\Frinf[Q]=\FO[Q]$ nach Konstruktion gilt. Der entscheidenden Ausgangspunkt für die Konstruktion der (Riesz)-Wavelet-Basis ist das folgende
\begin{thma}
 Sei $1\leq p <\infty$, $1\leq q<\infty$ und $0<s<u \in \N$. Sei $n \in \N$, $l_0 \in \N_0$ mit $0\leq l_0\leq n$ definiert wie in \eqref{l0def}, $r^l$ für $l_0\leq l\leq n-1$ definiert wie in \eqref{rldef} und $\overline{r}=\{r^{l_0},\ldots,r^{n-1}\}$. Dann gilt
\begin{align*}
\Frinf[Q]=\Frloc[Q] \times \Ext_{\Gamma}^{\overline{r},u} \prod_{l=l_0}^{n-1}  \prod_{\underset{|\alpha|\leq r^{l}}{\alpha \in \N_{l}^n}} F_{p,p}^{s-\frac{n-l}{p}-|\alpha|,\rloc}(\Gamma_{l}).
\end{align*}
(zueinander komplementierte Unterräume)
\end{thma}
Dieses Theorem ist eine Verallgemeinerung von \cite[Theorem 6.28]{Tri08} für die $F$-Räume mit einer abgeänderten Notation. Da alle Räume auf der rechten Seite (Riesz)-Wavelet-Basen haben, können wir eine (Riesz)-Wavelet-Basis für $\Frinf[Q]$ konstruieren.

\medskip

Die Resultate aus Kapitel \ref{chapterpointwise} sind eine Art Nebenprodukt unserer Erkenntnisse für Wavelet-Basen auf zellulären Gebieten. Diese Resulte werden in \cite{Sch13} veröffentlicht.\ Wir verallgemeinern das atomare Zerlegungstheorem aus \cite{Tri92, Tri97} für Besov- und Triebel-Lizorkin-Räume $\bspq$ und $\fspq$ und präsentieren zwei Anwendungen auf punktweise Multiplikatoren und Diffeomorphismen als stetige, lineare Operatoren in $\bspq$ resp.\ $\fspq$.

Nach \cite{Tri92} gilt, dass
 \begin{align*}
 P_{\varphi}: f \mapsto \varphi \cdot f
\end{align*}
den Raum $\bspq$ in den Raum $\bspq$ abbildet, falls $s> \sigma_p$ und $\varphi \in C^k(\R^n)$ mit $k>s$. Weiterhin bildet die Superposition mit einer Vektorfunktion $\varphi: \R^n \rightarrow \R^n$
 \begin{align*}
 D_{\varphi}: f \mapsto f \circ \varphi
\end{align*}
den Raum $\bspq$ nach $\bspq$ ab, falls $\varphi$ ein $k$-Diffeomorphismus ist und $k$ groß genug ist in Abhängigkeit von $s$ und $p$. Auch für $\fspq$ gibt es ähnliche Resultate. 

Die Hauptidee, um diese beiden Aussagen auf einfache Art zu beweisen, ist das atomare Zerlegungstheorem. Hauptsächlich muss man zeigen, dass die Multiplikation eines Atoms $a_{\nu,m}$ mit einer Funktion $\varphi$ bzw.\ die Superposition mit $\varphi$ immer noch ein Atom ist mit ähnlichen Eigenschaften. Aber es gab ein Problem: Falls $s \leq \sigma_p$ bzw.\ $s \leq \sigma_{p,q}$, müssen Atome Momentenbedingungen erfüllen, d.\,h.\
\begin{align}
\label{momentoD}
\int_{\mathbb{R}^n} &x^{\beta} a(x) \ dx=0 \text{ falls }  |\beta|\leq L-1
\end{align}
für $L \in \N_0$ und $L>\sigma_p-s$ bzw.\ $L>\sigma_{p,q}-s$. Aber diese Eigenschaften bleiben nach der Multiplikation bzw.\ Superposition nicht erhalten. In \cite{Skr98} wurde diese Momentenbedingungen durch die allgemeineren Bedingungen
\begin{align*} 
\left| \int_{d \cdot Q_{\nu,m}} \psi(x) a(x) \ dx \right|\leq C \cdot  2^{-\nu\left(s+L+n\left(1-\frac{1}{p}\right)\right)} \|\psi|C^L(\R^n)\|
\end{align*}
für alle $\psi \in C^L(\R^n)$ ersetzt. Nun ist die Situation anders: Diese Bedingungen bleiben  nach Multiplikation bzw.\ Superposition erhalten. Diese Ersetzung durch allgemeinere Bedingungen ist typisch, wenn man über Wavelet-Darstellungen auf Gebieten $\Om$ nachdenkt, siehe zum Beispiel die Bemerkungen über die Momentenbedingungen in \cite[Abschnitt 3.1]{Dah01}.

In diesem Kapitel gehen wir einen Schritt weiter. Wir zeigen, dass man die üblichen $C^K(\R^n)$-Bedingungen an Atome durch H\"older-Bedingungen ersetzen kann, wie in Definition \ref{Atoms}. Dies verallgemeinert die Definition von Atomen durch Triebel, Skrzypczak und Winkelvoss aus \cite{Tri97, Skr98, TrW96}. 

Es gibt bereits weitreichende Ergebnisse, wie man die Bedingungen
\begin{align*}
 \|a(2^{-\nu}\cdot)|\hold[K]\|
\end{align*}
 weiter durch $\|a(\cdot)|B_{p,p}^{K}(\R^n)\|$ mit $K>s$ ersetzen kann, meist im Zusammenhang mit Spline-Darstellungen, siehe z.\,B.\ \cite[Teil II, Abschnitt 6.5]{KaL95}, \cite[Abschnitt 2.2]{Tri06} und das noch nicht erschienene Paper von Schneider und Vybiral \cite{ScV12}. Von diesen Referenzen werden nur im ersten Buch auch Momentenbedingungen \eqref{momentoD} betrachtet.

\medskip
Als Folgerungen aus dem atomaren Zerlegungstheorem mit diesen verallgemeinerten Atomen sind wir in der Lage, die Schlüsseltheoreme für punktweise Multiplikatoren und Diffeomorphismen aus \cite{Tri92} zu erweitern. Es ist nicht das Ziel dieser Ausführungen, die bestmöglichen Bedingungen oder sogar exakte Charakterisierungen für punktweise Multiplikatoren in Funktionenräumen $\bspq$ und $\fspq$ zu geben. Dafür verweisen wir auf  Strichartz \cite{Str67}, Peetre \cite{Pee76} sowie auf Maz'ya und Shaposhnikova \cite{MaS85, MaS09} für die klassischen Sobolevräume, während wir für $B_{p,p}^s(\R^n)$ auf Franke \cite{Fra86}, Frazier und Jawerth \cite{FrJ90}, Netrusov \cite{Net92}, Koch, Runst und Sickel \cite{RS96, Sic99a, Sic99b, KoS02} sowie Triebel \cite[Abschnitt 2.3.3]{Tri06} für allgemeine Funktionenräume $\bspq$ und $\fspq$ verweisen.

Für punktweise Multiplikatoren für $\bspq$ erhalten wir:

\begin{thma}
Sei $0<p\leq \infty$ und $\rho > \max(s,\sigma_p-s)$. Dann existiert eine positive Zahl $c$, sodass
\begin{align*}
 \|\varphi f|\bspq\| \leq c \|\varphi|\hold[\rho]\| \cdot \|f|\bspq\|
\end{align*}
für alle $\varphi \in \hold[\rho]$ und alle $f \in \bspq$.
 \end{thma}

\noindent 

Für weitere hinreichende Resultate für Diffeomorphismen, die auch Charakterisierungen enthalten, verweisen wir für die klassischen Sobolevräume $W_{p}^k(\R^n)$ auf Gol'dshtein, Reshetnyak, Romanov, Ukhlov und Vodop'yanov \cite{GoR76, GoV75, GoV76, Vod89, UkV02},  \cite[Kapitel 4]{GoR90}, Markina \cite{Mar90} sowie auf Maz'ya und Shaposhnikova \cite{Maz69, MaS85}, während wir für Besov-Räume $\bspq$ mit $0<s<1$ auf Vodop'yanov, Bordaud und Sickel \cite{Vod89, BoS99} verweisen. Ein Spezialfall unseres Resultats (Lipschitz-Diffeomorphismen) kann man bei Triebel  \cite[Abschnitt 4]{Tri02} finden.

Wir werden im Falle $\bspq$ beweisen:

\begin{thma}
 Sei $0<p\leq \infty$, $\rho\geq 1$ und $\rho > \max(s,1+\sigma_p-s)$.\ Falls $\varphi$ ein $\rho$-diffeomorphismus ist, dann existiert eine Konstante $c$, sodass
\begin{align*}
 \|f(\varphi(\cdot))|\bspq\| \leq c \cdot \|f|\bspq\|.
\end{align*}
für alle $f \in \bspq$. Also bildet $D_{\varphi}$ den Raum $\bspq$ auf den Raum $\bspq$ ab.
\end{thma}
Diese beiden Anwendungen oder eher deren kurze Beweise sind die Hauptresultate aus Kapitel \ref{chapterpointwise}.

\bigskip

Diese Dissertation ist wie folgt angeordet: Im ersten Kapitel geben wir die nötigen Grundlagen an -- Definitionen, Standardeigenschaften und bekannte Resultate.

Im Kapitel \ref{chapterpointwise} präsentieren wir die erwähnten Resultate über atomare Zerlegungen und ihre Anwendungen auf punktweise Multiplikatoren und Diffeomorphismen für Funktionenräume von Besov- und Triebel-Lizorkin-Typ.

Im Kapitel \ref{Zerlegungsth2} werden wir uns die geforderten Extra-Bedingungen und Zerlegungen für eine feste Dimension $l$ mit $0\leq l \leq n-1$ anschauen, wobei $n$ die Dimension von $\R^n$ ist. Wir schauen uns den Funktionenraum $\FR$ und sein Verhalten in der Nähe von $l$-dimensionalen Ebenen an. Wir untersuchen die Probleme und das unterschiedliche Verhalten von $\FR$ im Ausnahmefall
\begin{align*}
 s-\frac{n-l}{p} \in \N_0
\end{align*}
und führen eingeschränkte Funktionenräume $\Frinf[\hyp]$ ein, die in den kritischen Fällen echte Teilmengen von $\FR$ sind. Weiterhin beweisen wir Hardy Ungleichungen an den $l$-dimensionalen Ebenen und zeigen Zerlegungstheoreme unter Benutzung von Spur- und Fortsetzungsoperatoren.

Im Kapitel \ref{Wavelets} nutzen wir die Resultate aus Kapitel \ref{Zerlegungsth2} und verallgemeinern die Extra-Bedingungen, nun für alle Dimensionen $l=\{0,\ldots,n-1\}$ gleichzeitig, auf den Würfel $Q$. Wir führen den eingeschränkten Funktionenraum $\Frinf[Q]$ ein und untersuchen sein Verhalten an den Rändern des Würfels. Letztendlich beweisen wir die Existenz einer (Riesz)-Wavelet-Basis für $\Frinf[Q]$ für alle Werte $s>0$. Wir konstruieren eine Wavelet-Basis, die aus inneren und Rand-Wavelets auf dem Würfel $Q$ bestehen. Am Ende betrachten wir den bekanntesten Spezialfall eines kritischen Raumes, den Sobolevraum $W_{2}^1(Q)$ und seinen eingeschränkten Raum $W_{2}^{1,\rinf}(Q).$

Im Kapitel \ref{Open} diskutieren wir die Ergebnisse und erwähnen die offenen Probleme. Wir fragen, ob die Extra-Bedingungen wirklich notwendig sind. Weiterhin untersuchen wir, ob und wie man die (Riesz)-Wavelet-Basis vom Würfel auf zelluläre Gebiete übertragen kann und beschreiben die Probleme, die dabei auftreten. Am Ende schauen wir uns die spezielle Situation für die Zerlegung des eingeschränkten Funktionenraums auf der Einheitskugel an.

%% file: Introduction.tex
Nowadays the theory and application of wavelet decompositions plays an important role not only for the study of function spaces (of Lebesgue, Hardy, Sobolev, Besov, Triebel-Lizorkin type) but also for its applications in signal and numerical analysis, partial differential equations and image processing.

A wavelet basis on $\R^n$ is essentially an orthonormal basis of $L_2(\R^n)$ consisting of easily described functions $\Phi=\{\Phi_r^j\}$. This means, that we can decompose
\begin{align*}
 f=\sum_{j,r} \lambda_r^j(f) \cdot \Phi_r^j
\end{align*}
with coefficients $\lambda_r^j(f)$ from the sequence space $l_2(\Z^n)$. The coefficients can be calculated using the scalar product in $L_2(\R^n)$, i.\,e.\
\begin{align}
\label{decointr}
 \lambda_r^j(f)=\left( f,\Phi_r^j\right).
\end{align}
Introductions to wavelets can be found e.\,g.\ in \cite{Woj97} and \cite{Mal99}.

But wavelets on $\R^n$ have more remarkable properties in connection with function spaces of Besov and Triebel-Lizorkin type. In \cite[Theorem 3.5]{Tri06} Triebel showed that the Daubechies wavelet basis of $L_2(\R^n)$ is also a basis for the Besov and Triebel-Lizorkin spaces using suitable sequence space conditions on the coefficients $\lambda_r^j(f)$. Triebel also mentioned some forerunners in \cite[Remark 1.66]{Tri06}. Introductions to Daubechies wavelets can be found in \cite{Dau92} and \cite{Woj97}. The observations of Triebel on wavelets are based on the theory of local means and atomic decompositions of function spaces from the late 1980s and 1990s. In this connection one should have a look at \cite[Chapter 2]{Tri92}, \cite{Ryc99} resp.\ \cite[Section 13]{Tri97}, \cite{FrJ85} and \cite{FrJ90}.    

A rather tricky question is the construction of wavelet bases on domains $\Om \subset \R^n$ for suitable function spaces. One starting point are the papers of Ciesielski and Figiel \cite{CiF83A}, \cite{CiF83B} and \cite{Cie84} dealing with spline bases for spaces of differentiable functions as well as classical Sobolev and Besov spaces on compact $C^{\infty}$ manifolds. The principle idea is to decompose arbitrary domains $\Om$ into simpler standard domains -- the so-called domain (decomposition) problem. But the topic is rather involved and complex. Related approaches and extensions in this direction were given in \cite{Dah97}, \cite{DaS99}, \cite{Dah01}, \cite{Coh03}, \cite{HaS04}, \cite{JoK07} and \cite{FoG08}. 

One special decomposition is reflected by the introduction of cellular domains in \cite[Definition 5.40]{Tri08}. Similar ideas can be found in \cite[Section 9.1]{Dah97}, \cite{Dah01}, \cite{Mal99} and \cite{CDV00}. A cellular domain is a disjoint union of diffeomorphic images of a cube. The most prominent example is the unit cube $Q$ in $\R^n$ which will play a significant role in this thesis. Furthermore, all $C^{\infty}$-domains are cellular domains.

When considering wavelets on domains one faces new problems  -- Triebel had to exclude the so-called critical values from his observations: On the one hand he construced wavelet (Riesz) frames for Triebel-Lizorkin spaces $\FO[\Om]$ for $C^{\infty}$-domains $\Om$ with natural exceptional values $s-\frac{1}{p} \in \N_0$ in \cite[Theorem 5.27]{Tri08}. In contrast to a wavelet basis a decomposition as in \eqref{decointr} is not unique for a wavelet frames. Triebel was not able to show that there is a wavelet (Riesz) basis for general dimensions and general smoothness parameter $s$ for bounded $C^{\infty}$-domains. 

On the other hand he constructed wavelet (Riesz) basis for $\FO[\Om]$ where $\Om$ is an $n$-dimensional cellular domain. But he had to exclude the exceptional values $s-\frac{k}{p} \notin \N_0$ for $k\in \{1,\ldots,n\}$, see \cite[Theorem 6.30]{Tri08}. For instance, the most prominent Sobolev space $W_2^1(Q)$ is exceptional and upto now there seems to be no construction of a wavelet basis in Triebel's sense for $W_2^1(Q)$. An overview of the situation for this classical space is given in \cite[Section 5.3.1, Remark 5.50]{Tri08} also refering to \cite{Gri85} and \cite{Gri92}.

The main aim of this thesis is the incorporation of the exceptional values. A proposal how to deal with these cases is given in \cite[Section 6.2.4]{Tri08}. At first, as in \cite{Tri08} one considers the situation for the unit cube $Q$ as the standard example of a cellular domain. The idea is to modify the spaces $\FO[Q]$ and to ``reinforce them''. One defines the reinforced function spaces $\Frinf[Q]$ as follows: One takes an $f \in \FO[Q]$ and for every critical value $l \in \{0,\ldots,n-1\}$, i.\,e.\ when
\begin{align*}
 s-\frac{n-l}{p} \in \N_0,
\end{align*}
one requires $f$ to fulfil the additional reinforce property $R_l^{r,p}$. Roughly speaking, this reinforce property asks for a certain decay of the derivatives of $f$ at the faces (edges, vertices) of dimension $l<n$ of the unit cube. 
The construction of the reinforced Triebel-Lizorkin function spaces $\Frinf[Q]$ ensures that in the non-critical cases considered in \cite[Theorem 6.30]{Tri08} the spaces $\Frinf[Q]$ and $\FO[Q]$ are the same. The main aim of this thesis is the construction of wavelet (Riesz) basis for the spaces $\Frinf[Q]$ without any exceptional values. We will show in Chapter \ref{Wavelets}
\begin{thma}
 Let $Q$ be the unit cube in $\R^n$ for $n\geq 2$. Let 
 \begin{align*}
  s>0, \quad 1\leq p <\infty \quad \text{and} \quad 1\leq q<\infty.
 \end{align*}
 Then $\Frinf[Q]$ has an oscillating $u$-Riesz basis for any $u \in \N_0$ with $u>s$. The related sequence space is $\fO[\overline{Q}]$. 
\end{thma}
This theorem is a generalization of Theorem 6.30 in \cite{Tri08} for the $F$-spaces since $\Frinf[Q]=\FO[Q]$ for the non-exceptional values by construction. The crucial starting point for the construction of the wavelet (Riesz) basis is the following 
\begin{thma}
 Let $1\leq p <\infty$, $1\leq q<\infty$ and $0<s<u \in \N$. Let $n \in \N$, $l_0 \in \N_0$ with $0\leq l_0\leq n$ defined as in \eqref{l0def}, $r^l$ for $l_0\leq l\leq n-1$ defined as in \eqref{rldef} and $\overline{r}=\{r^{l_0},\ldots,r^{n-1}\}$. Then it holds
\begin{align*}
\Frinf[Q]=\Frloc[Q] \times \Ext_{\Gamma}^{\overline{r},u} \prod_{l=l_0}^{n-1}  \prod_{\underset{|\alpha|\leq r^{l}}{\alpha \in \N_{l}^n}} F_{p,p}^{s-\frac{n-l}{p}-|\alpha|,\rloc}(\Gamma_{l}).
\end{align*}
(complemented subspaces)
\end{thma}
This theorem is the generalization of \cite[Theorem 6.28]{Tri08} for the $F$-spaces -- using a different notation. Since all the spaces on the right hand side have wavelet (Riesz) basis, we can construct a wavelet (Riesz) basis for $\Frinf[Q]$ by this decomposition.

\medskip

The results from Chapter \ref{chapterpointwise} are some kind of byproduct of our observations on wavelet bases on cellular domains. These results will be published in \cite{Sch13}. We generalize the atomic decomposition theorem from  \cite{Tri92, Tri97} for Besov and Triebel-Lizorkin spaces $\bspq$ and $\fspq$ and present two applications to pointwise multipliers and diffeomorphisms as continuous linear operators in $\bspq$ resp.\ $\fspq$. 
                                                                 
According to \cite{Tri92}
 \begin{align*}
 P_{\varphi}: f \mapsto \varphi \cdot f
\end{align*}
maps $\bspq$ into $\bspq$ if $s> \sigma_p$ and $\varphi \in C^k(\R^n)$ with $k>s$. Furthermore, the superposition with a vector function $\varphi: \R^n \rightarrow \R^n$
 \begin{align*}
 D_{\varphi}: f \mapsto f \circ \varphi
\end{align*}
maps $\bspq$ to $\bspq$ if $\varphi$ is a $k$-diffeomorphism and $k$ is large enough in dependence of $s$ and $p$. There are similar results for $\fspq$. 

The main idea for an easy proof is the atomic decomposition theorem. Mainly one has to show that a multiplication of an atom $a_{\nu,m}$ with a function $\varphi$ resp.\ the superposition with $\varphi$ is still an atom with similar properties. But there was one problem: If $s \leq \sigma_p$ resp.\ $s \leq \sigma_{p,q}$, then atoms need to fulfil moment conditions, i.\,e.\
\begin{align}
\label{momento}
\int_{\mathbb{R}^n} &x^{\beta} a(x) \ dx=0 \text{ if }  |\beta|\leq L-1
\end{align}
for $L \in \N_0$ and $L>\sigma_p-s$ resp.\ $L>\sigma_{p,q}-s$. But these properties are not preserved by multiplication resp.\ superposition. In \cite{Skr98} these moment conditions were replaced by the more general assumptions
\begin{align*} 
\left| \int_{d \cdot Q_{\nu,m}} \psi(x) a(x) \ dx \right|\leq C \cdot  2^{-\nu\left(s+L+n\left(1-\frac{1}{p}\right)\right)} \|\psi|C^L(\R^n)\|
\end{align*}
for all $\psi \in C^L(\R^n)$. Now the situation changes: These conditions remain true after multiplication resp.\ superposition.

This replacement is typical when thinking of atomic, in particular wavelet representations for functions on domains $\Om$, see the remarks on the cancellation property in \cite[Section 3.1]{Dah01}. 

In this chapter we go a step further. We show that one can replace the usual $C^K(\R^n)$-conditions on atoms by H\"older conditions ($\hold[K]$-spaces) as defined in Definition \ref{Atoms}. This generalizes the known definitions of atoms from Triebel, Skrzypczak and Winkelvoss \cite{Tri97, Skr98, TrW96}. 

 There is an existing theory replacing the conditions 
 \begin{align*}
 \|a(2^{-\nu}\cdot)|\hold[K]\|
 \end{align*}
by $\|a(\cdot)|B_{p,p}^{K}(\R^n)\|$ with $K>s$, mainly in connection with spline representations. For instance, see \cite[part II, Section 6.5]{KaL95}, \cite[Section 2.2]{Tri06} and the recent paper by Schneider and Vybiral \cite{ScV12}. Of these, only the first book incorporates the usual moment conditions as in \eqref{momento}.

\medskip
As corollaries of the atomic representation theorem with these more general atoms we are able to extend the key theorems on pointwise multipliers and diffeomorphisms from \cite{Tri92}.
It is not the aim of our observations to give best conditions or even exact characterizations for pointwise multipliers in function spaces $\bspq$ and $\fspq$. For this we refer to Strichartz \cite{Str67}, Peetre \cite{Pee76} as well as to Maz'ya and Shaposhnikova \cite{MaS85, MaS09} for the classical Sobolev spaces, while for $B_{p,p}^s(\R^n)$ we refer to Franke \cite{Fra86}, Frazier and Jawerth \cite{FrJ90}, Netrusov \cite{Net92}, Koch, Runst and Sickel \cite{RS96, Sic99a, Sic99b, KoS02} as well as to Triebel \cite[Section 2.3.3]{Tri06} for general function spaces $\bspq$ and $\fspq$.

We obtain for pointwise multipliers with respect to $\bspq$:

\begin{thma}
Let $0<p\leq \infty$ and $\rho > \max(s,\sigma_p-s)$. Then there exists a positive number $c$ such that
\begin{align*}
 \|\varphi f|\bspq\| \leq c \|\varphi|\hold[\rho]\| \cdot \|f|\bspq\|
\end{align*}
for all $\varphi \in \hold[\rho]$ and all $f \in \bspq$.
 \end{thma}

\noindent

For further sufficient results on diffeomorphisms including characterizations for classical Sobolev spaces $W_{p}^k(\R^n)$ we refer to Gol'dshtein, Reshetnyak, Romanov, Ukhlov and Vodop'yanov \cite{GoR76, GoV75, GoV76, Vod89, UkV02},  \cite[Chapter 4]{GoR90}, Markina \cite{Mar90} as well as to Maz'ya and Shaposhnikova \cite{Maz69, MaS85}, while for Besov spaces $\bspq$ with $0<s<1$ we refer to Vodop'yanov, Bordaud and Sickel \cite{Vod89, BoS99}. A special case of our result (Lipschitz diffeomorphisms) can be found in Triebel \cite[Section 4]{Tri02}.

We will prove (in case of $\bspq$):

\begin{thma}
 Let $0<p\leq \infty$, $\rho\geq 1$ and $\rho > \max(s,1+\sigma_p-s)$. If $\varphi$ is a $\rho$-diffeomorphism, then there exists a constant $c$ such that
\begin{align*}
 \|f(\varphi(\cdot))|\bspq\| \leq c \cdot \|f|\bspq\|.
\end{align*}
for all $f \in \bspq$. Hence $D_{\varphi}$ maps $\bspq$ onto $\bspq$.
\end{thma}
These two applications or rather their short proofs are the main results from Chapter \ref{chapterpointwise}.
\bigskip

This thesis is organized as follows: In the first chapter we will give the necessary preliminaries -- definitions, standard properties and known results.

In Chapter \ref{chapterpointwise} we will present the mentioned results on atomic representations and its applications to pointwise multipliers and diffeomorphisms for function spaces of Besov and Triebel-Lizorkin type. 

In Chapter \ref{Zerlegungsth2} we will consider the situation of reinforce properties and decompositions for one fixed dimension $l$ with $0\leq l \leq n-1$ where $n$ is the dimension of $\R^n$. We take a look at the function space $\FR$ and its behaviour at an $l$-dimensional plane. We examine the problems and differences for the function spaces $\FR$ in the exceptional cases 
\begin{align*}
 s-\frac{n-l}{p} \in \N_0
\end{align*}
and introduce reinforced function spaces $\Frinf[\hyp]$ which are proper subsets of $\FR$ for these exceptional values. Furthermore, we prove Hardy inequalities at the $l$-dimensional plane and show decomposition theorems using trace and extension operators.

In Chapter \ref{Wavelets} we use the observations made in Chapter \ref{Zerlegungsth2} and generalize the reinforce properties, now for arbitrary dimension $l=\{0,\ldots,n-1\}$, to the unit cube $Q$. We introduce the reinforced function spaces $\Frinf[Q]$ and observe their behaviour at the boundaries of the cube. We arrive at the main theorem - the existence of a wavelet (Riesz) basis for $\Frinf[Q]$ for all values $s>0$. We construct a wavelet basis consisting of interior and boundary wavelets on the cube $Q$. At the end we present the results for the most prominent exceptional Sobolev space $W_{2}^1(Q)$ and its reinforced space $W_{2}^{1,\rinf}(Q).$

In Chapter \ref{Open} we will discuss the results and mainly note what remains open. We discuss the necessity of reinforced properties. Furthermore, we examine the extension of the wavelet (Riesz) basis decomposition theorem from the cube to cellular domains and describe the problems which occur in this context. At the end we take a look at the situation of the decomposition for reinforced function spaces on the unit ball.

%% file: Grundlagen.tex
Let $\R^n$ be the Euclidean $n$-space, $\Z$ be the set of integers, $\N$ be the set of natural numbers, $\N_0=\N \cup \{0\}$ and $\overline{\N}_0=\N_0 \cup \{\infty\}$. By $|x|$ we denote the usual Euclidean norm of $x \in \R^n$, by $\|x|X\|$ the (quasi)-norm of an element $x$ of a (quasi)-Banach space $X$. If $S \subset \R^n$, then we denote the $n$-dimensional Lebesgue measure of $S$ by $|S|$. 

By $\Sc(\R^n)$ we mean the Schwartz space on $\R^n$, by $\Sc'(\R^n)$ its dual. The Fourier transform of $f \in \Sc'(\R^n)$ resp.\ its inverse will be denoted by $\hat{f}$ resp.\ $\check{f}$. The convolution of $f \in \Sc'(\R^n)$ and $\varphi \in \Sc(\R^n)$ will be denoted by $f * \varphi$. With $supp \ f$ we denote the support of a distribution $f \in \Sc'(\R^n)$.

By $L_p(\R^n)$ for $0<p\leq \infty$ we denote the usual quasi-Banach space of $p$-integrable complex-valued functions with respect to the Lebesgue measure $|\cdot|$ with quasi-norm
\begin{align*}
 \|f|L_p(\R^n)\|:=\left(\int_{\R^n} |f(x)|^p \ dx\right)^{\frac{1}{p}}
\end{align*}
with the usual $\sup$-norm modification for $p=\infty$.

Let $X,Y$ be quasi-Banach spaces. By the notation $X \hookrightarrow Y$ we mean that $X \subset Y$ and that the inclusion map is bounded. 

Let $B_r(x_0)=\{x \in \R^n: |x-x_0|<r\}$ be the open ball with centre $x_0$ and radius $r>0$. Furthermore, we shorten $B_r:=B_r(0)$ and $B:=B_1$. 

Throughout the text all unimportant constants will be called $c,c',C$ etc.\ or we will directly write $A \lesssim B$ which means that there is a constant $C>0$ such that $A \leq C \cdot B$. Only if extra clarity is desirable, the dependency of the parameters will be stated explicitly. The concrete value of these constants may vary in different formulas but remains the same within one chain of inequalities. By $A \sim B$ we mean that there are constants $C_1,C_2>0$ such that $ C_1 \cdot B \leq A \leq C_2 \cdot B$. 

\section{H\"older spaces of differentiable functions on $\R^n$}

Let $k \in \N_0$. Then by $C^k(\R^n)$ we denote the space of all functions $f: \R^n \rightarrow \C$ which are $k$-times continuously differentiable (continuous, if $k=0$) such that the norm
\begin{align*}
 \|f|C^k(\R^n)\|:=\sum_{|\alpha|\leq k} \sup |D^{\alpha} f(x)|
\end{align*}
is finite, where the $\sup$ is taken over $x \in \R^n$.

Furthermore, the set $C^{\infty}(\R^n)$ is defined by
\begin{align*}
 C^{\infty}(\R^n):=\bigcap_{k \in \N_0} C^k(\R^n),
\end{align*}
while $D(\R^n)$ is defined as the set of all $f \in C^{\infty}(\R^n)$ with compact support. With $D'(\R^n)$ we denote its topological dual.

\begin{Definition}
\label{Hoelder}
 Let $0<\sigma\leq 1$ and $f: \R^n \rightarrow \C$ be continuous. We define
\begin{align*}
 \|f|\lip[\sigma]\|:=\sup_{x,y\in \R^n, x\neq y} \frac{|f(x)-f(y)|}{|x-y|^{\sigma}}.
\end{align*}

If $s\in \mathbb{R}$, then there are uniquely determined $\sint \in \mathbb{Z}$ and $\srest \in (0,1]$ with $s=\sint+\srest$. 

Let $s>0$. Then the H\"older space with index $s$ is given by
\begin{align*}
 \hold&=\left\{f \in C^{\sint}(\R^n):  \|f|{\hold}\|<\infty\right\} \text{ with } \\
 \|f&|{\hold}\|:=\|f|C^{\sint}(\R^n)\|+\sum_{|\alpha|=\sint} \|D^{\alpha} f|\lip[\srest]\|.
\end{align*}
In the later observations we will also discuss the situation for $s=0$. Usually we will then consider $L_{\infty}(\R^n)$ instead of $\hold$, which is sufficient for the later statements, see e.\,g.\ Theorem \ref{AtomicRepr}.

\end{Definition}

\section{Besov and Triebel-Lizorkin function spaces on $\R^n$}
 
Let $\varphi_j$ for $j \in \mathbb{N}_0$ be elements of $\Sc(\mathbb{R}^n)$ with
\begin{align*}
 supp \ \varphi_0 &\subset  \{|\xi| \leq 2\}, \\
 supp \ \varphi_j &\subset\{2^{j-1}\leq |\xi| \leq 2^{j+1}\} \text{ for } j \in \mathbb{N}, \\
 \sum_{j=0}^{\infty} \varphi_j(\xi)&=1 \text{ for all } \xi \in \mathbb{R}^n, \\
 |D^{\alpha} \varphi_j(\xi)| &\leq c_{\alpha} 2^{-j|\alpha|} \text{ for all } \alpha \in \mathbb{N}_0^n.
\end{align*}
Then we call $\{\varphi_j\}_{j=0}^{\infty}$ a smooth dyadic resolution of unity. For instance one can choose $\Psi \in \Sc(\mathbb{R}^n)$ with $\Psi(\xi)=1$ for $|\xi|\leq 1$ and $\Psi(\xi)=0$ for $|\xi|\geq 2$ and set 
\begin{align*}
 \varphi_0(\xi):=\Psi(\xi), \quad \varphi_1(\xi):=\Psi(\xi/2)-\Psi(\xi), \quad \varphi_j(\xi):=\varphi_1(2^{-j+1}\xi) \text{ for } j \in \mathbb{N}.
\end{align*}
\begin{Definition}
 \label{GrundDefinitionB}
Let $s \in \mathbb{R}$, $0<p\leq \infty$, $0<q\leq \infty$ and $\{\varphi_j\}_{j=0}^{\infty}$ be a smooth dyadic resolution of unity. Then $\bspq$ is the collection of all  $f\in \Sc'(\mathbb{R}^n)$ such that the quasi-norm
\begin{align*}
 \|f|\bspq\|:=\left(\sum_{j=0}^{\infty} 2^{jsq}\|(\varphi_j \hat{f})\check{\ }|L_p(\R^n)\|^q\right)^{\frac{1}{q}}
\end{align*}
(modified if $q=\infty$) is finite.
\end{Definition}
\begin{Definition}
 \label{GrundDefinitionF}
Let $s \in \mathbb{R}$, $0<p<\infty$, $0<q\leq \infty$ and $\{\varphi_j\}_{j=0}^{\infty}$ be a smooth dyadic resolution of unity. Then $\fspq$ is the collection of all $f\in \Sc'(\mathbb{R}^n)$ such that the quasi-norm
\begin{align*}
 \|f|\fspq\|:=\left\|\left(\sum_{j=0}^{\infty} 2^{jsq} \left|(\varphi_j \hat{f})\check{\ }(\cdot)\right|^q \right)^{\frac{1}{q}}\big|L_p(\R^n)\right\|\end{align*}
(modified if $q=\infty$) is finite.
\end{Definition}
One can show that the introduced quasi-norms\footnote{In the following we will use the term ``norm`` even if we only have quasi-norms for $p<1$ or $q<1$.} for two different smooth dyadic resolutions of unity are equivalent for fixed $p$, $q$ and $s$, i.\,e.\ that the so defined spaces are equal. This follows from Fourier multiplier theorems, see \cite[Section 2.3.2]{Tri83}. Furthermore, the so defined spaces are (quasi)-Banach spaces.

\subsection{Basic properties of function spaces $\BR$ and $\FR$}
The following proposition is known as the Sobolev embedding of the Besov and Triebel-Lizorkin function spaces.
\begin{Proposition}[Sobolev embeddings for $\FR$]
\label{Sobolev}
Let $s_1<s_0 \in \R$, $0<p_0<p_1<\infty$ resp.\ $0<p_0<p_1<\infty$ and $0<q_0,q_1\leq \infty$. Furthermore, let
\begin{align*}
 s_0-\frac{n}{p_0} \geq s_1 - \frac{n}{p_1}.
\end{align*}
Then
\begin{align*}
 B_{p_0,q_0}^{s_0}(\R^n) &\hookrightarrow 
 B_{p_1,q_0}^{s_1}(\R^n)
\intertext{ and }
 F_{p_0,q_0}^{s_0}(\R^n) &\hookrightarrow 
 F_{p_1,q_1}^{s_1}(\R^n).
\end{align*}

\end{Proposition}
\begin{Definition}
 Let $s \in \R, 0<p\leq \infty, 0<q\leq \infty$ and $l \in \N_0$. Then we define
 \begin{align*}
  \sigma_p^l := l \cdot \left(\frac{1}{p}-1\right)_+ \text{ and } \sigma_{p,q}^l:= l \cdot \left(\frac{1}{\min(p,q)-1}\right)_+,
 \end{align*}
 where $a_+=\max(a,0)$. If $l$ is equal to the dimension $n$ of $\R^n$, we simply write
 \begin{align*}
  \sigma_p:=\sigma_p^n \text{ and } \sigma_{p,q}:=\sigma_{p,q}^n.
 \end{align*}
\end{Definition}
The next proposition, the so called Fatou property, is a classical observation for function spaces $\bspq$ and $\fspq$, see \cite{Fra86}.
\begin{Definition}
 Let $A$ be a quasi-Banach space with $\Sc(\R^n) \hookrightarrow A \hookrightarrow \Sc'(\R^n)$. Then we say that $A$ has the Fatou property if there exists a constant $c$ such that:
 If a sequence $\{f_n\}_{n \in \N} \subset A$ converges to $f$ with respect to the weak topology in $\Sc'(\R^n)$ and if $\|f_n|A\|\leq D$, then $f \in A$ and $\|f|A\|\leq c \cdot D$.
\end{Definition}

\begin{Proposition}[Fatou property for $\BR$ and $\FR$]
\label{Fatou}
Let $s \in \R$, $0<p\leq \infty$ resp.\ $0<p<\infty$ and $0<q\leq \infty$. Then $\bspq$ and $\fspq$ have the Fatou property.
\end{Proposition}

\begin{Remark}
\label{HoldB}
 If $\rho>0$ and $\rho \notin \N$, then $\hold[\rho]=B_{\infty,\infty}^{\rho}(\R^n)$. This is a classical observation, for instance see \cite[Sections 1.2.2, 2.6.5]{Tri92} or for the original source \cite[Lemma 4]{Zyg45}.
\end{Remark}

\begin{Proposition}[Homogeneity property of $\FR$]
\label{homogen}
 Let $0<p<\infty$, $0<q\leq \infty$ and $s>\sigma_{p,q}$. Then for all $\lambda \in (0,1]$ and $f \in \FR$ with 
 \begin{align*}
  supp\ f \subset B_{\lambda}=\{x \in \R^n: |x|<\lambda\}
 \end{align*}
 it holds
 \begin{align*}
  \|f(\lambda \cdot)|\FR\| \sim \lambda^{s-\frac{n}{p}} \|f|\FR\|
 \end{align*}
\end{Proposition}
\begin{Proof}
 This is a reformulation of \cite[Theorem 2.11]{Tri08} going back to \cite[Corollary 5.16]{Tri01}. 
\end{Proof}
\begin{Remark}
 By the Fourier analytical definition of $\FR$ in Definition \ref{GrundDefinitionF} one can also assume 
 \begin{align*}
  supp\ f \subset B_{\lambda}(y)=\{x \in \R^n: |x-y|<\lambda\}
 \end{align*}
 and the proposition is true with constants independent of $y \in \R^n$.
\end{Remark}

\subsection{Fubini property of Triebel-Lizorkin function spaces on $\R^n$}

Let $l \in \N$, $l<n$ and $1\leq j_1 < \ldots< j_l \leq n$. We set
\begin{align*}
x^{j_1,\ldots,j_l}:=(x_1,\ldots,x_{j_1-1},x_{j_1+1},\ldots,,x_{j_l-1},x_{j_l+1},\ldots,x_n) \in \R^{n-l}
\end{align*}
with obvious modifications if $j_1=1$ or $j_l=n$.
Let $ f: \R^n \rightarrow \C$. Then we define the function 
\begin{align*}
 f^{x^{j_1,\ldots,j_l}}(x_{j_1},\ldots, x_{j_l}):=f(x_1,\ldots,x_{j_1-1},x_{j_1},x_{j_1+1},\ldots,,x_{j_l-1},x_{j_l},x_{j_l+1},\ldots,x_n)
\end{align*}
as a function on $\R^l$ for a fixed $x^{j_1,\ldots,j_l} \in \R^{n-l}$.

\begin{Proposition}
 \label{Fubini}
 Let $n \geq 2$, $l \in \N$ and $l<n$. Let
 \begin{align*}
 0<p < \infty, 0< q \leq \infty \text{ and } s> \sigma_{p,q}.
 \end{align*}
Then $\FR$ has the Fubini property, i.\,e.\ for all $f \in  \FR$ it holds
 \begin{align}
 \label{Fubeq}
  \|f|\FR\| \sim \sum_{1\leq j_1 < \ldots< j_l \leq n}  \Big\| \big\| f^{x^{j_1,\ldots,j_l}} |\FO[\R^{l}] \big\| | L_p(\R^{n-l}) \Big\|  
 \end{align}
\end{Proposition}
\begin{proof}
 The proof is an application of the classical Fubini property for $\FR$, see \cite[Theorem 4.4]{Tri01}: By the Fubini property for dimension $l$ (instead of $n$) we get 
 \begin{align*}
  \big\| f^{x^{j_1,\ldots,j_l}} |\FO[\R^{l}] \big\| \sim  \sum_{k=1}^{l} \left\|\big\| \left(f^{x^{j_1,\ldots,j_l}}\right)^{x^{k}} |\FO[\R] \big\| | L_p(\R^{l-1}) \right\|.
 \end{align*}
Obviously, for a suitable $1\leq j \leq n$ one has 
\begin{align*}
 \left(f^{x^{j_1,\ldots,j_l}}\right)^{x^{j_k}}=f^{x^j}.
\end{align*}
Putting the right-hand side into \eqref{Fubeq} gives
\begin{align*}
 \sum_{1\leq j_1 < \ldots< j_l \leq n}  \Big\| \big\| f^{x^{j_1,\ldots,j_l}} |\FO[\R^{l}] \big\| | L_p(\R^{n-l}) \Big\| \sim \sum_{k=1}^n \Big\| \big\| f^{x^k} |\FO[\R] \big\| | L_p(\R^{n-1}) \Big\|.
\end{align*}
But now, by the Fubini property for dimension $n$ the right hand side is equivalent to $\|f | \FR\|$. 
\end{proof}

\section{Atomic decompositions and local means}
At first we describe the concept of atoms as one can find it in \cite[Definition 13.3]{Tri97}, now generalized using ideas from \cite{Skr98} and \cite{TrW96}. 

In particular, this gives the possibility to omit the distinction between $\nu=0$ and $\nu \in \mathbb{N}$ and now the usual parameters $K$ and $L$ are nonnegative real numbers instead of natural numbers.

\subsection{Generalized atoms}
Let $Q_{\nu,m}:=\{x \in \mathbb{R}^n: |x_i-2^{-\nu}m_i|\leq 2^{-\nu-1}\}$ be the cube with sides parallel to the axes, with center at $2^{-\nu}m$ and side length $2^{-\nu}$ for $m \in \mathbb{Z}^n$ and $\nu \in \mathbb{N}_0$.
\begin{Definition}
\label{Atoms}
Let $s \in \mathbb{R}$, $0<p\leq\infty$, $K,L \in \R$ and $K,L\geq 0$. Furthermore let $d>1$, $C>0$, $\nu \in \mathbb{N}_0, m \in \mathbb{Z}^n$. A function $a:\R^n \rightarrow \mathbb{C}$ is called $(s,p)_{K,L}$-atom located at $Q_{\nu,m}$ if 
\begin{align}
 \label{Atom1}supp \ a &\subset d \cdot Q_{\nu,m}  \\
 \label{Atom2} \|a(2^{-\nu}\cdot)|\hold[K]\| &\leq C \cdot 2^{-\nu(s-\frac{n}{p})} \text{ for } K>0
\end{align}
and for every $\psi \in \hold[L]$ it holds 	
\begin{align}
\label{Atom3} \left| \int_{d \cdot Q_{\nu,m}} \psi(x) a(x) \ dx \right|\leq C \cdot  2^{-\nu\left(s+L+n\left(1-\frac{1}{p}\right)\right)} \|\psi|\hold[L]\|,
\end{align}
where $\hold[K]$ resp.\ $\hold[L]$ are replaced by $L_{\infty}(\R^n)$ if $K=0$ resp.\ $L=0$.
The constant in the exponent will be shortened by $\Kap:=s+L+n\left(1-\frac{1}{p}\right)$. 
\end{Definition}
\begin{Remark}
If $L=0$, then condition \eqref{Atom3} follows from $\eqref{Atom1}$ and $\eqref{Atom2}$ with $K=0$.
If $K=0$, then we only require $a$ to be suitably bounded.

Later on, we will choose one $(s,p)_{K,L}$-atom for every $\nu \in \N_0$ and $m \in \Z^n$. Then the parameter $d>1$ shall be the same for all these atoms - it describes the overlap of these atoms at one fixed level $\nu \in \N_0$.
\end{Remark}

\begin{Remark}
\label{usualform}
The usual formulation of the general derivative condition $\eqref{Atom2}$ as in \cite{Tri97} was
\begin{align}
\label{Atom21}
  |D^{\alpha} a(x)| &\leq 2^{-\nu\left(s-\frac{n}{p}\right)+|\alpha|\nu} \text{ for all } |\alpha|\leq K
\end{align}
for a $K \in \N_0$. The modification here has been suggested in \cite{TrW96}. It is easy to see that \eqref{Atom2} follows from \eqref{Atom21} if $K$ is a natural number since $C^K(\R^n) \hookrightarrow \hold[K]$.  
\end{Remark}
\begin{Remark}
\label{Ordered}
The usual formulation of the general moment condition $\eqref{Atom3}$ as in \cite{Tri97} was
\begin{align}
  \label{Atom31}\int_{\mathbb{R}^n} &x^{\beta} a(x) \ dx=0 \text{ if }  |\beta|\leq L-1
\end{align}
for $\nu \in \mathbb{N}$, so $\nu \neq 0$. The modification here was suggested in \cite[Lemma 1]{Skr98} for natural numbers $L+1$ (using $C^{L}(\R^n)$ instead of $\hold[L]$). Now we extended this definition to general positive $L$. For natural $L-1$ one can derive $\eqref{Atom3}$ from $\eqref{Atom31}$ using a Taylor expansion, see \cite[Lemma 1, (12) and (14)]{Skr98} or the upcoming Lemma \ref{Momenter}. Hence formulation $\eqref{Atom3}$ is a generalization. 

An alternative formulation of the general moment condition \eqref{Atom3} is given by
\begin{align}
 \label{Atom32}
\left| \int_{d \cdot Q_{\nu,m}} (x-2^{-\nu}m)^{\beta} a(x) \ dx \right|\leq C \cdot  2^{-\nu\Kap} \text{ if } |\beta|\leq \Lint.
\end{align}
Obviously, this condition is covered by the general moment condition \eqref{Atom3}. For the other direction see also \cite[Lemma 1, (12) and (14)]{Skr98} or the upcoming Remark \ref{MomentBem}, in particular \eqref{IntMomenter}. It is also possible to assume this condition for all $\beta\in \N^n$ since the statements for $|\beta|\geq L$ follow from the support condition \eqref{Atom1} and the boundedness condition included in \eqref{Atom2}.

This shows that both general conditions \eqref{Atom2} and \eqref{Atom3} are ordered in $K$ resp.\ $L$, i.\,e.\ the conditions get stricter for increasing $K$ resp.\ $L$.
\end{Remark}

\subsection{Sequence spaces}
We introduce the sequence spaces $b_{p,q}$ and $f_{p,q}$ adapted to $\Z^n$. For this we refer to \cite[Definition 13.5]{Tri97}.
\begin{Definition}
\label{DefSeq}
 Let $0<p\leq \infty$, $0<q\leq \infty$ and 
 \begin{align*}
  \lambda=\left\{ \lambda_{\nu,m} \in \mathbb{C}: \nu \in \mathbb{N}_0, m \in \mathbb{Z}^n\right\}.
 \end{align*}
We set
\begin{align*}
 b_{p,q}:=\left\{\lambda: \|\lambda|b_{p,q}\|=\left(\sum_{\nu=0}^{\infty} \left(\sum_{m \in \mathbb{Z}^n} |\lambda_{\nu,m}|^p\right)^{\frac{q}{p}} \right)^{\frac{1}{q}} <\infty \right\}
\end{align*}
and
\begin{align*}
 f_{p,q}:=\left\{\lambda: \|\lambda|f_{p,q}\|=\left\|\left(\sum_{\nu=0}^{\infty} \sum_{m \in \mathbb{Z}^n} |\lambda_{\nu,m}\chi_{\nu,m}^{(p)}(\cdot)|^q\right)^{\frac{1}{q}}\big|L_p(\R^n)\right\| <\infty \right\}
\end{align*}
(modified in the case $p=\infty$ or $q=\infty$), where $\chi_{\nu,m}^{(p)}$ is the $L_p(\R^n)$-normalized characteristic function of the cube $Q_{\nu,m}$, i.\,e\
\begin{align*}
 \chi_{\nu,m}^{(p)}=2^{\frac{\nu n}{p}} \text{ if } x \in Q_{\nu,m} \text{ and }  \chi_{\nu,m}^{(p)}=0 \text{ if } x \notin Q_{\nu,m}.
\end{align*}
\end{Definition}

\subsection{Local means}
\label{LocalMeans}
Let $N \in \N_0$ be given. We choose $k_0,k \in \Sc(\R^n)$ with compact support - e.\,g.\ $supp \ k_0, supp \ k \subset e \cdot Q_{0,0}$ for a suitable $e>0$ - such that
\begin{align}
\label{MeanMomenter}
D^{\alpha}\hat{k}(0)=0 \text{ if } |\alpha|< N,
\end{align}
while $\hat{k_0}(0)\neq 0$. Furthermore, let there be an $\varepsilon>0$ such that $\hat{k}(x) \neq 0$ for $0<|x|<\varepsilon$.
 
Such a choice is possible, see \cite[Section 11.2]{Tri97}. We set $k_j(x):=2^{jn} k(2^jx)$ for $j\in \N$.
\begin{Proposition}
\label{RychkovLocal}
Let $N \in \N_0$ and $N>s$.

(i) Let $0<p\leq\infty$ and $0<q\leq \infty$. Then
\begin{align*}
  \|f|\bspq\|_{k_0,k}:= \|k_0*f|L_p(\R^n)\|+\left(\sum_{j=1}^{\infty} 2^{jsq}\|k_j*f|L_p(\R^n)\|^q\right)^{\frac{1}{q}}
\end{align*}
(modified for $q=\infty$) is an equivalent norm for $\|\cdot|\bspq\|$. It holds
\begin{align*}
\bspq=\left\{f \in \Sc'(\mathbb{R}^n): 
 \|f|\bspq\|_{k_0,k}<\infty \right\}.
\end{align*}

(ii) Let $0<p<\infty$ and $0<q\leq \infty$. Then
\begin{align*}
  \|f|\fspq\|_{k_0,k}:=\|k_0*f|L_p(\R^n)\|+\left\|\left(\sum_{j=1}^{\infty} 2^{jsq} \left|(k_j*f)(\cdot)\right|^q \right)^{\frac{1}{q}}\big|L_p(\R^n)\right\|
\end{align*}
(modified for $q=\infty$) is an equivalent norm for $\|\cdot|\fspq\|$. It holds
\begin{align*}
 \fspq=\left\{f \in \Sc'(\mathbb{R}^n): 
 \|f|\fspq\|_{k_0,k}<\infty \right\}.
\end{align*}
\end{Proposition}
\begin{Remark}
 This proposition is due to \cite{Ryc99}. Some minor technicalities of the proof where modified in the fourth step of \cite[Theorem 2.1]{Sch10} (for the more general vector-valued case).
\end{Remark}

\section{Function spaces on domains}
\label{defdom}
Let $\Om$ be a domain, i.\,e.\ non-empty open set, in $\R^n$, $\Gamma=\partial \Om$ its boundary and $\overline{\Om}$ its closure. By $D(\Om)$ we denote the set of all functions $f \in D(\R^n)$ with support in $\Om$ and by $D'(\Om)$ its usual topological dual space. 

Denote by $g|\Om \in D'(\Om)$ the restriction of $g$ to $\Om$, hence $ (g|\Om)(\ph)=g(\ph) \text{ for } \ph \in D(\Om)$. We introduce 
\begin{align*}
 \FO:=\{f \in D'(\Om): f&=g|\Om \text{ for some } g \in \FR\}, \\
     \|f|\FO\|&=\inf \|g|\FR\|,
\end{align*}
where the infimum is taken over all $g\in \FR$ with $g|\Om=f$. Moreover, let
\begin{align*}
\FtBar:=\{f \in \FR: supp \, f \in \overline{\Om} \} 
\end{align*}
with the quasi-norm from $\FR$. Then 
\begin{align*}
\Ft:=\{f \in D'(\Om): f&=g|\Om \text{ for some } g \in \FtBar\}, \\
\|f|\Ft\|&=\inf \|g|\FtBar\|,
\end{align*}
where the infimum is taken over all $g\in \FtBar$ with $g|\Omega=f$. 
\begin{Remark}
 Let $0<p<\infty$, $0<q\leq \infty$ and $s>\sigma_p$. Then by the Sobolev embedding (Proposition \ref{Sobolev}) we have
 \begin{align*}
  \FR \hookrightarrow L_{\max(1,p)}(\R^n). 
 \end{align*}
Furthermore, let $|\partial \Om|=0$ which will always be the case for the domains $\Om$ considered in this text. Hence the only function $h \in \FR$ with $supp\ h \subset \partial \Om$ is $h\equiv 0$. This shows that $\FtBar \cong \Ft$. For further explanation see \cite[Section 5.4]{Tri01}.

\end{Remark}

\section{Wavelets on $\R^n$ and on domains}
\label{Whitney}
Let $\Om$ be a domain in $\R^n$ and $\Gamma=\partial \Om$ its boundary. We start with a Whitney decomposition of $\Om$ in the same way as in \cite[Section 2.1.2]{Tri08}. For more details regarding the Whitney decomposition see \cite[Theorem 3, p.\ 16]{Ste70}. Let
\begin{align*}
 Q^0_{l,r} \subset Q^1_{l,r}, \quad l \in \N_0, r =1,\ldots,M_j \text{ with } M_j \in \overline{\N}_0
\end{align*}
be concentric (open) cubes in $\R^n$, sides parallel to the axes of coordinates, centred at $2^{-l}m^r$ for an $m^r \in \Z^n$. The sidelength of $Q^0_{l,r}$ shall be $2^{-l}$, the sidelength of $Q^1_{l,r}$ shall be $2^{-l+1}$. We call this collection of cubes a Whitney decomposition of $\Om$ if the cubes $Q^0_{l,r}$ are pairwise disjoint, if
\begin{align*}
 \Om=\bigcup_{l,r} \overline{Q}_{l,r}^{\,0}, \quad \dist(Q^1_{0,r},\Gamma) \gtrsim 1 \quad \text{and} \quad \dist(Q^1_{l,r},\Gamma) \sim 2^{-l} \text{ for } l \in \N . 
\end{align*}
By the construction in \cite[Theorem 3, p.\ 16]{Ste70} one can furthermore assume that for adjacent cubes $Q^0_{l,r}$ and $Q^0_{l',r'}$ it holds $|l-l'|\leq 1$.  
\begin{center}
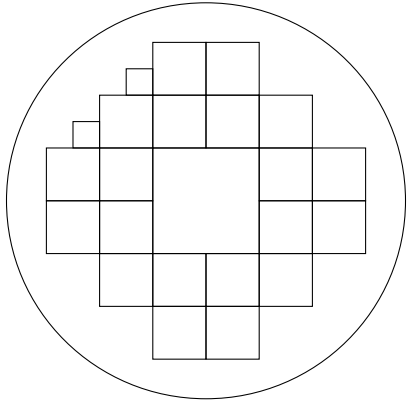

 \begin{tikzpicture}[scale=0.35]
 \draw (0,0) circle (7.5cm);
 \draw (-2,-2) -- (-2,2) -- (2,2) -- (2,-2) --cycle;
 \draw (-2,2) -- (0,2) -- (0,4) -- (-2,4) --cycle;
 \draw (0,2) -- (2,2) -- (2,4) -- (0,4) --cycle;
 \draw (-4,2) -- (-2,2) -- (-2,4) -- (-4,4) --cycle;
 \draw (-4,0) -- (-2,0) -- (-2,2) -- (-4,2) --cycle;
 \draw (-4,-2) -- (-2,-2) -- (-2,0) -- (-4,0) --cycle;
 \draw (-4,-4) -- (-2,-4) -- (-2,-2) -- (-4,-2) --cycle;
 \draw (-2,-4) -- (0,-4) -- (0,-2) -- (-2,-2) --cycle;
 \draw (0,-4) -- (2,-4) -- (2,-2) -- (0,-2) --cycle;
 \draw (2,-4) -- (4,-4) -- (4,-2) -- (2,-2) --cycle;
 \draw (2,-2) -- (4,-2) -- (4,0) -- (2,0) --cycle;
 \draw (2,0) -- (4,0) -- (4,2) -- (2,2) --cycle;
 \draw (2,2) -- (4,2) -- (4,4) -- (2,4) --cycle;

 \draw (-3,4) -- (-2,4) -- (-2,5) -- (-3,5) --cycle;
 \draw (-5,2) -- (-4,2) -- (-4,3) -- (-5,3) --cycle;
 \draw (-2,4) -- (0,4) -- (0,6) -- (-2,6) --cycle;
 \draw (0,4) -- (2,4) -- (2,6) -- (0,6) --cycle;
 \draw (-6,0) -- (-4,0) -- (-4,2) -- (-6,2) --cycle;
 \draw (-6,-2) -- (-4,-2) -- (-4,0) -- (-6,0) --cycle; 
 \draw (-2,-4) -- (0,-4) -- (0,-6) -- (-2,-6) --cycle;
 \draw (0,-4) -- (2,-4) -- (2,-6) -- (0,-6) --cycle;
 \draw (6,0) -- (4,0) -- (4,2) -- (6,2) --cycle;
 \draw (6,-2) -- (4,-2) -- (4,0) -- (6,0) --cycle;
\end{tikzpicture}
\captionof{figure}{First cubes $Q^0_{l,r}$ of the Whitney decomposition of a ball}
\end{center}

Now we introduce the interior sequence spaces for domains $\Om$. This is taken over from \cite[Section 2.1.2]{Tri08}.
\begin{Definition}[Interior sequence spaces for domains $\Om$]
\label{sequence}
Let $\Om$ be a domain with $\Om\neq \R^n$, $\Gamma=\partial \Om$ its boundary. 
Let $c_1,c_2,c_3$ be positive constants and 
\begin{align*}
 \Z^{\Om}=\{x_r^j \in \Om: j \in \N_0, r=1,\ldots,N_j \}
\end{align*}
with $N_j \in \overline{\N}_0$ such that
\begin{align*}
 |x_r^j - x_{r'}^{j}| \geq c_1 2^{-j} \text{ for } j \in \N_0 \text{ and } r \neq r' 
\end{align*}
and
\begin{align}
\label{distprop}
 \dist(B(x_r^j, c_2 2^{-j}),\Gamma) \geq c_3 2^{-j} \text{ for } j \in \N_0 \text{ and } r=1,\ldots,N_j. 
\end{align}
The interior sequence space $f_{p,q}^s(\Z^{\Om})$ adapted to $\Z^{\Om}$ is the collection of all sequences $\lambda \in \Z^{\Om}$ such that
\begin{align*}
 \|\lambda|\fO\|:=\left\|\left(\sum_{j=0}^{\infty} \sum_{r=1}^{N_j} 2^{jsq} |\lambda_{j}^r\, \chi_{j}^r|^{q}\right)^{\frac{1}{q}}|L_p(\Om)\right\|
\end{align*}
is finite. Here $\chi_{j}^r$ is the characteristic function of $B(x_r^j, c_3 2^{-j})$.
\end{Definition}
\begin{Remark}
We can also introduce a similar definition for $\Om=\R^n$. Then we just omit property \eqref{distprop}. A special admissible decomposition choice are the grids $\{2^{-j}\Z^n,$ $j=0,\ldots\}$ which are used for the usual atomic decomposition theorems for $\R^n$, see for instance Definition \ref{DefSeq}.
\end{Remark}
\begin{Remark}
\label{Whitneyatom}
 One possible way to find such a decomposition emerging from an interior wavelet decomposition (consisting of grids with distance $2^{-j}$) for $\Om \neq \R^n$  is given in \cite[Section 2.1.2]{Tri08}. To put it simply, one takes the Whitney decomposition as a starting point, then decomposes every cube $Q_{l,r}^0$ into $2^{(j-l)n}$ cubes of sidelength $2^{-j}$ and takes the centres as the points $x_j^r$. This is perfectly adapted to the construction of $u$-wavelet systems on $\Om$.
\end{Remark}
Later on we also need sequence spaces $\fO[\overline{\Om}]$ where the balls around $x_r^j$ can also be located near or on the boundary. This is taken over from \cite[Definition 5.23]{Tri08}.
\begin{Definition}[Sequence spaces for domains with values on the boundary]
\label{extsequence}

Let \\ $\Om$ with $\Om\neq \R^n$ be a domain, $\Gamma=\partial \Om$ its boundary. 
Let $c_1,c_2$ be positive constants and 
\begin{align*}
 \Z^{\overline{\Om}}=\{x_r^j \in \Om: j \in \N_0, r=1,\ldots,N_j \}
\end{align*}
with $N_j \in \overline{\N}_0$ such that
\begin{align*}
 |x_r^j - x_{r'}^{j}| \geq c_1 2^{-j} \text{ for } j \in \N_0 \text{ and } r \neq r'.
\end{align*}
The sequence space $f_{p,q}^s(\Z^{\overline{\Om}})$ adapted to $\Z^{\overline{\Om}}$ is the collection of all sequences $\lambda \in \Z^{\overline{\Om}}$ such that
\begin{align*}
 \|\lambda|\fO[\overline{\Om}]\|:=\left\|\left(\sum_{j=0}^{\infty} \sum_{r=1}^{N_j} 2^{jsq} |\lambda_{j}^r\, \chi_{j}^r|^{q}\right)^{\frac{1}{q}}|L_p(\Om)\right\|
\end{align*}
is finite. Here $\chi_{j}^r$ is the characteristic function of $B(x_r^j, c_2 2^{-j})$.
\end{Definition}

Now we can fix what we understand under a $u$-wavelet system. At first we define $u$-wavelet systems for $\R^n$. A special case of these are Daubechies wavelets as presented in \cite[Section 1.2.1]{Tri08}. Further details for Daubechies wavelets can be found in \cite{Dau92} and \cite[Section 4]{Woj97}.\ Additionally, we now allow $u=0$, which also incorporates the Haar wavelets as $0$-wavelets on $\R^n$. For more information and the definition of the Haar wavelet see the original source \cite{Haa10}, \cite[Section 1.1]{Woj97} or \cite[Section 2.5.1]{Tri08}.
\begin{Definition}[$u$-wavelet system for $\R^n$]
\label{u-waveletr}
Let $u \in \mathbb{N}_0$. Let 
\begin{align*}
 \Z^{\R^n}=\{x_r^j \in \R^n: j \in \N_0, r=1,\ldots,N_j \}
\end{align*}
be a collection of points as in Definition \ref{sequence} (and the following remark) for $\Om=\R^n$ omiting property \eqref{distprop}. Then 
\begin{align*}
\Phi=\left\{\Phi_r^j: j \in \N_0, r=1,\ldots,N_j \right\} \subset C^u (\R^n)
\end{align*}
is called an oscillating interior u-wavelet system for $\R^n$, adapted to $\Z^{\R^n}$, if it fulfils
\begin{itemize}
 \item support conditions: For some $c_1>0$ let
\begin{align*}
 supp \,  \Phi_r^j \subset B(x_{r}^j, c_1 2^{-j}), \, j \in \mathbb{N}_0,\, r=1,\ldots,N_j,
\end{align*}
 \item derivative conditions: For some $c_2>0$ and all $\alpha \in \mathbb{N}_0^n$ with $0\leq |\alpha| \leq u$ let
\begin{align*}
 \left|D^{\alpha} \Phi_{r}^j(x)\right|\leq c_2 2^{j\frac{n}{2}+j|\alpha|}, \, j \in \mathbb{N}_0,\, r=1,\ldots,N_j, x \in \R^n.
\end{align*}
\item (substitute) moment conditions: For some $c_3>0$ 
\begin{align*}
 \left|\int_{\R^n} \psi(x)\Phi_r^j(x) \ dx\right|&\leq c_3 2^{-j\frac{n}{2}-ju}\left\|\psi|C^{u}(\R^n)\right\|
\end{align*}
for all $\psi \in C^{u}(\R^n)$.
\end{itemize}
\end{Definition}
\begin{Remark}
Here we use a somehow different notation as in \cite{Tri08} since we include $\R^n$ in the definition of $u$-wavelet systems. Since wavelet bases for $\FR$ have already been constructed in \cite[Section 3]{Tri06} this is not a substantial change. Only the (substitute) moment conditions are a genuine generalization. But this will be covered by the observations in Section \ref{atomsec}.
\end{Remark}

Now we consider $u$-wavelet systems on domains $\Om \neq \R^n$ - this means $\partial \Om \neq 0$. We adopt the definition of $u$-wavelet systems from \cite[Definition 6.3]{Tri08} allowing $\Om$ now to be unbounded. There are at least three kind of definitions for $u$-wavelet bases in \cite{Tri08} - Definition 2.4, 5.25 and 6.3. In Definition 2.4 Triebel derived $u$-wavelet systems from a Daubechies wavelet basis of $L_2(\R^n)$ resp.\ of $\FR$ while Definitions 5.25 and 6.3 were more general but restricted to bounded domains. Additionally, we now allow $u=0$, which also incorporates the Haar wavelets on arbitrary domains $\Om$, described in \cite[Section 2.5.1]{Tri08},  as $0$-wavelets on $\Om$.
\begin{Definition}[$u$-wavelet system for $\Om$]
\label{u-wavelet}
Let $\Om$ be an arbitrary domain in $\R^n$ with $\Om\neq \R^n$. Let $\Gamma=\partial \Om$ and let $u \in \mathbb{N}_0$. Let 
\begin{align*}
 \Z^{\Om}=\{x_r^j \in \Om: j \in \N_0, r=1,\ldots,N_j \}
\end{align*}
be a collection of points as in Definition \ref{sequence}. 
 
(i) Then 
\begin{align*}
\Phi=\left\{\Phi_r^j: j \in \mathbb{N}_0, r=1, \ldots, N_j\right\} \subset C^u (\Om)
\end{align*}
is called a $u$-wavelet system in $\overline{\Om}$, adapted to $\Z^{\Om}$, if it fulfils
\begin{itemize}
 \item support conditions: For some $c_1>0$ let
\begin{align*}
 supp \,  \Phi_r^j \subset B(x_{r}^j, c_1 2^{-j}) \cap \overline{\Om}, \, j \in \mathbb{N}_0,\, r=1, \ldots, N_j,
\end{align*}
 \item derivative conditions: For some $c_2>0$ and all $\alpha \in \mathbb{N}_0^n$ with $0\leq |\alpha| \leq u$ let
\begin{align*}
 \left|D^{\alpha} \Phi_{r}^j(x)\right|\leq c_2 2^{j\frac{n}{2}+j|\alpha|}, \, j \in \mathbb{N}_0,\, r=1, \ldots, N_j, x \in \Om.
\end{align*}
\end{itemize}

(ii) The above $u$-wavelet system is called oscillating if it fulfils 
\begin{itemize}
\item (substitute) moment conditions: Let $c_3$ and $c_4<c_5$ be constants such that
\begin{align*}
 \dist(B(x_{r}^0,c_3),\Gamma)&\geq c_4, \text{ for } r=1, \ldots, \mathbb{N}_0 \text{ and } \\
 \left|\int_{\Om} \psi(x)\Phi_r^j(x) \ dx\right|&\leq c_3 2^{-j\frac{n}{2}-ju}\left\|\psi|C^{u}(\Om)\right\|, \psi \in C^{u}(\Om)
\end{align*}
for all $\Phi_r^j$ with $j \in \mathbb{N}$ and 
\begin{align*}
 \dist(B(x_r^j,c_1 2^{-j}),\Gamma) \notin [c_4 2^{-j},c_5 2^{-j}]. 
\end{align*}
\end{itemize}

(iii) An oscillating $u$-wavelet system is called interior if it fulfils
\begin{itemize}
\item interior support conditions, namely
\begin{align*}
 \dist(B(x_{r}^j,c_1 2^{-j}),\Gamma)\geq c_4 2^{-j}, \, j \in \mathbb{N}_0,\, r=1, \ldots, N_j.
\end{align*}
\end{itemize}
\end{Definition}

\begin{Remark}
 Condition (iii) implies that all $\Phi_r^j$ are supported inside the domain $\Om$.
\end{Remark}

We now take over the definition of a $u$-wavelet basis from Definition 2.31 in \cite{Tri08} but our definition is now more general since we use a more general definition of u-wavelet systems. We also incorporate $\Om=\R^n$ and $u=0$.
\begin{Definition}[$u$-wavelet basis]
\label{u-basis}
Let $\Om$ be an arbitrary domain in $\R^n$ and let $u \in \N_0$. Then 
\begin{align*}
 \{\Phi_r^j: j \in \N_0, r=1, \ldots, N_j \} \text{ with } N_j \in \overline{\N}_0
\end{align*}
is called an orthonormal $u$-wavelet basis in $L_2(\Om)$ if it is both an orthonormal basis in $L_2(\Om)$ and an oscillating interior u-wavelet system - according to Definition \ref{u-wavelet} in case $\Om\neq \R^n$ and according to Definition \ref{u-waveletr} in case $\Om=\R^n$. 
\end{Definition}

\begin{Theorem}
\label{orthbases}
Let $\Om$ be an arbitrary domain in $\R^n$. For any $u \in \N_0$ there are orthonormal u-wavelet bases for a suitable collection of points $\Z^{\Om}$. 
\end{Theorem}
\begin{Proof}
 For the case $\Om=\R^n$ on can use Daubechies wavelets as in \cite[Section 1.2.1]{Tri08} and $\Z^{\R^n}=\Z^n$. The case $\Om\neq \R^n$ is covered by \cite[Theorem 2.33]{Tri08}. 
\end{Proof}

\begin{Definition}[$u$-Riesz basis]
\label{u-Riesz}
Let $\Om$ be an arbitrary domain in $\R^n$. An [oscillating] \{interior\} $u$-wavelet system $\Phi=\{\Phi_r^j\}$ as introduced in Definition \ref{u-wavelet} is called an [oscillating] \{interior\} u-Riesz basis for $A$ where $A \subset D'(\Om)$ is a suitable function space on $\Om$ together with a suitable sequence space $a$, if it has the following properties
\begin{enumerate}
 \item An element $f \in D'(\Om)$ belongs to $A$ if, and only if, it can be represented as
 \begin{align}
 \label{represent}
 f=\sum_{j=0}^{\infty}\sum_{r=1}^{N_j} \lambda_r^j(f) 2^{-\frac{jn}{2}}\Phi_r^j, \quad \lambda \in a
 \end{align}
 with unconditional convergence in $A$.
 \item
 The representation \eqref{represent} is unique and the coefficient mappings
 \begin{align*}
  \lambda_r^j: f \mapsto \lambda_r^j(f)
 \end{align*}
 are linear and continuous functionals on $A$. 
 \item
 Furthermore, 
 \begin{align*}
  f \mapsto \{\lambda_r^j(f)\} 
 \end{align*}
 is an isomorphic map of $A$ onto $a$. 
\end{enumerate}
\end{Definition}

\begin{Proposition}[$u$-wavelet basis for $\R^n$]
\label{waveletre}
Let $0<p<\infty$, $0<q<\infty$, $s \in \R$ and 
 \begin{align*}
 u>\max(s, \sigma_{p,q}-s). 	
\end{align*}
Then every orthonormal u-wavelet basis $\Phi$ in $L_2(\R^n)$ according to Definition \ref{u-basis} is an oscillating u-Riesz basis for $\FR$ with a suitable sequence space $\fO[\Z^{\R^n}]$ and
\begin{align*}
 \lambda_r^j(f)= 2^{jn/2} (f,\Phi_r^j).
\end{align*}

\end{Proposition}
\begin{Proof}
This is a reformulation of \cite[Theorem 1.20]{Tri08} now using a more general class of $u$-wavelet systems. The replacement of $\Sc'(\R^n)$ by $D'(\R^n)$ is neglectable since the wavelets $\Phi_r^j$ have compact support. Furthermore, considering a collection
\begin{align*}
 \Z^{\R^n}=\{x_r^j \in \R^n: j \in \N_0, r=1,\ldots,N_j \}
\end{align*}
instead of the usual grids $\{2^{-j}\Z^n,$ $j=0,\ldots\}$ does not make any problems.

Theorem 1.20 in \cite{Tri08} only deals with Daubechies wavelets which are wavelets on $\R^n$ with compact support, see \cite{Dau92}, \cite[Section 4]{Woj97} and \cite[Section 1.2.1]{Tri08}. But the proof did not use the special structure but rather the property of the wavelets to serve as atoms and local means, see Definition \ref{Atoms} and Proposition \ref{RychkovLocal} in \cite{Tri08}. But this holds true for orthonormal u-wavelet bases $\Phi$ in $L_2(\R^n)$ according to Definition \ref{u-basis}.

There is only one slight difference: In Theorem 1.20 of \cite{Tri08} one used classical moment conditions \eqref{Atom31}, but now we assume more general moment conditiions \eqref{Atom3}. This generalization is covered by the underlying proof of Theorem 1.15 of \cite{Tri08}, see also the proof of the upcoming atomic representation Theorem \ref{AtomicRepr}.
\end{Proof}

%% file: Pointwise.tex
\label{chapterpointwise}
The aim of this chapter is to generalize the atomic decomposition theorem from Triebel \cite{Tri92, Tri97} for Besov and Triebel-Lizorkin spaces $\bspq$ and $\fspq$ and to present two applications to pointwise multipliers and diffeomorphisms as continuous linear operators in $\bspq$ resp.\ $\fspq$. A detailed (historical) treatment and references are given in the introduction of this text.

\section[A general atomic representation theorem]{A general atomic representation theorem for Besov and Triebel-Lizorkin function spaces on $\R^n$}
\label{atomsec}
The main theorem of this section is Theorem \ref{AtomicRepr} where we generalize the atomic representation theorem from \cite[Theorem 13.8]{Tri97}.
\subsection{Interaction of atoms and local means}
We start with a lemma which helps us to understand the relation between moment conditions like \eqref{Atom3} and \eqref{Atom31} and which will be heavily used in the proof of the atomic representation theorem. It also shows that local means and atoms are related, see condition \eqref{MeanMomenter}.
\begin{Lemma}
 \label{Momenter}
Let $j \in \N_0$. If $k_0$ and $k_j=2^{jn}k(2^{j}\cdot)$ for $j \in \N$ are local means as in Definition \ref{LocalMeans}, then $2^{-j\left(s+n\left(1-\frac{1}{p}\right)\right)}\cdot k_j$ is an $(s,p)_{K,L}$-atom located at $Q_{j,0}$ for arbitrary $K\geq 0$ and for $L<N$ with $N$ from the moment conditions \eqref{MeanMomenter} of $k$.
\end{Lemma}
\begin{proof}
 For $j=0$ there is nothing to prove since the general moment condition \eqref{Atom3} follows from the support condition \eqref{Atom1} and the boundedness condition included in \eqref{Atom2}. So we can concentrate on $j \in \N$: The support condition \eqref{Atom1} follows from the compact support of $k$ with a suitable $d>0$. Furthermore, for $K>0$ 
\begin{align*}
 \|k_j(2^{-j})|\hold[K]\|=2^{jn} \|k|\hold[K]\|\leq C \, 2^{jn}
\end{align*}
since $K$ is arbitrarily often differentiable. The $L_{\infty}(\R^n)$-condition for $K=0$ follows trivially. Hence, the H\"older condition \eqref{Atom2} is shown.

Now we have to show the general moment condition \eqref{Atom3} for $j\geq 1$ and $L>0$. Hence, we can use the moment conditions \eqref{MeanMomenter}. Let $L=\Lint+\Lrest$ as in Definition \ref{Hoelder} and let $\psi \in \hold[L]$. We expand the $\Lint$-times continuously differentiable function $\psi$ into its Taylor series of order $\Lint-1$. Then there exists a $\theta \in (0,1)$ with
\begin{align*}
 \psi(x)&=\sum_{|\beta|\leq \Lint-1} \frac{1}{\beta!} \, D^{\beta}\psi(0) \cdot x^{\beta}+\sum_{|\beta|= \Lint} \frac{1}{\beta!} \, D^{\beta}\psi(\theta x) \cdot x^{\beta}. 
\end{align*}
Hence
\begin{align*}
 \Big|\psi(x)-\sum_{|\beta|\leq \Lint} \frac{1}{\beta!} \, D^{\beta}\psi(0) \cdot x^{\beta}\Big| &=\Big|\sum_{|\beta|= \Lint} \frac{1}{\beta!} \, \big(D^{\beta}\psi(\theta x)-D^{\beta}\psi(0)\big)x^{\beta}\Big| \\
&\lesssim \|\psi|\hold[L]\| \cdot |x|^{L}.
\end{align*}
Using the moment conditions \eqref{MeanMomenter} for $k_j$ and $\Lint\leq N-1$ we can insert the polynomial terms of order $|\beta|\leq \Lint$ into the integral and get
\begin{align}
\label{IntMomenter}
  \Big|\int_{d \cdot Q_{j,0}} \!\!\psi(x) k_j(x) \ dx \Big| \lesssim \ \|\psi|\hold[L]\| \int_{d \cdot Q_{j,0}} \!\!\!\!|k_j(x)| \cdot |x|^{L} \ dx \lesssim  2^{-jL} \|\psi|\hold[L]\|.
\end{align}
Hence $2^{-j\left(s+n\left(1-\frac{1}{p}\right)\right)}k_j$ fulfils the general moment condition \eqref{Atom3}. The constants in the inequalities do not depend on $j \in \N_0$.
\end{proof}
\begin{Remark}
\label{MomentBem}
 If we take a look at the proof, we see that instead of the moment condition \eqref{MeanMomenter} it suffices to have
\begin{align}
\label{MeanMomenter2}
 \Big|\int_{d \cdot Q_{j,0}} x^{\beta} k_j(x) \ dx \Big| \leq C \cdot 2^{-jL} \text{ if } |\beta|\leq \Lint\text{, i.\,e.\ }|\beta|<N. 
\end{align}
In fact, this condition is equivalent to the general condition \eqref{Atom3} for $k_j$ since $\|x^{\beta} \cdot \psi| \hold[L]\|\leq C$ if $|\beta|\leq \Lint$, where $\psi \in C^{\infty}(\R^n)$ is a cutoff function, i.\,e.\ with compact support and $\psi(x)=1$ for $x \in supp \ k$, hence for $x \in supp \ k_j$, too.
\end{Remark}

Now we will see what happens if an atom is dilated.
\begin{Lemma}
 \label{Dilation}
Let $j,\nu \in \N_0$ and $j\leq \nu$. If $a_{\nu,m}$ is an $(s,p)_{K,L}$-atom located at the cube $Q_{\nu,m}$, then $2^{j(s-\frac{n}{p})}\cdot a_{\nu,m}(2^{-j}\cdot)$ is an $(s,p)_{K,L}$-atom located at $Q_{\nu-j,m}$. 
\end{Lemma}
\begin{proof}
 The support condition \eqref{Atom1} and the H\"older-condition \eqref{Atom2}  are easy to verify. Considering the general moment condition \eqref{Atom3} we have 
\begin{align*}
 \Big|\int_{d \cdot Q_{\nu-j,m}} \psi(x) a_{\nu,m}(2^{-j}x) \ dx \Big|&=2^{jn} \cdot \Big|\int_{d \cdot Q_{\nu,m}} \psi(2^jx) a_{\nu,m}(x) \ dx \Big| \\
& \leq C \cdot 2^{j n} \cdot 2^{-\nu\Kap} \cdot \|\psi(2^j\cdot)| \hold[L]\| \\
& \leq C \cdot 2^{jn} \cdot 2^{-\nu\Kap}  \cdot 2^{jL} \cdot \|\psi|\hold[L]\| \\
& = C \cdot 2^{-(\nu-j)\Kap} \cdot 2^{-j(s-\frac{n}{p})} \cdot \|\psi|\hold[L]\|.
\end{align*}
This is what we wanted to prove.
\end{proof}

\subsection{The proof of the general atomic representation theorem}

Now we come to the essential part - showing the atomic representation theorem. We will use an approach as in Theorem 13.8 of \cite{Tri97}. Using the more general form of the atoms we are able to simplify the proof: One has to estimate
\begin{align*}
 \int k_j (x-y) a_{\nu,m}(y) \ dy,
\end{align*}
where $k_j$ are the local means from Section \ref{LocalMeans} and $a_{\nu,m}$ are atoms located at $Q_{\nu,m}$. One has to distinguish between $j\geq \nu$ and $j<\nu$ as in the original proof - but now  both cases can be proven very similarly with our more general approach of atoms. 

At first we prove the convergence of the atomic series in ${\cal S}'(\R^n)$. 
\begin{Lemma}
\label{HarmS'-KonvAtom}
Let $s \in \R$, $0<p\leq\infty$ resp.\ $0<p<\infty$ and $0<q\leq \infty$. Let $K\geq 0$, $L\geq 0$ with $L> \sigma_p-s$. Then 
\begin{align*}
 \sum_{\nu=0}^{\infty} \sum_{m \in \Z^n} \lambda_{\nu,m} a_{\nu,m}
\end{align*}
converges unconditionally in ${\cal S}'(\R^n)$, where $a_{\nu,m}$ are $(s,p)_{K,L}$-atoms located at $Q_{\nu,m}$ and $\lambda \in b_{p,q}$ or $\lambda \in f_{p,q}$.
\end{Lemma}
\begin{proof}
Let $\varphi \in {\cal S}(\R^n)$. Having in mind the support conditions \eqref{Atom1} and the moment conditions \eqref{Atom3} we obtain
\begin{align*}
 \sum_{m}\left| \int_{\R^n}  \lambda_{\nu,m} a_{\nu,m}(x)  \varphi(x) \ dx\right| \lesssim 2^{-\nu \Kap} \sum_{m} |\lambda_{\nu,m}| \cdot  \|\varphi\cdot \psi(2^{\nu} \cdot -m)|\hold[L]\|,\\ 
\end{align*}
where $\psi \in C^{\infty}(\R^n)$, $\psi(x)=1$ for $x \in d \cdot Q_{0,0}$ and $supp \ \psi \in (d+1) \cdot Q_{0,0}$.

Observing $\Kap=s+L+n\left(1-\frac{1}{p}\right)$ and $L>\sigma_p-s$ we get
\begin{align}
\label{Kap}
 \Kap > \begin{cases}
                0, & 0<p\leq 1 \\
		n\left(1-\frac{1}{p}\right), & 1<p\leq \infty.
               \end{cases}
\end{align}
Furthermore, since $\varphi \in {\cal S}(\R^n)$ we have
\begin{align*}
  \|\varphi\cdot \psi(2^{\nu} \cdot -m)|\hold[L]\| \leq C_M \cdot \left(1+|2^{-\nu}m|\right)^{-M},
\end{align*}
where $M \in \N_0$ is at our disposal and $C_M$ does not depend on $\nu$ and $m$. 

Let at first be $0<p\leq 1$. Then we choose $M=0$ and get
\begin{align*}
 \sum_{m}\left| \int_{\R^n}  \lambda_{\nu,m} a_{\nu,m}(x)  \varphi(x) \ dx\right| \lesssim 2^{-\nu\Kap} \sum_{m} |\lambda_{\nu,m}| \lesssim 2^{-\nu \Kap} \left(\sum_{m} |\lambda_{\nu,m}|^p\right)^{\frac{1}{p}}.
\end{align*}
Summing up over $\nu \in \N_0$ using $\Kap>0$ we finally arrive at
\begin{align}
\label{Summing}
 \sum_{\nu} \sum_{m}\left| \int_{\R^n}  \lambda_{\nu,m} a_{\nu,m}(x)  \varphi(x) \ dx\right| \lesssim \|\lambda|b_{p,\infty}\|. 
\end{align}
In the case $1<p\leq \infty$ we choose $M \in \N_0$ such that $Mp'>n$, where $1=\frac{1}{p}+\frac{1}{p'}$. Using H\"older's inequality we get
\begin{align*}
 \sum_{m}\left| \int_{\R^n}  \lambda_{\nu,m} a_{\nu,m}(x)  \varphi(x) \ dx\right| &\lesssim 2^{-\nu\Kap} \sum_{m} |\lambda_{\nu,m}| \cdot \left(1+|2^{-\nu}m|\right)^{-M} \\
&\lesssim2^{-\nu \Kap} \left(\sum_{m} \left(1+|2^{-\nu}m|\right)^{-Mp'}\right)^{\frac{1}{p'}} \cdot \left(\sum_{m} |\lambda_{\nu,m}|^p\right)^{\frac{1}{p}} \\
&\lesssim 2^{-\nu \Kap} \cdot 2^{\nu \frac{n}{p'}} \cdot \left(\sum_{m} |\lambda_{\nu,m}|^p\right)^{\frac{1}{p}} .
\end{align*}
By \eqref{Kap} the exponent is smaller than zero. Hence summing over $\nu \in \N_0$ gives the same result as in \eqref{Summing}. 

Since 
\begin{align*}
b_{p,q} \hookrightarrow b_{p,\infty} \text{ resp.\ } 
f_{p,q} \hookrightarrow f_{p,\infty} 
\end{align*}
we have shown the absolut and hence unconditional convergence in ${\cal S}'(\R^n)$ for all cases.

\end{proof}

\begin{Theorem}
\label{AtomicRepr}
 (i) Let $s \in \mathbb{R}$, $0<p\leq \infty$ and $0<q\leq \infty$. Let $K,L \in \R$, $K,L\geq 0$, $K>s$ and $L >\sigma_p-s$. Then $f \in {\cal S}'(\mathbb{R}^n)$ belongs to $\bspq$ if and only if it can be represented as
\begin{align*}
 f=\sum_{\nu=0}^{\infty} \sum_{m \in \Z^n}\lambda_{\nu,m}a_{\nu,m} \quad \text{with convergence in } {\cal S}'(\mathbb{R}^n).
\end{align*} 
Here $a_{\nu,m}$ are $(s,p)_{K,L}$-atoms located at $Q_{\nu,m}$ (with the same constants $d>1$ and $C>0$ in Definition \ref{Atoms} for all $\nu \in \N_0, m \in \Z$) and $\|\lambda|b_{p,q}\| < \infty$ . Furthermore, we have in the sense of equivalence of norms
\begin{align*}
 \|f|\bspq\| \sim \inf \, \|\lambda|b_{p,q}\|,
\end{align*}
where the infimum on the right-hand side is taken over all admissible representations of $f$.

(ii) Let $s \in \mathbb{R}$, $0<p<\infty$ and $0<q\leq \infty$. Let $K,L \in \R$, $K,L\geq 0$, $K>s$ and $L >\sigma_{p,q}-s$. Then $f \in {\cal S}'(\mathbb{R}^n)$ belongs to $\fspq$ if and only if it can be represented as
\begin{align*}
 f=\sum_{\nu=0}^{\infty} \sum_{m \in \Z^n}\lambda_{\nu,m}a_{\nu,m} \quad \text{with convergence in } {\cal S}'(\mathbb{R}^n).
\end{align*}
Here $a_{\nu,m}$ are $(s,p)_{K,L}$-atoms located at $Q_{\nu,m}$ (with the same constants $d>1$ and $C>0$ in Definition \ref{Atoms} for all $\nu \in \N_0, m \in \Z$) and $\|\lambda|f_{p,q}\| < \infty$. Furthermore, we have in the sense of equivalence of norms
\begin{align*}
 \|f|\fspq\| \sim \inf \, \|\lambda|f_{p,q}\|,
\end{align*}
where the infimum on the right-hand side is taken over all admissible representations of $f$.
\end{Theorem}
\begin{proof}
 We rely on the proof of Theorem 13.8 of \cite{Tri97}, now modified keeping in mind the more general H\"older conditions \eqref{Atom2} and moment conditions \eqref{Atom3} instead of \eqref{Atom21} and \eqref{Atom31}. There are two directions we have to prove. 

At first, let us assume that $f$ from $\bspq$ or $\fspq$ is given. Then we know from Theorem 13.8 of \cite{Tri97} that $f$ can be written as an atomic decomposition, with atoms now fulfilling the standard conditions \eqref{Atom21} and \eqref{Atom31} for given natural numbers $K'>s$ and $L'+1>\sigma_{p}-s$ resp.\ $L'+1>\sigma_{p,q}-s$. Hence, because of $C^{K'}(\R^n) \subset \hold[K']$, condition \eqref{Atom2} is fulfilled for all $K\leq K'$.

The general moment conditions \eqref{Atom3} are generalizations of the classical moment conditions \eqref{Atom31} and are ordered in $L$, see Remark \ref{Ordered}. 

Thus, every classical $(s,p)_{K',L'}$-atom is an $(s,p)_{K,L}$ atom in the sense of definition \ref{Atoms} for $K\leq K'$ and $L\leq L'+1$ and this immediately shows that we find a decomposition of $f$ from $\bspq$ or $\fspq$ for arbitrary $K$ and $L$ in terms of the general atoms we introduced.

Now we come to the essential part of the proof. We have to show that, although we weakened the conditions on the atoms, a linear combination of atoms is still an element of $\bspq$ resp.\ $\fspq$. We modify the proof of Theorem 13.8 of \cite{Tri97} (or Proposition 4.7 of  \cite{Sch10} where some minor technical details are modified - for the more general vector-valued case). There one uses the equivalent characterization by local means $k_0, k_j:=2^{jn} k(2^{j}\cdot)$ with a suitably large $N$ (see Proposition \ref{RychkovLocal}) and distinguishes between the cases $j\geq \nu$ and $j<\nu$. In both cases the crucial part is the estimate of
\begin{align*}
 \int k_j (x-y) a_{\nu,m}(y) \ dy,
\end{align*}
where $a_{\nu,m}$ is an $(s,p)_{K,L}$-atom centered at $Q_{\nu,m}$.
The idea now is to use that not only $a_{\nu,m}$ but also $k_j$ can be interpreted as atoms and admit derivative resp.\ moment conditions as in \eqref{Atom2} resp.\ \eqref{Atom3}. This follows from Lemma \ref{Momenter}.

Let at first be $j\geq \nu$. The function $k$ has compact support and fulfils moment conditions \eqref{Atom31}. At first we transform the integral, having in mind the form of the derivative condition \eqref{Atom2} of $a_{\nu,m}$,
\begin{align*}
 2^{js} \int k_j (y) a_{\nu,m}(x-y) \ dy= 2^{js} \int k_{j-\nu}(y)a_{\nu,m}(x-2^{-\nu}y) \ dy. 
\end{align*}
Surely, this integral vanishes for $x \notin c \cdot Q_{\nu,m}$ for a suitable $c>0$ because of $j\geq \nu$. So we concentrate on $x\in c \cdot Q_{\nu,m}$: By Lemmata \ref{Momenter} and \ref{Dilation} the function 
\begin{align*}
2^{-(j-\nu)\left(s+n\left(1-\frac{1}{p}\right)\right)} \cdot k_{j-\nu}=2^{-(j-\nu)\left(s+n\left(1-\frac{1}{p}\right)\right)} \cdot 2^{-\nu n} \cdot k_{j}(2^{-\nu}\cdot)
\end{align*}
is an $(s,p)_{M,N}$-atom located at $Q_{j-\nu,0}$ for $M$ arbitrarily large and $N$ from \eqref{MeanMomenter}, so that also $N$ may be arbitrarily large, but fixed. Now we will use the moment condition \eqref{Atom3} for $k_{j-\nu}$ and the H\"older condition \eqref{Atom2} for $a_{\nu,m}$. Hence, with $\psi(y)=a_{\nu,m}(x-2^{-\nu}y)$ and $N \geq K$ we have
\begin{align*}
 2^{js} \Big| \int k_{j-\nu}(y) a_{\nu,m}(x-2^{-\nu}y) \ dy \Big| &\leq C \cdot 2^{js} \cdot 2^{-(j-\nu)K} \cdot \|a_{\nu,m}(x-2^{-\nu}\cdot)|\hold[K]\| \\
&= C \cdot 2^{js} \cdot 2^{-(j-\nu)K} \cdot \|a_{\nu,m}(2^{-\nu}\cdot)|\hold[K]\| \\
&\leq C \cdot 2^{js} \cdot 2^{-(j-\nu)K} \cdot 2^{-\nu(s-\frac{n}{p})} \\
&= C' \cdot 2^{-(j-\nu)(K-s)} \cdot \chi^{(p)}(c\cdot Q_{\nu,m}),
\end{align*}
where $\chi^{(p)}(c\cdot Q_{\nu,m})$ is the $L_p(\R^n)$-normalized characteristic function of $c \cdot Q_{\nu,m}$. This inequality is certainly also true for $x \notin c \cdot Q_{\nu,m}$. Hence (13.37) in \cite{Tri97} is shown.

Now let $j<\nu$. We will interchange the roles of $k_j$ and $a_{\nu,m}$ using the atomic condition \eqref{Atom2} now for $k_j$ and \eqref{Atom3} for $a_{\nu,m}$. Hence we start with
\begin{align*}
 2^{js} \int k_j (x-y) a_{\nu,m}(y) \ dy =2^{js} \int k(2^{j}x-y)a_{\nu,m}(2^{-j}y) \ dy.
\end{align*}
Surely, this integral vanishes for $x \notin c\cdot 2^{\nu-j}\cdot Q_{\nu,m}$. So we concentrate on $x\in c \cdot 2^{\nu-j}\cdot Q_{\nu,m}$: By Lemma \ref{Dilation} we know that $2^{j(s-\frac{n}{p})} \cdot a_{\nu,m}(2^{-j}\cdot)$ is an $(s,p)_{K,L}$-atom located at $Q_{\nu-j,m}$ while $k$ is an $(s,p)_{M,N}$-atom located at $Q_{0,0}$. Thus, using \eqref{Atom2} for $k$ with $M\geq L$, we get
\begin{align*}
 2^{js} \Big|\int k(2^{j}x-y)a_{\nu,m}(2^{-j}y) \ dy \Big| &\leq c \cdot 2^{js} \cdot 2^{-(\nu-j)\Kap} \cdot 2^{-j(s-\frac{n}{p})} \cdot \|k(2^{j}x-\cdot)|\hold[L]\| \\
&\leq c \cdot 2^{-(\nu-j)(L+s)} \cdot 2^{\nu\frac{n}{p}} \cdot 2^{-(\nu-j)n}  \\
&= c \cdot 2^{-(\nu-j)(L+s)} \cdot 2^{\nu\frac{n}{p}} \cdot 2^{-(\nu-j)n}\cdot \chi(c \cdot 2^{\nu-j}\cdot Q_{\nu,m}).
\end{align*}
where $\chi(c \cdot 2^{\nu-j}\cdot Q_{\nu,m})$ is the characteristic function of $c \cdot 2^{\nu-j}\cdot Q_{\nu,m}$. This estimate is the same as (13.41) combined with (13.42) in \cite{Tri97} or (72) and (73) in \cite{Sch10}, observing that we use $L$ instead of $L+1$ in the atomic representation theorem. 

Starting with these two estimates we can follow the steps in \cite{Tri97} or \cite{Sch10} and finish the proof, since $K>s$ and $L>\sigma_p-s$ resp.\ $L>\sigma_{p,q}-s$. To be precise, we arrive (in the $\bspq$-case) at
\begin{align*}
 \Big\|\sum_{\nu\leq \nu_0} \sum_{|m|\leq m_0}\lambda_{\nu,m} a_{\nu,m} \big|\bspq\Big\|\leq C \cdot \|\lambda|b_{p,q}\|
\end{align*}
for all $\nu_0,m_0 \in \N_0$ with a constant $C$ independent of $\nu_0$ and $m_0$. Using Lemma \ref{HarmS'-KonvAtom} (the convergence of the atomic series in $\in {\cal S}'(\R^n)$) and the Fatou property of the spaces $\bspq$ resp.\ $\fspq$ (see Proposition \ref{Fatou}) we are finally done, i.\,e.\
\begin{align*}
 \Big\|\sum_{\nu\in \N_0} \sum_{m \in \Z^n}\lambda_{\nu,m} a_{\nu,m} \big|\bspq\Big\|\leq C \cdot \|\lambda|b_{p,q}\|.
\end{align*}
\end{proof}

\begin{Remark}
 The atomic conditions $\eqref{Atom2}$ and $\eqref{Atom3}$ for the atomic representation theorem can be slightly modified: If $K>0$, then it is possible to replace $\|\cdot|\hold[K]\|$ by $\|\cdot|B_{\infty,\infty}^{K}(\R^n)\|$ in condition $\eqref{Atom2}$. This is clear for $K \notin \N$, see Remark \ref{HoldB}. If $K \in \N$, this follows from 
\begin{align*}
\hold[K]  \hookrightarrow B_{\infty,\infty}^{K}(\R^n) \hookrightarrow \hold[K-\varepsilon]
\end{align*}
for $\varepsilon>0$. 

A similar result holds true for $L>0$, $L \notin \N$ and condition \eqref{Atom3} by trivial means. If $L \in \N$, then $\|\cdot|\hold[L]\|$ can be replaced by $\|\cdot|C^L(\R^n)\|$, where the condition needs to be true for all $\psi \in C^L(\R^n)$. This follows from the fact, that both conditions imply the polynomial condition \eqref{Atom32}. Hence they are equivalent. 

It is not clear to the author whether $\|\cdot|\hold[L]\|$ can be replaced by $\|\cdot|B_{\infty,\infty}^{L}(\R^n)\|$ for $L \in \N$. Since $\hold[L] \subsetneq B_{\infty,\infty}^{L}(\R^n)$ (for example see \cite{Zyg45}) this would be a stronger condition.

\end{Remark}

\subsection{A local mean theorem as a corollary}
 In the proof of Theorem \ref{AtomicRepr} we assumed that the local means $k_j$ are arbitrarily often differentiable and fulfil as many moment conditions as we wanted. But if we take a look into the proof, we see that we did not use the specific structure $k_j=2^{jn}k(2^j\cdot)$. It is sufficient to know that there are constants $c$ and $C$ such that for all $j\in \N_0$ it holds $supp \ k_j \subset c\cdot Q_{j,0}$, that
\begin{align}
 \label{LocalMean2}
\|k_j(2^{-j}\cdot)|\hold[M]\| &\leq C \cdot  2^{jn}
\end{align}
with $M \geq L$ and that for every $\psi \in \hold[N]$ it holds 	
\begin{align}
\label{LocalMean3} \left| \int_{d \cdot Q_{j,0}} \psi(x) k_j(x) \ dx \right|\leq C \cdot  2^{-jN} \cdot \|\psi|\hold[N]\|
\end{align}
with $N \geq K$ because the atomic conditions \eqref{Atom2} and \eqref{Atom3} are ordered in $N$ and $M$, see Remark \ref{Ordered}. As before, condition \eqref{LocalMean3} can be strengthened by
\begin{align*}
 \int x^{\beta}k_j(x) \ dx = 0 \text{ for all } |\beta|< N.
\end{align*}
Through these considerations the idea arises how to prove a counterpart of Theorem \ref{AtomicRepr}  for the local mean characterization in \cite[Theorem 1.15]{Tri08} without further substantial efforts. This is done in the following corollary, including some technical issues concerning the definition of a dual pairing (see \cite[Remark 1.14]{Tri08}). It will be obvious that the original version of Theorem 1.15 in \cite{Tri08} is just some kind of modification of this corollary. 
\begin{Corollary}
\label{FolgLocal}
(i) Let $0<p\leq \infty$, $0<q\leq \infty$ and $s \in \mathbb{R}$. Let $M,N\in \R$, $M,N\geq 0$, $M>\sigma_p-s$ and $N>s$. Assume that for all $j \in \N_0$ it holds that $k_j \in \hold[M]$, $supp \ k_j \subset c\cdot Q_{j,0}$ and $k_j$ fulfils  \eqref{LocalMean2} and \eqref{LocalMean3}. Then there is a constant c such that
\begin{align*}
  \|f|\bspq\|_k:=\|k_0*f|L_p(\R^n)\|+\left(\sum_{j=1}^{\infty} 2^{jsq}\|k_j*f|L_p(\R^n)\|^q\right)^{\frac{1}{q}} \leq c \cdot \|f|\bspq\|
\end{align*}
(modified for $q=\infty$) for all $f \in \bspq$.

(ii) Let $0<p< \infty$, $0<q\leq \infty$ and $s \in \mathbb{R}$. Let $M,N\in \R$, $M,N\geq 0$, $M>\sigma_{p,q}-s$ and $N>s$. Assume that for all $j \in \N_0$ it holds that  $k_j \in \hold[M]$, $supp \ k_j \subset c\cdot Q_{j,0}$ and $k_j$ fulfils \eqref{LocalMean2} and \eqref{LocalMean3}. Then there is a constant c such that
\begin{align*}
  \|f|\fspq\|_k:&=\|k_0*f|L_p(\R^n)\|+\bigg\|\Big(\sum_{j=1}^{\infty} 2^{jsq} \left|(k_j*f)(\cdot)\right|^q \Big)^{\frac{1}{q}} \Big|L_p(\R^n)\bigg\| \\
 &\leq c \cdot \|f|\fspq\|
\end{align*}
(modified for $q=\infty$) for all $f \in \fspq$.
\end{Corollary}
\begin{proof}
 There is nearly nothing left to prove because the crucial steps were done in the proof before: Let $f \in \bspq$ (analogously for $f \in \fspq$) be given. By Theorem \ref{AtomicRepr} we can represent $f \in \bspq$ by an ''optimal'' atomic decomposition 
\begin{align*}
 f=\sum_{\nu=0}^{\infty} \sum_{m \in \mathbb{Z}^n}\lambda_{\nu,m}a_{\nu,m}, 
\end{align*}
where $a_{\nu,m}$ is an $(s,p)_{N,M}$-atom located at $Q_{\nu,m}$ and $\|f|\bspq\| \sim \|\lambda|b_{p,q}\|$ (with constants independent of $f$).

But, by the second step of the proof of Theorem \ref{AtomicRepr} and the considerations in the succeeding remark we have

 \begin{align}
\label{AbschAtom}
 \Big\|\sum_{\nu\leq \nu_0} \sum_{|m|\leq m_0}\lambda_{\nu,m} a_{\nu,m} \big|\bspq \Big\|_k\leq C \cdot \|\lambda|b_{p,q}\| \sim \|f|\bspq\|
\end{align}
for all $\nu_0,m_0 \in \N_0$ with a constant $C$ independent of $\nu_0$ and $m_0$. 

Finally, we use a similar duality argument as in \cite[Remark 1.14]{Tri08} or \cite[Section 5.1.7]{Tri06} to justify the dual pairing of $k_j$ and $f$. Looking into the proof of Lemma \ref{HarmS'-KonvAtom}, we see that
\begin{align}
\label{KonvB}
 \sum_{\nu} \sum_{m}\left| \int_{\R^n}  \lambda_{\nu,m} a_{\nu,m}(x) \varphi(x) \ dx\right| \leq C' \cdot \|\varphi|\hold[M-\varepsilon]\|\cdot \|\lambda|b_{p,\infty}\|
\end{align}
for $\varphi \in \hold[M]$ with compact support, $M-\varepsilon\geq 0$ and $M-\varepsilon>\sigma_p-s$, where $C'$ depends on the support of $\varphi$. This includes the functions $k_j$ for $j \in \N_0$. Because of this absolut convergence the dual pairing of $f$ and $\varphi$ is given by
\begin{align*}
 \lim_{m_0,\nu_0 \rightarrow \infty} \, \sum_{\nu\leq \nu_0} \sum_{|m|\leq m_0 } \int_{\R^n}  \lambda_{\nu,m} a_{\nu,m}(x) \varphi(x) \ dx .
\end{align*}

Furthermore, for two different atomic decompositions of $f$ these limits are the same: By definition of a distribution $f \in {\cal S}'(\R^n)$ and Lemma \ref{HarmS'-KonvAtom} this is valid for $\varphi \in {\cal S}(\R^n)$. For arbitrary $\varphi \in \hold[M]$ with compact support this follows by \eqref{KonvB} and density arguments because $C^{\infty}(\R^n)$ is dense in $\hold[M]$ with respect to the norm of $\hold[M-\varepsilon]$. For instance, this can be seen using
\begin{align*}
 \hold[M] \hookrightarrow B_{\infty,\infty}^M(\R^n) \hookrightarrow B_{\infty,q}^{M-\varepsilon}(\R^n) \hookrightarrow  B_{\infty,\infty}^{M-\varepsilon}=\hold[M-\varepsilon]
\end{align*}
for $M-\varepsilon \notin \N_0$ and the fact that $C^{\infty}(\R^n)$ is dense in $\bspq$ if $q<\infty$. 

Hence we have 
\begin{align*}
 \sum_{\nu\leq \nu_0} \sum_{|m|\leq m_0}\lambda_{\nu,m} \left(a_{\nu,m} * k_j\right)(x) \rightarrow (f * k_j) (x) \text{ for } \nu_0, m_0 \rightarrow \infty 
\end{align*}
for all $x \in \R^n$. Using the standard Fatou lemma and \eqref{AbschAtom} we finally get 
\begin{align*}
 \|f|\bspq\|_k\leq C \cdot \|\lambda|b_{p,q}\| \sim \|f|\bspq\|.
\end{align*}
\end{proof}

\section{Pointwise multipliers}
\label{PointwiseSect}
Triebel proved in \cite[Section 4.2]{Tri92} the following assertion.
\begin{Theorem}
\label{pointwise}
 Let $s \in \mathbb{R}$ and $0<q\leq \infty$.

(i) Let $0<p\leq \infty$ and $\rho > \max(s,\sigma_p-s)$. Then there exists a positive number $c$ such that
\begin{align*}
 \|\varphi f|\bspq\| \leq c \|\varphi|\hold[\rho]\| \cdot \|f|\bspq\|
\end{align*}
for all $\varphi \in \hold[\rho]$ and all $f \in \bspq$.

(ii) Let $0<p<\infty$ and $\rho > \max(s,\sigma_{p,q}-s)$. Then there exists a positive number $c$ such that
\begin{align*}
 \|\varphi f|\fspq\| \leq c \|\varphi|\hold[\rho]\| \cdot \|f|\fspq\|
\end{align*}
for all $\varphi \in \hold[\rho]$ and all $f \in \fspq$.
\end{Theorem}
He excluded the cases $\rho \in \N$. This is not necessary in our considerations.

The very first idea to prove this result is to take an atomic decomposition of $f$, to multiply it by $\varphi$ and to prove that the resulting sum is again a sum of atoms. Hence one has to check whether a product of an $(s,p)_{K,L}$-atom and a function $\varphi$ is still an $(s,p)_{K,L}$-atom. 

But there was a problem: Moment conditions like $\eqref{Atom31}$ are (in general) destroyed by multiplication with $\varphi$. So the atomic approach in \cite{Tri92} only worked when no moment conditions were required, hence if $s>\sigma_p$ resp.\ $s>\sigma_{p,q}$, and the full generality of Theorem \ref{pointwise} had to be obtained by an approach via local means. Looking at the more general moment condition $\eqref{Atom3}$ instead the situation when multiplying by $\varphi$ is now different. 

Furthermore, the atomic approach only worked for $\varphi \in C^{k}(\R^n)$ with $k \in \N$ and $k>s$ having in mind the derivative condition \eqref{Atom21}. Now we are able to give a new proof based on our more general version of atoms, using the H\"older condition \eqref{Atom2}.

We start with a first standard analytical observation.
\begin{Lemma}
\label{helpHoelder}
 Let $s>0$. There exists a constant $c>0$ such that for all $f,g \in \hold$ the product $f\cdot g$ belongs to $\hold$ and it holds
\begin{align*}
  \|f\cdot g|\hold\| \leq  c \cdot \|f|\hold\| \cdot  \|g|\hold\|. 
\end{align*}
The same result holds true for $L_{\infty}(\R^n)$ instead of $\hold$.
\end{Lemma}
\begin{proof}
This can be proven using standard arguments, in particular Leibniz' formula.
\end{proof}
Now we are ready to prove Theorem \ref{pointwise}. This is done by the following lemma together with the atomic representation Theorem \ref{AtomicRepr} using the mentioned technique of atomic decompositions. For some further technicalities see the upcoming Remark \ref{ProdDef} or \cite[4.2.2, Remark 1] {Tri92}. This covers also the well-definedness of the product.

\begin{Lemma}
\label{ProdAtom}
 There exists a constant $c$ with the following property: For all $\nu \in \N_0$, $m \in \Z$, all $(s,p)_{K,L}$-atoms $a_{\nu,m}$ with support in $d \cdot Q_{\nu,m}$ and all $\varphi \in \hold[\rho]$ with $\rho \geq \max(K,L)$ the product
\begin{align*}
 c \cdot \|\varphi|\hold[\rho]\|^{-1} \cdot \varphi \cdot a_{\nu,m} 
\end{align*}
is an $(s,p)_{K,L}$-atom with support in $d\cdot Q_{\nu,m}$.
\end{Lemma}
\begin{proof}
Regarding the H\"older conditions \eqref{Atom2} Lemma \ref{helpHoelder} gives (for $K>0$)
\begin{align*}
 \|(\varphi \cdot a)(2^{-\nu}\cdot)|\hold[K]\| &\leq  c \cdot \|\varphi(2^{-\nu}\cdot)|\hold[K]\| \cdot \|a(2^{-\nu}\cdot)|\hold[K]\|  \\
&\leq c' \cdot \|\varphi|\hold[K]\| \cdot 2^{-\nu(s-\frac{n}{p})}. 
\end{align*}
Now we come to the preservation of the general moment conditions \eqref{Atom3}. As before we can assume $L>0$. By our assumptions there exists a constant $C>0$ such that for every $\psi \in \hold[L]$ it holds 	
\begin{align*}
\left| \int_{d \cdot Q_{\nu,m}} \psi(x) a(x) \ dx \right|\leq C \cdot  2^{-\nu\Kap} \|\psi|\hold[L]\|.
\end{align*}
Using this inequality now for $\psi \cdot \varphi$ instead of $\psi$ together with Lemma \ref{helpHoelder} it follows
\begin{align*}
\left| \int_{d \cdot Q_{\nu,m}} \psi(x) \big(\varphi(x)\cdot a(x)\big) \ dx \right|&=\left| \int_{d \cdot Q_{\nu,m}} \big(\psi(x) \cdot \varphi(x)\big)a(x) \ dx \right| \\
&\leq C \cdot 2^{-\nu\Kap} \|\psi\cdot\varphi|\hold[L]\| \\
&\leq C' \cdot 2^{-\nu\Kap} \|\psi|\hold[L]\| \cdot \|\varphi|\hold[L]\|.
\end{align*}
Hence our lemma is shown.
\end{proof}
\begin{Remark}
 This is the more general version of part 1 of Lemma 1 in \cite{Skr98} using now the wider atomic approach from Definition \ref{Atoms} which yields a stronger result than in \cite{Skr98}.  
\end{Remark}

\begin{Remark}
 \label{ProdDef}
As at the end of Corollary \ref{FolgLocal} we have to deal with some technicalities. We concentrate on the $\bspq$-case, the $\fspq$-case is nearly the same. In principle, Lemma \ref{ProdAtom} together with Lemma \ref{HarmS'-KonvAtom} show that
\begin{align}
\label{AtomKonv}
 \sum_{\nu} \sum_{m}\lambda_{\nu,m} (a_{\nu,m} \cdot \varphi) 
\end{align}
converges unconditionally in ${\cal S}'(\R^n)$ where
\begin{align*}
 f=\sum_{\nu} \sum_{m}\lambda_{\nu,m} a_{\nu,m} \text{ in } {\cal S}'(\R^n)
\end{align*}
and the limit belongs to $\bspq$ if $f$ belongs to $\bspq$. 

To define the product of $\varphi$ and $f$ as this limit, we have to show that the limit does not depend on the atomic decomposition we chose for $f$. 

Hence we are pretty much in the same situation as at the end of Corollary \ref{FolgLocal}: Let at first be $\varphi \in C^{\infty}(\R^n)$. Then the multiplication with $\varphi$ is a continuous operator mapping  ${\cal S}'(\R^n)$ to ${\cal S}'(\R^n)$. So \eqref{AtomKonv} converges to $\varphi \cdot f$ for all choices of atomic decompositions of $f$. Using Lemma \ref{ProdAtom} and the Fatou property \ref{Fatou} of $\bspq$ we get
\begin{align*}
 \|\varphi \cdot f|\bspq\| \leq c \cdot \|\varphi|\hold[\rho]\| \cdot \|f|\bspq\|
\end{align*}
for all $f \in \bspq$ and $\varphi \in C^{\infty}(\R^n)$. 

For arbitrary $\varphi \in \hold[\rho]$ we use a density argument similar to that at the end of Corollary \ref{FolgLocal}. We know 
\begin{align*}
 \|\varphi^* \cdot f|\bspq\| \leq c \cdot \|\varphi^*|\hold[\rho-\varepsilon]\| \cdot \|f|\bspq\|
\end{align*}
for $\varphi^* \in C^{\infty}(\R^n)$, $\rho$ as in Lemma \ref{ProdAtom} and $\varepsilon$ small enough. Now using the density of $C^{\infty}(\R^n)$ in $\hold[\rho]$ with respect to the norm of $\hold[\rho-\varepsilon]$ the uniqueness of the product and 
\begin{align*}
 \|\varphi \cdot f|\bspq\| \leq c \cdot \|\varphi|\hold[\rho-\varepsilon]\| \cdot \|f|\bspq\| \leq  c \cdot \|\varphi|\hold[\rho]\| \cdot \|f|\bspq\|
\end{align*}
follows.

\end{Remark}

\begin{Remark}
 Since   
\begin{align*}
 \hold[L] \hookrightarrow B_{\infty,\infty}^{L}(\R^n) \hookrightarrow \hold[L-\varepsilon]
\end{align*}
for $L-\varepsilon \geq 0$, we can replace $\|\varphi|\hold[\rho]\|$ by $\|\varphi|B_{\infty,\infty}^{\rho}(\R^n)\|$, even by $\|\varphi|B_{\infty,q}^{\rho}(\R^n)\|$ for arbitrary $0<q \leq \infty$, in Lemma \ref{pointwise}.

The condition $\rho > \max(s,\sigma_{p,q}-s)$ for the $\fspq$-spaces in Theorem \ref{pointwise} can be replaced by $\rho > \max(s,\sigma_{p}-s)$. This is a matter of complex interpolation, see the proof of the corollary in Section 4.2.2 of \cite{Tri92}.

\end{Remark}

\begin{Remark} 
Our Theorem \ref{pointwise} is a special case of Theorem 4.7.1 in \cite{RS96}: By Remark \ref{HoldB} it holds $\hold[\rho]=B_{\infty,\infty}^{\rho}(\R^n)$ for $\rho>0$ and $\rho \notin \N$. So, let $f \in \bspq$ or $f \in \fspq$ as well as $\varphi \in \hold[\rho]$ with $\rho>s$. Then $\varphi \in B_{\infty,\infty}^{\rho'}(\R^n)$ for $s<\rho'<\rho$. By Theorem 4.7.1 
of \cite{RS96} it holds 
\begin{align*}
 \bspq \cdot B_{\infty,\infty}^{\rho'}(\R^n) \hookrightarrow \bspq \quad \text{resp.}\quad \fspq \cdot B_{\infty,\infty}^{\rho'}(\R^n) \hookrightarrow \fspq
\end{align*}
if
\begin{align*}
 \rho'>s \quad \text{and} \quad s+\rho'> \sigma_p \quad \Leftrightarrow \quad \rho'>s \quad \text{and} \quad \rho'>\sigma_p-s.
\end{align*}
In case of $\bspq$ these are the same conditions as in Theorem \ref{pointwise} - in case of $\fspq$ these are even better (no dependency on $q$).

It was not the idea to give such a detailed and comprehensive treatment as in the book \cite{RS96} by Runst and Sickel but to show an application of the more general atomic decompositions where the proof is easy to follow (see Triebel \cite[Section 4.1]{Tri92}). 
\end{Remark}

\section{Diffeomorphisms}
We want to study the behaviour of the mapping
\begin{align*}
 D_{\varphi}: f \mapsto f(\varphi(\cdot)), 
\end{align*}
where $f$ is an element of the function space $\bspq$ resp.\ $\fspq$ and $\varphi: \R^n \rightarrow \R^n$ is a suitably smooth map. 

One would like to deal with this problem analogously to the pointwise multiplier problem in Section \ref{PointwiseSect}. Hence we start with an atomic decomposition of $f$ and compose with $\varphi$. Then we are confronted with functions of the form $a_{\nu,m}\circ \varphi$ originating from the atoms $a_{\nu,m}$. This was the idea of \cite[Section 4.3.1]{Tri92}. But in general, moment conditions of type \eqref{Atom31} are destroyed by this operator. So $s>\sigma_p$ resp.\ $s>\sigma_{p,q}$ was necessary. As we will see, general moment conditions like \eqref{Atom3} behave more friendly under diffeomorphisms. 
 
Furthermore, we are confronted with more difficulties than in Section \ref{PointwiseSect} because the support of an atom changes remarkably. In particular, after composing with $\varphi$ two or more atoms can be associated with the same cube $Q_{\nu,m}$ which is not possible in the atomic representation Theorem \ref{AtomicRepr}. This has not been considered in detail in \cite[Section 4.3.1]{Tri92} while there is some work done in \cite[Lemma 3]{Skr98}.

The special case of bi-Lipschitzian maps, also called Lipschitz diffeomorphisms, is treated in \cite[Section 4.3]{Tri02}. The main theorem there is used to obtain results for characteristic functions of Lipschitz domains as pointwise multipliers in $\bspq$ and $\fspq$.

\subsection{Introduction of $\rho$-diffeomorphisms and basic properties}

\begin{Definition}
\label{LipDef}
Let $\rho\geq 1$. 

(i) Let $\rho=1$. We say that the map $\varphi: \R^n \rightarrow \R^n$ is a $\rho$-diffeomorphism if $\varphi$ is a bi-Lipschitzian map, i.\,e.\ that there are constants $c_1,c_2>0$ such that
\begin{align}
 \label{biLip}
 c_1\leq \frac{|\varphi(x)-\varphi(y)|}{|x-y|} \leq c_2. 
\end{align}
for all $x,y \in \R^n$ with $0<|x-y|\leq 1$. 

(ii) Let $\rho>1$. We say that the one-to-one map $\varphi: \R^n \rightarrow \R^n$ is a $\rho$-diffeomorphism if the components $\varphi_i$ of $\varphi(x)=(\varphi_1(x),\ldots,\varphi_n(x))$ have classical derivatives up to order $\rint$ with $\frac{\partial \varphi_i}{\partial x_j} \ \in \hold[\rho-1]$ for all $i,j\in \{1,\ldots,n\}$ and if $|\det J(\varphi)(x)| \geq c$ for some $c>0$ and all $x \in \R^n$. Here $J(\varphi)(x)$ stands for the Jacobian matrix of $\varphi$ at the point $x \in \R^n$. 
\end{Definition}
\begin{Remark}
 It does not matter, whether we assume $\eqref{biLip}$ for all $x,y \in \R^n$ with $x\neq y$ or for all $x,y \in \R^n$ with $0<|x-y|<c$ for a constant $c>0$. This is obvious for the upper bound. For the lower bound we have to use the upper bound of the bi-Lipschitzian property of the inverse $\varphi^{-1}$ of $\varphi$. Its existence independent of the given exact definition of a bi-Lipschitzian map is shown in the following lemma.
\end{Remark}

\begin{Lemma}
\label{DiffRem}
 Let $\rho\geq 1$.

(i) If $\varphi$ is a $1$-diffeomorphism, then $\varphi$ is bijective and $\varphi^{-1}$ is a $1$-diffeomorphism, too.

(ii) Let $\rho>1$. If $\varphi$ is a $\rho$-diffeomorphism, then its inverse $\varphi^{-1}$ is a $\rho$-dif\-feo\-mor\-phism as well. 

(iii) If $\varphi$ is a $\rho$-diffeomorphism, then $\varphi$ is a $\rho'$-diffeomorphism for $1\leq \rho' \leq \rho$. Hence $\varphi$ is a bi-Lipschitzian map.
\end{Lemma}
\begin{proof}
To prove part (i) we use Brouwer's invariance of domain theorem (see \cite{Bro12}): Let $U$ be an open set in $\R^n$. Since $\varphi: \R^n \rightarrow \R^n$ is continuous and injective, the image $\varphi(U)$ of $U$ is also an open set by this theorem.

On the other hand, if $U$ is a closed set, then also $\varphi(U)$ is closed: If $\varphi(x_n) \rightarrow y$ with $x_n \in U$, then $x_n$ converges to some $x \in U$ by \eqref{biLip} and hence $\varphi(x_n) \rightarrow \varphi(x)=y$. Thus $\varphi$ maps $\R^n$ to $\R^n$. The inverse $\varphi^{-1}$ is automatically a bi-Lipschitzian map, see \eqref{biLip}.

The proof of observation (iii) for $\rho'>1$ is trivial. Hence, we have to show that every $\rho$-diffeomorphism is a bi-Lipschitzian map for $\rho>1$. The estimate
\begin{align*}
 \frac{|\varphi(x)-\varphi(y)|}{|x-y|} \leq c_2
\end{align*}
follows from the fact that the derivatives $\frac{\partial \varphi_i}{\partial x_j}$ are bounded for all $i,j\in \{1,\ldots,n\}$. The formula
\begin{align}
\label{Det}
 J(\varphi^{-1})(\varphi(x))=\left(J(\varphi)(x)\right)^{-1}
\end{align}
and $|\det J(\varphi)(x)| \geq c$ together show that the derivatives of the inverse $\frac{\partial (\varphi^{-1})_i}{\partial x_j}$ are bounded for all $i,j\in \{1,\ldots,n\}$, for instance using the adjugate matrix formula. By the mean value theorem there exists a $c'>0$ such that
\begin{align*}
 \frac{|\varphi^{-1}(x)-\varphi^{-1}(y)|}{|x-y|} \leq c'
\end{align*}
and so part (iii) is shown.

Finally, for (ii) we have to show that $\frac{\partial (\varphi^{-1})_i}{\partial x_j} \ \in \hold[\rho-1]$ and $|\det J(\varphi^{-1})(x)| \geq c$ for $\rho>1$. The latter part follows from $\eqref{Det}$ and the boundedness of $\frac{\partial \varphi_i}{\partial x_j}$. For the first we have to argue inductively in the same way as in the inverse function theorem, starting with
\begin{align*}
 J(\varphi^{-1})(x)=\left(J(\varphi)(\varphi^{-1}(x))\right)^{-1}.
\end{align*}
It is well known that
\begin{align*}
 A \rightarrow A^{-1}
\end{align*}
is a $C^{\infty}(\R^{n\times n})$-mapping for invertible $A$. Together with the upcoming Lemma \ref{HoelderDiff} this shows: If the components of $J(\varphi)$ belong to $\hold[\rho-1]$ and $\varphi^{-1}$ is  an  $l$-diffeomorphism, then the components of $J(\varphi^{-1})$ belong to $\hold[\min(\rho-1,l)]$ and hence $\varphi^{-1}$ is a $\min(l+1,\rho)$-diffeomorphism. This inductive argument and the induction starting point that $\varphi^{-1}$ is a $1$-diffeomorphism (by part (i) and (iii)) prove  that $\varphi^{-1}$ is a $\rho$-diffeomorphism. Thus the lemma is shown.  
\end{proof}

\subsection{Diffeomorphisms for H\"older and Lebesgue spaces}

We continue with two standard analytical observations which pave the way for our diffeomorphism theorem for function spaces on $\R^n$.
\begin{Lemma}
 \label{HoelderDiff}
Let $\varphi$ be a $\rho$-diffeomorphism and let $\max(1,s)\leq \rho$. Then there exists a constant $C$ depending on $\rho$ such that for all $f \in \hold$ it holds
\begin{align*}
 \|f\circ \varphi|\hold\|\leq C_{\varphi} \cdot \|f|\hold\|.
\end{align*}
\end{Lemma}
\begin{proof}
By definition
\begin{align*}
 \|f \circ \varphi |\hold\|= \|f \circ \varphi |C^{\sint}(\R^n)\|+\sum_{|\alpha|=\sint} \|D^{\alpha} \left[f \circ \varphi\right]|\lip[\srest]\|.
\end{align*}
The lemma follows now by using the chain rule and Leibniz rule for spaces of differentiable functions and for H\"older spaces $\hold$.
\end{proof}
\begin{Remark}
\label{DiffUni} 
As one can easily see, the constant in Lemma \ref{HoelderDiff} depends on $\sum\limits_{i=1}^n \sum\limits_{j=1}^n \left\|\frac{\partial \varphi_i}{\partial x_j}|\hold[\rho-1]\right\|$. If we have a sequence of functions $\{\varphi^m\}_{m \in \N}$ and 
\begin{align*}
\sup_{m \in \N}\sum_{i=1}^n\sum_{j=1}^n  \left\|\frac{\partial \varphi_i^m}{\partial x_j}\big|\hold[\rho-1]\right\|< \infty,
\end{align*} 
then there is a universal constant $C$ with $C_{\varphi_m}\leq C$, i.\,e.\ for all $m \in \N$ it holds
\begin{align*}
 \|f\circ \varphi_m|\hold\|\leq C \cdot \|f|\hold\|.
\end{align*}
\end{Remark}

\begin{Lemma}
\label{Lpdiff}
 Let $0< p \leq \infty$. Let $\varphi: \R^n \rightarrow \R^n$ be bijective and let there be a constant $c>0$ such that
\begin{align}
\label{Lip2}
 c \leq \frac{|\varphi(x)-\varphi(y)|}{|x-y|}
\end{align}
for $x,y \in \R^n$ with $x\neq y$.
Then there is a constant $C>0$ such that
\begin{align}
\label{DiffLp}
 \|f \circ \varphi|L_p(\R^n) \| \leq C \cdot \|f\|L_p(\R^n)\|.
\end{align}
\end{Lemma}
\begin{proof}
If $p<\infty$, it suffices to prove \eqref{DiffLp} for 
\begin{align*}
 f=\sum_{j=1}^N a_j \chi_{A_j},
\end{align*}
where $a_j \in \C$, $A_j$ are pairwise disjoint rectangles in $\R^n$ and $\chi_{A_j}$ is the characteristic function of $A_j$. We have
\begin{align*}
 \int |(f \circ \varphi)(x)|^p \ dx = 
 \int \Big|\sum_{j=1}^N a_j \chi_{\varphi^{-1}(A_j)}(x)\Big|^p \ dx= \sum_{j=1}^N |a_j|^p \mu(\varphi^{-1}(A_j))
\end{align*}
because the preimages $\varphi^{-1}(A_j)$ are also pairwise disjoint.
Hence we have to show:

\textit{
There is a constant $C>0$ such that for all rectangles $A$ it holds
\begin{align}
\label{DiffMu}
 \mu(\varphi^{-1}(A)) \leq C \cdot \mu(A).
\end{align}
}
To prove this let $B_r(x_0)=\{x \in \R^n: |x-x_0|<r\}$ be the open ball around $x_0 \in \R^n$ with radius $r>0$. Then by \eqref{Lip2} we have
\begin{align}
\label{Balls}
 \varphi^{-1}(B_r(x_0)) \subset B_{\frac{r}{c}} (\varphi^{-1}(x_0)).
\end{align}
Hence there is a constant $C>0$ such that
\begin{align*}
 \mu(\varphi^{-1}(B_r(x_0)))< C \cdot \mu(B_r(x_0))
\end{align*}
for all $x_0 \in \R^n$, $r>0$.

Now, we cover a given rectangle $A$ with finitely many open balls $\{B_j\}_{j=1}^M$ such that
\begin{align}
\label{Covering}
 \mu\left(\bigcup_{j=1}^M B_j\right) \leq 2\mu(A). 
\end{align}
Afterwards we make use of the following Vitali covering lemma: There exists a subcollection $B_{j_1},\ldots,B_{j_m}$ of these balls which are pairwise disjoint and satisfy
\begin{align*}
  \bigcup_{j=1}^M B_j \subset \bigcup_{k=1}^m 3\cdot B_{j_k}.
\end{align*}
Using this, \eqref{Covering} and \eqref{Balls} for the balls $3 \cdot B_{j_k}$ finally gives
\begin{align*}
 \mu(\varphi^{-1}(A)) &\leq \mu\left(\varphi^{-1}\left(\bigcup_{j=1}^M B_j\right)\right)\leq \mu\left(\varphi^{-1}\left(\bigcup_{k=1}^M 3\cdot B_{j_k}\right)\right)
=\mu\left(\bigcup_{k=1}^M \varphi^{-1}(3\cdot B_{j_k})\right) \\
&=\sum_{k=1}^M \mu(\varphi^{-1}(3\cdot B_{j_k})) \leq C \cdot \sum_{k=1}^M \mu(3\cdot B_{j_k})\leq C \cdot 3^n  \cdot \sum_{k=1}^M \mu(B_{j_k})\\
& \leq 2C \cdot  3^n \cdot \mu(A).
\end{align*}
This proves the result for $0<p<\infty$.

For $p=\infty$ we have to show 
\begin{align*}
 \|f\circ \varphi|L_{\infty}(\R^n)\| \leq \|f|L_{\infty}(\R^n)\|.
\end{align*}
This follows from: If $\mu(\{x \in \R^n: |f(x)|>a)\})=0$, then also $\mu(\{x \in \R^n: |f(\varphi(x))|>a\})=0$,
which is a consequence of \eqref{DiffMu}:

\textit{Let $M$ be a measurable set with $\mu(M)=0$. Then also $\mu(\varphi^{-1}(M))=0$.}

Hence the lemma is shown for $p=\infty$, too.

\end{proof}

\begin{Remark}
 A proof of a more general observation using the Radon-Nikodym derivative and the Lebesgue point theorem can be found in Corollary 1.3 and Theorem 1.4 of \cite{Vod89} - but here we wanted to give a direct, more instructive proof for our special situation. 
\end{Remark}

\begin{Remark}
 By the previous proof it is obvious that the measure theoretical Condition \eqref{DiffMu} is equivalent to the boundedness of the diffeomorphism \eqref{DiffLp} for $0<p<\infty$. Condition \eqref{DiffMu} does not depend on $p$. For Condition \eqref{DiffMu} it is necessary that the measure $m$ with $m(A):=\mu(\varphi^{-1}(A))$ is absolutely continuous with respect to the Lebesgue measure $\mu$. In case of $p=\infty$ this condition is also sufficient for \eqref{DiffMu} by the previous proof.
\end{Remark}

\subsection{Diffeomorphisms for function spaces on $\R^n$}

Now we are ready for the main theorem of this section.
\begin{Theorem}
\label{Diffeo}
 Let $s \in \mathbb{R}$, $0<q\leq \infty$ and $\rho\geq 1$.

(i) Let $0<p\leq \infty$ and $\rho > \max(s,1+\sigma_p-s)$. If $\varphi$ is a $\rho$-diffeomorphism, then there exists a constant $c$ such that
\begin{align*}
 \|f(\varphi(\cdot))|\bspq\| \leq c \cdot \|f|\bspq\|
\end{align*}
for all $f \in \bspq$. Hence $D_{\varphi}$ maps $\bspq$ onto $\bspq$.

(ii) Let $0<p<\infty$ and $\rho > \max(s,1+\sigma_{p,q}-s)$. If $\varphi$ is a $\rho$-diffeomorphism, then there exists a constant $c$ such that
\begin{align*}
 \|f(\varphi(\cdot))|\fspq\| \leq c \cdot \|f|\fspq\|
\end{align*}
for all $f \in \fspq$. Hence $D_{\varphi}$ maps $\fspq$ onto $\fspq$.
\end{Theorem}
\begin{proof}
At first, beside the two atomic conditions \eqref{Atom2} and \eqref{Atom3} we need to take a closer look at the centres and supports of the atoms. Briefly speaking, the decisive local properties of the set of atoms $a_{\nu,m}$ are maintained by a superposition with the diffeomorphism $\varphi$.

To be more specific: Let $M_{\nu}=\left\{x \in \R^n: x=2^{-\nu}m, m \in \mathbb{Z}^n \right\}$. Having in mind Lemma \ref{DiffRem} there is a $c_2>0$ with
\begin{align}
\label{DiffLip}
  |x-y| \leq c_2 |\varphi^{-1}(x)-\varphi^{-1}(y)|.  
\end{align}
 By a simple volume argument for $Q_{\nu,m}$ and by $|2^{-\nu}m-2^{-\nu}m'|\geq c \cdot 2^{-\nu}$ for $m\neq m'$ there is a constant $M \sim c_2^n$  such that 
\begin{align*}
  |\varphi^{-1}(M_{\nu}) \cap Q_{\nu,m}| \leq M
\end{align*}
for all $\nu \in \mathbb{N}_0,m \in \mathbb{Z}$. Hence we can take our atomic decomposition and split it into $M$ disjunct sums, i.\,e.\
\begin{align*}
 f = \sum_{j=1}^M \sum_{\nu \in \N_0} \sum_{m \in M_{\nu,j}} \lambda_{\nu,m} a_{\nu,m} 
\end{align*}
with
\begin{align*}
 \bigcup_{j=1}^M M_{\nu,j}= \Z^n, \quad M_{\nu,j} \cap M_{\nu,j'} = \emptyset \text{ for } j\neq j'
\end{align*}
so that for all $\nu \in \N_0$, $m \in \Z^n$ and $j \in \{1,\ldots,M\}$
\begin{align}
\label{injective}
  |\left\{m' \in \Z^n: m' \in M_{\nu,j} \text{ and } \varphi^{-1}(2^{-\nu}m') \in Q_{\nu,m}\right\}|\leq 1.
\end{align}
Therefore, not more than one function $a_{\nu',m'}\circ \varphi$ is located at the cube $Q_{\nu,m}$ for each of the $M$ sums.

The support of a function $a_{\nu,m}\circ \varphi$ is contained in $\varphi^{-1}(d \cdot Q_{\nu,m})$ by \eqref{Atom1}. By Lemma \ref{DiffRem} there exists a $c_1>0$ with
\begin{align*}
 |\varphi^{-1}(x)-\varphi^{-1}(y)| \leq \frac{1}{c_1} |x-y|.
\end{align*}
Hence we get
\begin{align*}
  \varphi^{-1}(d \cdot Q_{\nu,m}) \subset c \cdot  \frac{d}{c_1} \cdot B_{2^{-\nu}}(\varphi^{-1}(2^{-\nu}m)), 
\end{align*}
where $B_r(x_0)=\left\{ x \in \R^n: |x-x_0|\leq r \right\}$. Hence, together with \eqref{injective} it follows: There is a constant $d'$ depending on $c_1$ such that for every $\nu \in \N_0$ and every $j \in \{1,\ldots,M\}$ there is an injective map $\Phi_{\nu,j}: M_{\nu,j} \rightarrow \Z^n$ with
\begin{align}
\label{DiffSupp}
 supp \ (a_{\nu,m}\circ \varphi) \subset d' \cdot Q_{\nu,\Phi_{\nu,j}(m)}.
\end{align}
for all $m \in M_{\nu,j}$. The constant $d'$ does not depend on $\nu$ or $m$.

Thus, if we take the derivative conditions \eqref{Atom2} and the general moment conditions \eqref{Atom3} for $a_{\nu,m}\circ \varphi$ now for granted (which will be shown later), then 
\begin{align*}
  f_{j} \circ \varphi=\sum_{\nu \in \N_0} \sum_{m \in M_{\nu,j}} \lambda_{\nu,m} (a_{\nu,m} \circ \varphi)
\end{align*}
is an atomic decomposition of the function $f_j \circ \varphi$. Finally, we have to look at the sequence space norms, see Definition \ref{DefSeq}. 

We will concentrate on the $\fspq$-case since the $\bspq$-case is easier because it does not matter if one changes the order of summation over $m$. By the atomic representation theorem and \eqref{DiffSupp} we will have
\begin{align*}
 \|f_j \circ \varphi|\fspq\| \lesssim \left\|\left(\sum_{\nu=0}^{\infty} \sum_{m \in M_{\nu,j}} |\lambda_{\nu,m}\chi_{\nu,\Phi_{\nu,j}(m)}^{(p)}(\cdot)|^q\right)^{\frac{1}{q}}\big|L_p(\R^n)\right\|.
\end{align*}
To transfer this into the usual sequence space norm we make use of
\begin{align}
\label{MapLip}
  Q_{\nu,\Phi_{\nu,j}(m)} \subset \varphi^{-1} (c \cdot Q_{\nu,m}) 
\end{align}
with a constant $c$ depending on $c_2$ from \eqref{DiffLip}, but independent of $\nu$ and $m$. This follows from $\varphi^{-1} (2^{-\nu}m) \in Q_{\nu,\Phi_{\nu,j}(m)}$. Hence assuming that $a_{\nu,m}\circ \varphi$ fulfil the atomic conditions \eqref{Atom2} and \eqref{Atom3} we obtain
\begin{align*}
 \|f_{j} \circ \varphi|\fspq\| &\lesssim \left\|\left(\sum_{\nu=0}^{\infty} \sum_{m \in M_{\nu,j}} |\lambda_{\nu,m}\chi_{\nu,m}^{(p)}(\varphi (\cdot))|^q\right)^{\frac{1}{q}}\big|L_p(\R^n)\right\| \\
	  &\lesssim \left\|\left(\sum_{\nu=0}^{\infty} \sum_{m \in M_{\nu,j}} |\lambda_{\nu,m}\chi_{\nu,m}^{(p)}(\cdot)|^q\right)^{\frac{1}{q}}\big|L_p(\R^n)\right\| \\
&\lesssim \left\|\left(\sum_{\nu=0}^{\infty} \sum_{m \in \Z^n} |\lambda_{\nu,m}\chi_{\nu,m}^{(p)}(\cdot)|^q\right)^{\frac{1}{q}}\big|L_p(\R^n)\right\| \\
	  &\lesssim \|f|\fspq\|.
\end{align*} 
In the first step we used \eqref{MapLip}, in the second step we used Lemma \ref{Lpdiff} and part (iii) of Lemma \ref{DiffRem} and in the last step we applied the atomic decomposition theorem for $f$. As done in the first step, one can replace the characteristic function of $c \cdot Q_{\nu,m}$ by the characteristic function of $Q_{\nu,m}$ in the sequence space norm getting equivalent norms, see \cite[Section 1.5.3]{Tri08}. This can be proven using the Hardy-Littlewood maximal function.

Finally, we have to take a look at the derivative conditions \eqref{Atom2} and the general moment conditions \eqref{Atom3}. The latter part is also considered in \cite[Lemma 5]{Skr98} using the atomic approach with derivative condition \eqref{Atom21}. 

Let $a_{\nu,m}$ be an $(s,p)_{K,L}$-atom and let $\rho\geq \max(K,L+1)$. If we can show that $\varphi \circ a_{\nu,m}$ is an $(s,p)_{K,L}$-atom as well, we are done with the proof since we can choose $K$ and $L$ suitably small enough by the atomic decomposition theorem \ref{AtomicRepr}. Let $T_{\nu}(x):=2^{-\nu}x$ and ${\cal T}_{\nu}(\varphi)= T_{\nu}^{-1} \circ \varphi \circ T_{\nu}$. Then
\begin{align*}
 \|\left(a_{\nu,m}\circ \varphi\right)(2^{-\nu}\cdot)|\hold[K]\|&=  \|a_{\nu,m}\circ \varphi \circ T_{\nu}|\hold[K]\| \\
&= \|a_{\nu,m}\circ T_{\nu} \circ {\cal T}_{\nu}(\varphi)|\hold[K]\|.
\end{align*}
By a simple dilation argument for the H\"older spaces $\hold[\rho-1]$ with $\rho\geq 1$ it holds 
\begin{align*}
  \left\|\frac{\partial \left({\cal T}_{\nu}(\varphi)\right)_i}{\partial x_j}\big|\hold[\rho-1]\right\|\leq \left\|\frac{\partial \varphi_i}{\partial x_j}\big|\hold[\rho-1]\right\|
\end{align*}
for all $i,j\in \{1,\ldots,n\}$ and $\nu \in \N_0$. Hence by Lemma \ref{HoelderDiff} and Remark \ref{DiffUni} we find a constant $C$ independent of $\nu$ and $m$ such that
\begin{align*}
  \|\!\left(a_{\nu,m}\circ \varphi\right)(2^{-\nu}\cdot)|\hold[K]\|= \|a_{\nu,m}\circ T_{\nu} \circ {\cal T}_{\nu}(\varphi)|\hold[K]\| \leq \! C\!	 \cdot 
 \|a_{\nu,m}(2^{-\nu}\cdot)|\hold[K]\|
\end{align*}
So the derivative condition \eqref{Atom2} is shown. 

Regarding the general moment condition \eqref{Atom3} of $a_{\nu,m}\circ \varphi$ we consider two cases: At first, let $\varphi$ be a $\rho$-diffeomorphism with $\rho>1$. Then $\varphi$ and $\varphi^{-1}$ are differentiable. We use the general moment condition \eqref{Atom3} of $a_{\nu,m}$ itself and Lemma \ref{HoelderDiff} to get
\begin{align*}
\left| \, \int\limits_{d' \cdot Q_{\nu,\Phi_{\nu,j}(m)}} \psi(x) \cdot a(\varphi(x)) \ dx \right| &=\left|\,\int\limits_{\varphi^{-1}\left(d \cdot Q_{\nu,m}\right)} \psi(x) \cdot a(\varphi(x)) \ dx \right| \\
&=\left| \, \int\limits_{d \cdot Q_{\nu,m}} \psi\left(\varphi^{-1}(x)\right) \cdot |\det \varphi^{-1}|(x) \cdot a(x) \ dx \right| \\
&\leq C  \cdot 2^{-\nu\Kap} \cdot \||\det \varphi^{-1}(x)|\cdot \left(\psi\circ \varphi^{-1}\right)|\hold[L]\| \\
&\leq C' \cdot 2^{-\nu\Kap} \cdot \|\psi|\hold[L]\|.
\end{align*}
We used the transformation formula for integrals and 
\begin{align*}
  \det\ J\left(\varphi^{-1}\right)\in \hold[L]
\end{align*}
since $\varphi$ is a $\rho$-diffeomorphism with $\rho \geq L+1$. Furthermore, the sign of $\det J\left(\varphi^{-1}\right)$ is constant by Definition \ref{LipDef}. 

If $\rho=1$, then $L=0$ by our choice of $\rho$. This means, that no moment conditions are needed. Hence we have nothing to prove. The choice of $\rho=1$ is only allowed if $\sigma_p<s<1$ resp.\ $\sigma_{p,q}<s<1$. 

For some further technicalities similar as in Remark \ref{ProdDef} see Remark \ref{DiffDef}.
\end{proof}

\begin{Remark}
 This has been proven (in a sketchy way) in \cite[Lemma 3]{Skr98} for the more special atomic definition there. 
\end{Remark}

\begin{Remark}
 If $\sigma_p<s<1$ resp.\ $\sigma_{p,q}<s<1$, then the choice of $\rho=1$ is possible for these values of $s$. This gives the same result as in \cite[Proposition 4.1]{Tri02}, where the notation of Lipschitz diffeomorphisms as in Definition \ref{LipDef} is used. This results in
\begin{Theorem}
Let $0<q\leq \infty$.

(i) Let $0<p\leq \infty$ and $\sigma_p<s<1$. If $\varphi:\R^n \rightarrow \R^n$ is a bi-Lipschitzian map, then there exists a constant $c$ such that
\begin{align*}
 \|f(\varphi(\cdot))|\bspq\| \leq c \cdot \|f|\bspq\|.
\end{align*}
for all $f \in \bspq$. Hence $D_{\varphi}$ maps $\bspq$ onto $\bspq$.

(ii) Let $0<p<\infty$ and $\sigma_{p,q}<s<1$. If $\varphi:\R^n \rightarrow \R^n$ is a bi-Lipschitzian map, then there exists a constant $c$ such that
\begin{align*}
 \|f(\varphi(\cdot))|\fspq\| \leq c \cdot \|f|\fspq\|.
\end{align*}
for all $f \in \fspq$. Hence $D_{\varphi}$ maps $\fspq$ onto $\fspq$.
\end{Theorem}

\end{Remark}

\begin{Remark}
\label{DiffDef} 
We have to deal with some technicalities of the proof of Theorem \ref{Diffeo}. We concentrate on the $\bspq$-case, the $\fspq$-case is nearly the same.

Let at first be $\rho>1$. In principle, Theorem \ref{Diffeo} and Lemma \ref{HarmS'-KonvAtom} show that
\begin{align}
\label{AtomKonv2}
 \sum_{\nu} \sum_{m}\lambda_{\nu,m} (a_{\nu,m} \circ \varphi) 
\end{align}
converges unconditionally in ${\cal S}'(\R^n)$, where
\begin{align*}
 f=\sum_{\nu} \sum_{m}\lambda_{\nu,m} a_{\nu,m} \text{ in } {\cal S}'(\R^n),
\end{align*}
and the limit belongs to $\bspq$ if $f$ belongs to $\bspq$. 

To define the superposition of $f$ and $\varphi$ as this limit, we have to show that the limit does not depend on the atomic decomposition we chose for $f$. Let $\psi \in \Sc(\R^n)$ with compact support be given. Then
\begin{multline*}
 \sum_{\nu} \sum_{m}\left| \int_{\R^n}  \lambda_{\nu,m} \left(a_{\nu,m}\circ \varphi \right)(x) \psi(x) \ dx\right| \\
 =\sum_{\nu} \sum_{m}\left| \int_{\R^n}  \lambda_{\nu,m} a_{\nu,m}(x) \left[ \psi\left(\varphi^{-1}(x)\right) \cdot |\det \varphi^{-1}(x)| \right] \ dx\right|
\end{multline*}
makes sense, see \eqref{KonvB}, because by Lemma \ref{HoelderDiff} the function $\psi\left(\varphi^{-1}(x)\right) \cdot |\det \varphi^{-1}(x)|$ has compact support and belongs to $\hold[M]$ for a suitable $M>0$ with $M>\sigma_p-s$. Now the achievements at the end of Corollary \ref{FolgLocal} show that this integral limit does not depend on the choice of the atomic decomposition for $f$. Hence we obtain that the limit in \eqref{AtomKonv2} (considered as an element in ${\cal S}'(\R^n)$) is the same for all choices of atomic decompositions.  

If the choice of $\rho=1$ is allowed, then automatically $s>\sigma_p$ and $\bspq$ consists of regular distributions by Sobolev's embedding inProposition \ref{Sobolev}. Hence the superposition of $f \in \bspq \subset L_p(\R^n)$ for $1 \leq p \leq \infty$ resp.\ $f \in \bspq \subset L_1(\R^n)$ for $0<p\leq 1$  with a $1$-diffeomorphism $\varphi$ is defined as the superposition of a regular distribution with a $1$-diffeomorphism and is continuous as an operator from $L_p(\R^n)$ resp.\ $L_1(\R^n)$ into $L_p(\R^n)$ resp.\ $L_1(\R^n)$ by Lemma \ref{Lpdiff}.

If $p<\infty$, then atomic decompositions of $f \in \bspq$ converge to $f$ with respect to the norm of $L_p(\R^n)$ for $1\leq p < \infty$ resp.\ with respect to the norm of $L_1(\R^n)$ for $0<p<1$, see \cite[Section 2.12]{Tri06}. Hence the limit does not depend on the choice of the atomic decomposition and is equal to the usual definition of the superposition of a regular distribution $f$ and the $1$-diffeomorphism $\varphi$. 

If $p=\infty$, we use the local convergence of the atomic decompositions of $f$ in $L_{\infty}(\R^n)$, i.\,e.\ we restrict $f$ and its atomic decomposition to a compact subset $K$ of $\R^n$. Then this restricted atomic decomposition converges to the restricted $f$ with respect to the norm of $L_{\infty}(K)$. This suffices to prove uniqueness of the limit which is an $L_{\infty}(\R^n)$-function.
\end{Remark}

\begin{Remark}
For fixed $s,p$ and $q$ the constant $c$ in Theorem \ref{Diffeo} depends on the $\rho$-diffeomorphism $\varphi$. Looking into the proof of Theorem \ref{Diffeo} and Remark \ref{DiffUni} the following definition is useful: 
\end{Remark}

\begin{Definition}
 Let $\rho\geq 1$. We call $\{\varphi^m\}_{m \in \N}$ a bounded sequence of $\rho$-dif\-feo\-mor\-phisms if every $\varphi^m$ is a $\rho$-dif\-feo\-mor\-phism and if there are universal constants $c_1,c_2>0$ with
\begin{align*}
 c_1\leq \frac{|\varphi^m(x)-\varphi^m(y)|}{|x-y|} \leq c_2 
\end{align*}
for $m \in \N$, $x,y\in \R^n$ with $0<|x-y|\leq 1$ and if - for $\rho>1$ - there is a universal constant $c$ with   
\begin{align*}
\sum_{i=1}^n\sum_{j=1}^n  \left\|\frac{\partial \varphi_i^m}{\partial x_j}\big|\hold[\rho-1]\right\|< c.
\end{align*}
for $m \in \N$.
\end{Definition}

\begin{Remark}
 If $\{\varphi^m\}_{m \in \N}$ is a bounded sequence of $\rho$-diffeomorphisms, then $(\varphi^m)^{-1}$ exists for all $m \in \N$ and $\{(\varphi^m)^{-1} \}_{m \in \N}$ is a bounded sequence of $\rho$-diffeomorphisms, too. This follows by the arguments of Lemma \ref{DiffRem}.
\end{Remark}
Now, by going through the proof of Theorem \ref{Diffeo} and Remark \ref{DiffUni} it follows
\begin{Corollary}
 Let $s \in \mathbb{R}$, $0<q\leq \infty$ and $\rho\geq 1$.

(i) Let $0<p\leq \infty$ and $\rho > \max(s,1+\sigma_p-s)$. If $\{\varphi^m\}_{m \in \N}$ is a bounded sequence of $\rho$-diffeomorphisms, then there exists a constant $C$ such that
\begin{align*}
 \|f(\varphi^m(\cdot))|\bspq\| \leq C \cdot \|f|\bspq\|
\end{align*}
for all $f \in \bspq$ and $m \in \N$.

(ii) Let $0<p<\infty$ and $\rho > \max(s,1+\sigma_{p,q}-s)$. If $\{\varphi^m\}_{m \in \N}$ is a bounded sequence of $\rho$-diffeomorphisms, then there exists a constant $C$ such that
\begin{align*}
 \|f(\varphi^m(\cdot))|\fspq\| \leq C \cdot \|f|\fspq\|
\end{align*}
for all $f \in \fspq$ and $m\in \N$.
\end{Corollary}

%% file: Zerlegungstheorem.tex
\label{Zerlegungsth2}
\section{Basic notation}
\label{Zerlegungsth}
Let $n \in \N$ with $l<n$. Let $\R^n=\R^{l} \times \R^{n-l}$ and $x=(y,z) \in \R^n$,
\begin{align*}
 y=(y_1,\ldots,y_l) \in \R^l, z=(z_1,\ldots,z_{n-l}) \in \R^{n-l}.
\end{align*}
We identify $\R^l$ with the hyperplane $\{z=0\} \subset \R^n$. Hence, in our understanding 
\begin{align*}
\hyp=\left\{x=(y,z) \in \R^n: z \neq 0\right\}.   
\end{align*}
Furthermore, let 
\begin{align*}
 Q_l = \{ x=(y,z) \in \R^n: z=0, 0<y_m<1, m=1, \ldots, l\} \subset  \R^l
\end{align*}
be the unit cube in this hyperplane and let
\begin{align*}
 Q_l^n = \{ x=(y,z) \in \R^n: (y,0) \in Q_l, z \in \R^{n-l}\}	
\end{align*}
be the related cylindrical domain in $\R^n$, see \cite[Section 6.1.3]{Tri08}.

Let
\begin{align*}
 \N_l^n=\left\{ \alpha=(\alpha_1,\ldots,\alpha_n) \in \N_0^n: \alpha_1=\ldots=\alpha_l=0\right\}.
\end{align*}
Then by 
\begin{align*}
 D^{\alpha}f = \frac{\partial^{|\alpha|} f}{\partial z_1^{\alpha_{l+1}} \ldots \partial z_n^{\alpha_{n}}}, \alpha \in \N_l^n
\end{align*}
we denote the derivatives perpendicular to $\R^l$.

\section{Reinforced spaces for $\R^n \setminus \R^l$}
\label{Reinf}
In \cite[Section 6.1.4]{Tri08} Triebel showed the following, crucial property which paved the way to the wavelet characterization for the cube $Q$:
\begin{Proposition}[Triebel]
\label{decomptri}
 Let $l \in \N$ and $l<n$. Let
\begin{align*}
1\leq p < \infty, 0<q<\infty, 0<s-\frac{n-l}{p} \notin \N. 
\end{align*}
Then $D(Q_l^n \setminus Q_l)$ is dense in 
\begin{align*}
 \left\{f \in \At[Q_l^n]: \tr_l^r f=0 \right\}
\end{align*}
with
\begin{align*}
 r=\lfloor s-\frac{n-l}{p}\rfloor
\end{align*}
with $\lfloor \cdot \rfloor$ as defined in Definition \ref{Hoelder}.
\end{Proposition}
Here $\tr_l^r f$ is the trace operator onto $Q_l$, see the introduction of traces in front of the upcoming Proposition \ref{Tra}. 

However, when $s-\frac{n-l}{p} \in \N$, Proposition \ref{decomptri} cannot be proven in this way and should not be true in general. As suggested in \cite[Section 6.2.3]{Tri08} we have to ``reinforce'' the function spaces $\FR$. Furthermore, for our substitute of Proposition \ref{decomptri} we replace $Q_l$ by $\R^n \setminus \R^l$ and $Q_l^n$ by $\R^n$. We start with some basic observations regarding Hardy inequalities for function spaces on $\R^n$. 
 
\subsection{Hardy inequalities at $l$-dimensional planes}

\begin{Definition}
Let $x$ in $\Om$ and
\begin{align*}
 d(x)=\dist(x,\Gamma)= \inf\{|x-y|: y \in \Gamma \}
\end{align*}
be the distance of $x$ to $\Gamma=\partial \Om$. 

Let 
\begin{align*}
 \Om_{\eps}=\left\{x \in \Om: d(x)<\eps\right\}
\end{align*}
for $\eps>0$ sufficiently small.
\end{Definition}
Let $\Om=\R^n\setminus \R^l$ and $x=(x',x'') \in \R^n=\R^l \times \R^{n-l}$. Then in our special situation we have $d(x)=|x''|$, where $|\cdot|$ is the Euclidean distance in $\R^{n-l}$.

\begin{Proposition}[Sharp Hardy inequalities - the critical case]
 \label{Hardycrit}
 Let $0<\eps<1$, $1<p<\infty$ and $0<q \leq \infty$. Let $\varkappa$ be a positive monotonically decreasing function on $(0,\eps)$. Then
\begin{align*}
 \int_{(\hyp)_{\eps}} \left|\frac{\varkappa(d(x))f(x)}{\log d(x) } \right|^p \frac{ dx}{d^{n-l}(x)} \leq c \left\|f| \FR[\frac{n-l}{p}] \right\|^p
\end{align*}
for some $c>0$ and all $f \in \FR[\frac{n-l}{p}]$ 
\begin{align*}
\text{if and only if $\varkappa$ is bounded.}
\end{align*}
\end{Proposition}
\begin{Proof}
 The proof is a generalization of the discussion in \cite[Section 16.6]{Tri01}. There the case $l=n-1$ is considered. One uses the one-dimensional version of the Hardy inequality (16.8) in \cite{Tri01}.
 
For the ``if-part'' let at first $1<q \leq \infty$. We now use the $(n-l)$-dimensional version of (16.8) in \cite{Tri01}. Let $x=(x',x'') \in \R^n=\R^l \times \R^{n-l}$. We fix $x' \in \R^l$ and get
\begin{align*}
 \int_{|x''|<\eps} \left|\frac{\varkappa(|x''|)|f(x',x'')}{\log |x''|} \right|^p \frac{ dx''}{|x''|^{n-l}} \lesssim \left\|f| F_{p,q}^{\frac{n-l}{p}}(\R^{n-l}) \right\|^p.
\end{align*}
We now integrate over $x' \in \R^l$ and make use of the Fubini property of $\FR$, see Proposition \ref{Fubini}. Using $d(x)=|x''|$ this shows
\begin{align}
\begin{split}
 \label{Hardyl}
 \int_{(\hyp)_{\eps}} \left|\frac{\varkappa(d(x))f(x)}{\log d(x) } \right|^p \frac{ dx}{d^{n-l}(x)} &\lesssim \left\|\big\|f| F_{p,q}^{\frac{n-l}{p}}(\R^{n-l}) \big\||L_p(\R^{l})\right\| \\
& \lesssim \|f | \FR\|.
\end{split}
\end{align}
Since for fixed $p$ with $1<p<\infty$ the spaces $F_{p,q}^{\frac{n-l}{p}}(\R^{n-l})$ are monotonic with respect to $q$, inequality \eqref{Hardyl} holds for all $1<p<\infty$ and $0<q \leq \infty$.

For the ``only if-part'' we have to show that $\varkappa$ must be bounded. The proof is a generalization of the discussion in \cite[Section 16.6]{Tri01} for dimension $l=n-1$. We consider the set
\begin{align*}
 S_J^l=\{x=(x',x'') \in \R^{l} \times \R^{n-l}: |x'|<1, |x''|<2^{-J} \}, \ J \in \N
\end{align*}
and
\begin{align*}
 S_J^{l,*}=S_J^l \setminus S_{J+1}^l.
\end{align*}
We will construct an $(n-l)$-dimensional substitute of $f_J$ from (16.29) in \cite{Tri01}, such that $f_J \in F_{p,q}^{n-l}(\R^n)$, 
\begin{align}
\label{fJ}
 f_J(x) = J^{\frac{1}{p'}} \text{ for } x \in S_J^l \text{ and } \|f_J|\FR[\frac{n-l}{p}]\| \lesssim 1.
\end{align}
Remember $d(x)=|x''|$ in our setting. Then we will have
\begin{align*}
 \int_{(\hyp)_{\eps}} \left|\frac{\varkappa(d(x))f_J(x)}{\log d(x) } \right|^p \frac{ dx}{d^{n-l}(x)} &\gtrsim  \varkappa(2^{-J})^p J^{1-p}  \int_{S_J^{l,*}} \left|\frac{1}{\log d(x) } \right|^p \frac{ dx}{d^{n-l}(x)} \\
&\sim\varkappa(2^{-J})^p J^{1-p}\int_{2^{-J-1}}^{2^{-J}} r^{n-l-1} \frac{1}{r^{n-l}|\log r|^p} \ dr \\
&\sim \varkappa(2^{-J})^p J^{1-p} \int_{2^{-J-1}}^{2^{-J}} \frac{1}{r|\log r|^p} \ dr \\
&\gtrsim \varkappa(2^{-J})^p J^{1-p} J^{p-1} \\
&= \varkappa(2^{-J})^p
\end{align*}
using $(n-l)$-dimensional spherical coordinates and $p>1$. Since the constants do not depend on $J \in \N$, this shows $\varkappa \lesssim 1$ keeping in mind $\|f_J|\FR[\frac{n-l}{p}]\| \lesssim 1$. Hence the proof will be finished after construction of such a series of functions $f_J$ with $\eqref{fJ}$. We can define them in the following way: For every $j \in \N$ we choose lattice points $x^{j,k} \in S_j^{l,*}$ for $k \in \{1,\ldots, C_j\}$ such that
\begin{align}
\label{overlap2}
 S_j^{l,*} \subset \bigcup_{k=1}^{C_j} B_{2^{-j}}(x^{j,k})
\end{align}
and $|x^{j,k}-x^{j,k'}| \geq 2^{-j}$ for $k \neq k'$. By a simple volume argument we have 
\begin{align}
\label{overlap1}
C_j \sim \frac{|S_j^{l,*}|}{2^{-jn}} \sim 2^{jl}.
\end{align}
Let $\psi \in \Sc(\R^n)$ be non-negative, $\psi(x)=1$ for $|x|\leq\frac{1}{2}$ and $\psi(x)=0$ for $|x|\geq 1$. We set 
\begin{align*}
 f_J(x):=J^{-\frac{1}{p}} \sum_{j=1}^{J} \sum_{k=1}^{C_j}2^{-j \frac{l}{p}} \left[2^{j \frac{l}{p}} \psi(2^{j-1}(x-x^{j,k}))\right].
\end{align*}
At least when $q\geq 1$ and no moment conditions are necessary the functions 
\begin{align*}
\left[2^{j \frac{l}{p}} \psi(2^{j-1}(x-x^{j,k}))\right]
\end{align*}
are correctly normalized atoms in $\FR[\frac{n-l}{p}]$ by Definition \ref{Atoms}. Furthermore, by the support properties we can use the arguments in \cite[Section 2.15]{Tri01} about a modification of the sequence space. The slight overlapping of the functions $\psi(2^{j-1}(x-x^{j,k}))$ for different $j$ can be neglected. Hence by the atomic representation Theorem \ref{AtomicRepr} and $\eqref{overlap1}$ we have
\begin{align*}
 \|f_J|\FR[\frac{n-l}{p}]\| \lesssim J^{-\frac{1}{p}} \left(\sum_{j=1}^{J}\sum_{k=1}^{2^{jl}} \left(2^{-j \frac{l}{p}}\right)^p \right)^{\frac{1}{p}} \sim 1.
\end{align*}
On the other hand using $\eqref{overlap2}$ and the support properties of $\psi$ we get
\begin{align}
\label{overlap3}
 f_J(x) \geq J^{-\frac{1}{p}} \sum_{j=1}^J 1 = J^{\frac{1}{p'}} \text{ for } x \in S_J^l
\end{align}
Hence we have constructed such a function for $q\geq 1$. For $0<q<1$ one has to modify the functions $f_J$ to get moment conditions. These modifications are described in Step 5 of the proof of Theorem 13.2 in \cite{Tri01}. Then one has to define \eqref{overlap2} such that the functions $\psi(2^{j-1}(x-x^{j,k}))$ have disjoint support for fixed $j$ and different $k$. Then they cannot satisfy \eqref{overlap3}. But this is not necessary - it suffices to have $f_J(x)\geq J^{\frac{1}{p'}}$ on a set $A_J^l \subset S_J^l$ with $|A_J^l| \sim |S_J^l|$. This is possible.
\end{Proof}

\begin{Proposition}[Sharp Hardy inequalities - the subcritical case]
 \label{Hardysubcrit}
Let $0<\eps<1$, $1 \leq p<\infty$ and $0<q \leq \infty$. Let 
\begin{align*}
 0<s<\frac{n-l}{p}
\end{align*}
and $\varkappa$ be a positive monotonically decreasing function on $(0,\eps)$. Then
\begin{align*}
 \int_{(\hyp)_{\eps}} \left|\varkappa(d(x))f(x) \right|^p \frac{ dx}{d^{sp}(x)} \leq c \left\|f| \FR \right\|^p
\end{align*}
for some $c>0$ and all $f \in \FR$ 
\begin{align*}
\text{if and only if $\varkappa$ is bounded.}
\end{align*}
\end{Proposition}
\begin{Proof}
The ``if-part'' can be handled in the same way as in Proposition \ref{Hardycrit} before. Now we use the $(n-l)$-dimensional version of $(16.15)$ in \cite{Tri01} having in mind $s-\frac{n-l}{p}=-\frac{n}{r}$. Here $p=1$ is allowed. Then we integrate over $x' \in \R^l$ and make use of the Fubini property \ref{Fubini} to get the desired result, using $d(x)=|x''|$.

For the ``only if-part'' we make a similar approach as in (15.11) of \cite{Tri01}. We take
\begin{align*}
 f_j := 2^{j(-s+\frac{n}{p}-\frac{n}{2})} \cdot \Phi_{r}^j
\end{align*}
for $j \in \N$ with a wavelet $\Phi_{r}^j$ choosen from an oscillating $u$-Riesz basis, see Proposition \ref{waveletre}, such that
\begin{align}
\label{wavedist}
 \dist(supp \ \Phi_{r}^j, \R^l) \sim 2^{-j},
\end{align}
for instance choose $m=(0,\ldots,0,1,1,\ldots,1)$ where the first $l$ coordinates are $0$. Obviously,
\begin{align*}
 f_j(x) = 2^{-j(s-\frac{n}{p})} \Phi_{r'}^0(2^j x)
\end{align*}
for a suitable $r' \in \Z^n$. Then by the atomic representation Theorem \ref{AtomicRepr}
\begin{align*}
 \|f_j|\FR\| \sim 1 
\end{align*}
with constants independent of $j \in \N$. As before, it holds $d(x)=|x''|$ with $x=(x',x'') \in \R^n=\R^l \times \R^{n-l}$. By \eqref{wavedist} we have for large $j$ (in dependence of $\eps$)
\begin{align*}
 \int_{(\hyp)_{\eps}} \left|f_j(x) \right|^p \frac{ dx}{d^{sp}(x)} &\gtrsim 2^{-jp(s-\frac{n}{p})} \cdot \int_{supp \ \Phi_{r}^j} \frac{ dx}{d^{sp}(x)} \\
&\gtrsim 2^{-jp(s-\frac{n}{p})} \cdot 2^{-jn} \cdot 2^{jsp} \\
&\gtrsim 1. 
\end{align*}
Hence $\varkappa(t)$ must be bounded for $t\rightarrow 0$, otherwise we would obtain a contradiction.
\end{Proof}

\subsection{Definition of reinforced function spaces $\Frinf[\hyp]$}

Propositions \ref{Hardycrit} and \ref{Hardysubcrit} describe the different behaviour of the spaces $\FR[\frac{n-l}{p}]$ and $\FR$ for $0<s<\frac{n-l}{p}$ in terms of Hardy inequalities. For the space $\FR[\frac{n-l}{p}]$ we have a weaker inequality - this leads to the following definition of the reinforced spaces for $\Om=\hyp$ with $\partial \Om=\R^l$.

\begin{Definition}
\label{reinforcedhyp}
Let $1\leq p < \infty$, $0<q \leq \infty$ and $s> 0$. 

(i) Let $s-\frac{n-l}{p} \notin \N_0$. Then
\begin{align*}
 \Frinf[\hyp] :=\FR.
\end{align*}

(ii) Let $s-\frac{n-l}{p}=r \in \N_0$. Then
\begin{multline*}
  \Frinf[\hyp]\\
:=\left\{ f \in \FR: d^{-\frac{n-l}{p}} \cdot  D^{\alpha} f \in L_p((\hyp)_{\eps}) \text{ for all } \alpha \in \N_l^n, |\alpha|=r \right\}.
\end{multline*}
 \end{Definition}
\begin{Remark}
 For $s-\frac{n-l}{p}=r \in \N_0$ this space can be normed by
\begin{align*}
 \|f|\Frinf[\hyp]\|&:= \|f|\FR\| + \sum_{\underset{|\alpha|=r}{\alpha \in \N_l^n}}  \left(\int_{(\hyp)_{\eps}}  |D^{\alpha} f(x)|^p  \frac{ dx}{d^{n-l}(x)} \right)^{\frac{1}{p}} \\
&=\|f|\FR\| + \sum_{\underset{|\alpha|=r}{\alpha \in \N_l^n}}  \left(\int_{(\hyp)_{\eps}}  |D^{\alpha} f(x)|^p  \frac{ dx}{d^{(s-r)p}(x)} \right)^{\frac{1}{p}} 
\end{align*}

\end{Remark}

\begin{Remark}
\label{indeps}
 The space $\Frinf[\hyp]$ does not depend on the choice of $\eps$ in the sense of equivalent norms since 
for $|\alpha|=r$ we have $s-r>0$ and hence
\begin{align*}
 D^{\alpha} f \in F_{p,q}^{s-r}(\R^n) \subset L_p(\R^n). 
\end{align*}
Furthermore, we can replace $d(x)$ by $\delta(x)=\min(d(x),1))$.
\end{Remark}
\begin{Remark}
 This definition is adapted by Definition 6.44 in \cite{Tri08}, where the case of a $C^{\infty}$-domain $\Om$ is considered and in this sense $l=n-1$. Then there is only one direction of derivatives to be treated - the normal derivative at the boundary $\Gamma$.  
\end{Remark}
\begin{Remark}
Let $s-\frac{n-l}{p} \notin \N_0$. Then $\Frinf[\hyp]=\FR$ by definition. Let $f \in \FR$, $r:=\lfloor s-\frac{n-l}{p}\rfloor+1$ and additionally assume $s-r>0$: By classical properties of $\FR$ it holds
\begin{align*}
 D^{\alpha} f \in \FR[s-r] \text{ for } |\alpha|=r.
\end{align*}
Using the Hardy inequality from Proposition \ref{Hardysubcrit} and $s-r<\frac{n-l}{p}$ we automatically have 
\begin{align*}
 \int_{(\hyp)_{\eps}} \left|D^{\alpha} f(x) \right|^p \frac{ dx}{d^{(s-r)p}(x)} \lesssim \left\|D^{\alpha}f| \FR[s-r] \right\|^p \lesssim c \left\|f| \FR \right\|^p. 
\end{align*}
The counterpart of this Hardy inequality for $s-\frac{n-l}{p} \in \N_0$ is given by Proposition \ref{Hardycrit} (at least for $1<p<\infty$) and differs from the version for $0<s-r<\frac{n-l}{p}$ -- an extra $\log$-term comes in. This explains why it is somehow natural to require stricter Hardy inequalities for $s-\frac{n-l}{p} \in \N_0$ in the definition of $\Frinf[\hyp]$. 

\end{Remark}

\begin{Remark}
\label{Frinfsmaller}
For the Triebel-Lizorkin spaces $\FR$ we always have
\begin{align*}
F_{p,q}^{s+\sigma}(\R^n)   \hookrightarrow \FR
\end{align*}
for $\sigma>0$. We cannot transfer such an embedding from $\FR$ to $\Frinf[\hyp]$: For incorporating the critical cases ($s-\frac{n-l}{p} \in \N_0$) we would have to show
\begin{align*}
 \|d^{-\frac{n-l}{p}} f|L_p((\hyp)_{\eps})\| \lesssim \|f|F_{p,q}^{\frac{n-l}{p}+\sigma}(\R^n)\|
\end{align*}
for $1\leq p <\infty$ and $\sigma>0$.

Let $x=(x',x'') \in \R^l \times \R^{n-l}$ as before. If we take a function $\psi \in D(\R^n)$ with $\psi(x)=1$ with $|x'|\leq 1, |x''|\leq 1$, then 
\begin{align*}
 \|d^{-\frac{n-l}{p}} \psi |L_p((\hyp)_{1})\|^p \geq  \int_{|x'|\leq 1} \int_{|x''|\leq 1} |x''|^{-(n-l)} \ dx'' \ dx'=\infty.
\end{align*}
But $\psi \in \FR$ for all $s>0$. This shows
\begin{align*}
 F_{p,q}^{\frac{n-l}{p}+\sigma}(\R^n)\lhook\joinrel\not\rightarrow F_{p,q}^{\frac{n-l}{p},\rinf}(\hyp)
\end{align*}
for $0<\sigma<1$ and also	
\begin{align*}
  F_{p,q}^{\frac{n-l}{p},\rinf}(\hyp) \subsetneq F_{p,q}^{\frac{n-l}{p}}(\R^n). 
\end{align*}
If we take a function $\psi^{r} \in D(\R^n)$ with $\psi^{r}(x)=x_n^r$ for $|x'|\leq 1, |x''|\leq 1$, where $x_n$ is the $n$-th coordinate of $x$, then $\left(D^{\alpha} \psi^r\right)(x) =1$ for $|x'|\leq 1, |x''|\leq 1$ and $\alpha=(0,\ldots,0,r)$ -  the $r$-th derivative in $x_n$-direction. By the previous steps we have
\begin{align*}
 \|d^{-\frac{n-l}{p}} D^{\alpha}\psi^r |L_p((\hyp)_{1})\|^p=\infty.
\end{align*}
On the other hand surely $\psi^{r} \in \FR$ for all $s>0$. This shows
\begin{align*}
 F_{p,q}^{r+\frac{n-l}{p}+\sigma}(\R^n) \lhook\joinrel\not\rightarrow F_{p,q}^{r+\frac{n-l}{p},\rinf}(\hyp)
\end{align*}
for $0<\sigma<1$ and also
\begin{align*}
 F_{p,q}^{r+\frac{n-l}{p},\rinf}(\hyp) \subsetneq F_{p,q}^{r+\frac{n-l}{p}}(\R^n).
\end{align*}
In Corollary \ref{embed_rinf} we will prove a weaker version under additional properties. It holds
\begin{align*}
 f \in F_{p,q}^{r+\frac{n-l}{p}+\sigma}(\R^n) \text{ belongs to } F_{p,q}^{r+\frac{n-l}{p},\rinf}(\hyp) 
\end{align*}
for $0<\sigma<1$ if
\begin{align*}
\tr_l D^{\alpha}f= 0 \text{ for all } \alpha \in \N_l^n \text{ with } |\alpha|=r.
\end{align*}

\end{Remark}

\section{Refined localization spaces}
\label{section:refined}
\subsection{Definition of refined localization spaces}
In Section \ref{defdom} we introduced the spaces $\Ft$ which where used in Proposition \ref{decomptri}. However, if we look at the domains $\hyp$, we need to find a substitute. Actually, we have
\begin{align*}
 \Ft[\hyp] \cong \FR \text{ since } \overline{\Om}=\R^n \text{ and } \Ft[\partial \Om]=\{0\}
\end{align*}
at least when $s>0$ and hence for $1\leq p \leq \infty$ 
\begin{align*}
 \FR \hookrightarrow L_p(\R^n).
\end{align*}

So we take over the definition from \cite[Section 2.2.3]{Tri08}: Let $\Om$ be an arbitrary open domain and let $Q_{j,r}^0,Q_{j,r}^1$ be the Whitney cubes according to Section \ref{Whitney}. Let $\varrho=\{\varrho_{j,r}\}$ be a suitable resolution of unity, i.\,e.\
\begin{align}
\label{resun}
 supp \ \varrho_{j,r} \subset Q_{j,r}^1, \quad \|D^{\alpha} \varrho_{j,r} (x)\| \leq c_{\alpha} 2^{j|\alpha|}, x \in \Om, \alpha \in \N_0^n
\end{align}
for some $c_{\alpha}>0$ independent of $x,j,r$ and 
\begin{align*}
 \sum_{j=0}^{\infty} \sum_{r=1}^{M_j} \varrho_{j,r}(x)=1 \text{ if } x \in \Om.
\end{align*}

\begin{Definition}
\label{FrlocDef}
Let $\Om$ be an arbitrary open domain with $\Om\neq \R^n$ and let $\varrho=\{\varrho_{j,r}\}$ be the above introduced resolution of unity. Let $0\leq p < \infty$, $0<q\leq\infty$ and $s>\sigma_{p,q}$. Then
\begin{align*}
 \Frloc:=\left\{ f \in D'(\Om): \|f| \Frloc\|_{\varrho}< \infty\right\} 
\end{align*}
with
\begin{align*}
 \|f| \Frloc\|_{\varrho}:=\left(\sum_{j=0}^{\infty}\sum_{r=1}^{M_j}\| \varrho_{j,r} f|\FR\|^p\right)^{\frac{1}{p}}. 
\end{align*}
\end{Definition}
\begin{Remark}
 The definition of $\Frloc$ is independent of the choice of the re\-so\-lution of unity $\varrho=\{\varrho_{l,r}\}$, see \cite[Theorem 2.16]{Tri08}. There is also a definition given for $\Frloc$ when $s<0$ and $s=0$, but this is not necessary for our later considerations. 
\end{Remark}
\begin{Remark}
\label{rlocDense}
 The space $D(\Om)$ is dense in $\Frloc$ if $0<p<\infty, 0<q<\infty$. This can be seen as follows: Let $f \in \Frloc$ and take $J,R \in \N_0$ sufficiently large such that
\begin{align*}
 \left(\sum_{j\geq J}\sum_{r=1}^{M_j}\| \varrho_{j,r} f|\FR\|^p\right)^{\frac{1}{p}} < \eps
\end{align*}
and (if some $M_j=\infty$, this means if $\Om$ is unbounded)
\begin{align*}
 \left(\sum_{j< J}\sum_{r>R} \| \varrho_{j,r} f|\FR\|^p\right)^{\frac{1}{p}} < \eps.
\end{align*}
It holds
\begin{align}
\label{ressupp}
 \dist(supp \, \varrho_{j,r}, \Gamma) \gtrsim 2^{-j}
\end{align}
for all $j\leq J$ and $r \in \{1,\ldots,N_j\}	$. For $j\leq J$ and $r \leq R$ we take a sequence of functions $\{g_{j,r}^{n}\}_{n\in \N} \subset D(\R^n)$ approximating $\varrho_{j,r} f$ in $\|\cdot|\FR\|$. By \eqref{ressupp} and a pointwise multiplier argument, using Theorem \ref{pointwise}, we may assume that 
\begin{align*}
 \dist(x,\Gamma) \gtrsim 2^{-j} \text{ for } x \in supp \, g_{j,r}^{n} \text{ as well as } supp \, g_{j,r}^{n} \subset \Om,
\end{align*}
hence $g_{j,r}^{n} \in D(\Om)$.

Now we sum up:
\begin{align*}
 f- \sum_{j\leq J}\sum_{r=1}^{N_j} g_{j,r}^{n} = \sum_{j> J}\sum_{r=1}^{N_j} \varrho_{j,r}f + \sum_{j\leq J}\sum_{r > R} \varrho_{j,r}f+ \sum_{j\leq J}\sum_{r=1}^{R} \left(\varrho_{j,r}f-g_{j,r}^{n}\right). 
\end{align*}
The $\Frloc$-norm estimation of the three sums is a technical matter, using \eqref{resun} and pointwise multiplier observations similar to the concept in the proof of Theorem 2.16 (using Theorem 2.13) in \cite{Tri08}.
\end{Remark}

\subsection{Atomic decompositions and wavelet bases for refined localization spaces}

The following theorem gives an alternative approach to define $\Frloc$ which is nowadays maybe the more common way.
\begin{Theorem}[Wavelet basis for $\Frloc$]
\label{rlocwavelet}
Let $\Om$ be an arbitrary domain in $\R^n$ with $\Om\neq \R^n$. Let 
\begin{align*}
 0<p<\infty, 0<q<\infty, s>\sigma_{p,q} \text{ and } u>s. 
\end{align*}
Then there is an orthonormal u-wavelet basis 
\begin{align*}
\Phi=\left\{\Phi_r^j: j \in \mathbb{N}_0, r=1, \ldots, N_j\right\} \subset C^u (\Om)
\end{align*}
in $L_2(\Om)$ according to Definition \ref{u-basis} which is an interior u-Riesz basis (according to Definition \ref{u-Riesz}) for $\Frloc$ with the sequence space $\fO$. It holds
\begin{align*}
 \lambda_r^j(f)= 2^{jn/2} (f,\Phi_r^j).
\end{align*}
\end{Theorem}
\begin{Proof}
 This is a reformulation of \cite[Theorem 2.38]{Tri08}. The theorem is proven for the special $u$-wavelet basis constructed in \cite[Theorem 2.33]{Tri08} - emerging from Daubechies wavelets for function spaces on $\R^n$. The sequence space $\fO$ perfectly suits the construction there.  

 Furthermore, in \cite[Theorem 2.38]{Tri08} Triebel assumed $f \in L_{v}(\Om)$ with 
\begin{align*}
 \max(1,p)<v \leq \infty, s-\frac{n}{p}> -\frac{n}{v}
\end{align*}
but this is not necessary since automatically $f \in \FR$ if $f$ has a wavelet decomposition by the atomic decomposition Theorem  \ref{AtomicRepr}.

On the other hand, if $f$ belongs to $\Frloc$, then $f \in L_v(\R^n)$ by the Sobolev embedding from Proposition \ref{Sobolev} and $F_{v,2}^{0}(\R^n)=L_v(\R^n)$. 
\end{Proof}

\begin{Definition}[Atomic decompositions for $\Frloc$]
\label{Atomsrloc}
 Let $\Z^{\Om}$ be the interior collection as in Definition $\ref{sequence}$ and $x_r^j$ be the elements of $Z^{\Om}$, see also Section \ref{Whitney} and Remark \ref{Whitneyatom} for a possible construction. Let $K\geq 0$. A function $a_{j,r}:\R^n \rightarrow \mathbb{C}$ is called $K$-atom located at $x_r^j \in \Z^{\Om}$ if 
\begin{align*}
supp \ a_{j,r} & \subset B(x_r^j, C_1 \cdot 2^{-j})  \text{ and }\\
 \|a_{j,r} (2^{-j}\cdot)|\hold[K]\| &\leq C_2
\end{align*}
for suitable constants $C_1,C_2>0$ independent of $j,r$ such that 
\begin{align*}
 supp \ a_{j,r}  \subset \Om \text{ and } \dist(supp \ a_{j,r}, \partial \Om) \gtrsim 2^{-j}.
\end{align*}
\end{Definition}

\begin{Remark}
 This is the adaption of Definition \ref{Atoms} to domains $\Om$ and $\Frloc$. Because we always assume $s>\sigma_{p,q}$, we don't need moment conditions \eqref{Atom3}, but we could assume them, if we want. The constants $2^{-j(s-\frac{n}{p})}$ are now incorporated in the sequence spaces $f_{p,q}^s(\Z^{\Om})$, not within the atoms. 
\end{Remark}

\begin{Theorem}
\label{AtomicReprrloc}
 Let $0<p< \infty$, $0<q\leq \infty$, $s>\sigma_{p,q}$, $K>s$ and $\Z^{\Om}$ as in Definition \ref{sequence}. Then there is a $v \in (1,\infty)$ such that 
\begin{align}
\label{Lv}
 \max(1,p)<v \leq \infty, s-\frac{n}{p}> -\frac{n}{v}.
\end{align}
Then $f \in L_v(\Om)$ belongs to $\Frloc$ if 
\begin{align}
\label{repres}
 f=\sum_{j=0}^{\infty} \sum_{r=1}^{N_j} \lambda_{j,r} a_{j,r} \quad \text{with convergence in }  L_v(\Om)
\end{align}
and it holds
\begin{align*}
 \|f|\Frloc\| \lesssim \|\lambda|\fO\|.
\end{align*}
Here $a_{j,r}$ are $K$-atoms located at $x_{j,r} \in \Z^{\Om}$. Furthermore, there is a suitable $\Z^{\Om}$ such that for every $f \in \Frloc$ there exists a representation \eqref{repres}
with  
\begin{align*}
 \|\lambda|\fO\| \lesssim \|f|\Frloc\|.
\end{align*}
\end{Theorem}
\begin{proof}
 At first, we prove the existence of such a decomposition 
 \begin{align*} 
 f=\sum_{j=0}^{\infty} \sum_{r=1}^{N_j} \lambda_{j,r} a_{j,r} 
\end{align*}
with
\begin{align*}
  \|\lambda|\fO\| \lesssim \|f|\Frloc\|.
\end{align*}
But this has already been done in Theorem \ref{rlocwavelet}, which is a reference to \cite[Theorem 2.38]{Tri08}. Triebel proved wavelet decompositions of $\Frloc$ with the same sequence space $\fO$. He used a special grid $\Z^{\Om}$ covered by Definition \ref{sequence}. The wavelet blocks can be choosen as functions in $\hold[K]$ for an arbitrary $K>0$, but not as functions in $C^{\infty}(\R^n)$. Furthermore, if $f \in \Frloc$, then $f \in L_v(\Om)$ by the Sobolev embedding from Proposition \ref{Sobolev} and $F_{v,2}^{0}(\R^n)=L_v(\R^n)$.

For the second part of the proof we have to show that for every atomic decomposition 
\begin{align*}
  f=\sum_{j=0}^{\infty} \sum_{r=1}^{N_j} \lambda_{j,r} a_{j,r} 
\end{align*}
it holds
\begin{align*}
\|f|\Frloc\| \lesssim  \|\lambda|\fO\|. 
\end{align*}
The $L_v(\Om)$-convergence follows automatically.

Let $\{\varrho_{j,r}\}$ be a suitable resolution of unity adapted to the Whitney cubes $Q_{j,r}^0,Q_{j,r}^1$, hence
\begin{align*}
 supp \ \varrho_{j,r} \subset Q_{j,r}^1, \quad \|D^{\alpha} \varrho_{j,r} (x)\| \leq c_{\alpha} 2^{j|\alpha|}, x \in \Om, \alpha \in \N_0^n
\end{align*}
for some $c_{\alpha}>0$ independent of $x,j,r$ and 
\begin{align*}
 \sum_{j=0}^{\infty} \sum_{r=1}^{M_j} \varrho_{j,r}(x)=1 \text{ if } x \in \Om.
\end{align*}
Then an atomic decomposition of $f \in \Frloc$ by $K$-atoms results in an atomic decomposition of $(\varrho_{j,r} f)(2^{-j}\cdot) \in \FR$. We have
\begin{align*}
  \left(\varrho_{j,r} f \right)(2^{-j}\cdot)=\sum_{j'=j-L'}^{\infty} \sum_{r'=1}^{N_{j'}} \lambda_{j',r'} \varrho_{j,r}(2^{-j}\cdot) a_{j',r'}(2^{-j}\cdot)
\end{align*}
for a constant $L'>0$ independent of $j$ and $r$ resulting from the support properties of $\varrho_{j,r}$ and $a_{j',r'}$.
The functions $a_{j',r'}(2^{-j}\cdot)$ are atoms located at $x_{j'}^{r'}$ as well (see Lemma \ref{Dilation}) - but with a support of length $\sim 2^{-j+j'}\lesssim 1$. By construction of the resolution of unity we have
\begin{align*}
 \|\varrho_{j,r}(2^{-j}\cdot)|\hold[K]\|\lesssim 1
\end{align*}
independent of $j$ and $r$. Using Lemma \ref{ProdAtom} it holds that $\left(\varrho_{j,r}a_{j',r'}\right)(2^{-j}\cdot)$ is a $K$-atom located at $x_{j'}^{r'}$, but with bigger support (resulting in a constant $2^{j(s-\frac{n}{p})}$ in the sequence space norm). We have
\begin{align*}
\|\left(\varrho_{j,r}a_{j',r'}\right)(2^{-j}\cdot)|\hold[K]\| \lesssim 1 .
\end{align*}
Putting this together and using the atomic decomposition theorem for $\FR$ (see Theorem \ref{AtomicRepr}) we obtain
\begin{align*}
 \| \left(\varrho_{j,r} f \right)(2^{-j}\cdot)|\FR\| \lesssim 2^{j(s-\frac{n}{p})} \|\lambda^{j,r}|\fO\|,
\end{align*}
where $\lambda^{j,r}$ consists only of the $\lambda_{j',r'}$ such that the support of $a_{j',r'}$ and $\varrho_{j,r}$ overlap. Using the homogeneity Property \ref{homogen} we have
\begin{align*}
 \| \left(\varrho_{j,r} f \right)|\FR\| \lesssim \|\lambda^{j,r}|\fO\|
\end{align*}
with constants independend of $j$ and $r$. 

Furthermore, by the support condition of $\varrho_{j,r}$ and the construction of the atoms only finitely many of the supports of $\varrho_{j,r}$ overlap (in the parameters $j,r$) - and hence for every $a_{j',r'}$ there are only finitely many $\varrho_{j,r}$ where the support of $\varrho_{j,r}$ overlaps with $a_{j',r'}$. 

Hence we have
\begin{align*}
 \|f|\Frloc\|&= \left(\sum_{j=0}^{\infty}\sum_{r=1}^{M_j}\| \varrho_{j,r} f|\FR\|^p\right)^{\frac{1}{p}} \\ 
  &\lesssim \left(\sum_{j=0}^{\infty}\sum_{r=1}^{M_j} \|\lambda^{r,j}|\fO\|^p\right)^{\frac{1}{p}} \\
  &\lesssim \|\lambda|\fO\|.
\end{align*}

\end{proof}

\subsection{Properties and alternative characterizations of refined localization spaces}

\begin{Definition}
Let $\Om$ be a domain with $\Om\neq \R^n$, $\Gamma=\partial \Om$, $d(x)=\dist(x,\Gamma)$. Then we define
\begin{align*}
 \delta(x)=\min(d(x),1).
\end{align*}
\end{Definition}
\begin{Proposition}
 \label{rlocDiff}
Let $\Om$ be an arbitrary domain in $\R^n$. Let 
\begin{align*}
 0<p<\infty, 0<q<\infty, s-r>\sigma_{p,q}, \alpha \in \N^n   \text{ with }  |\alpha|=r. 
\end{align*}
It holds: If $f$ belongs to $\Frloc$, then $D^{\alpha} f$ belongs to $F_{p,q}^{s-r,\rloc}(\Om)$ with 
\begin{align*}
 \|D^{\alpha} f|F_{p,q}^{s-r,\rloc}(\Om)\|\lesssim \|f|\Frloc\| 
\end{align*}
independently of $f$.
\end{Proposition}
\begin{proof}
It suffices to prove the Proposition for $|\alpha|=1$ -- the general assertion follows by induction. We will give two different proofs - one using atomic decompositions of $\Frloc$, see Theorem \ref{AtomicReprrloc}, and one direct proof, using the homogeneity Property \ref{homogen}.

\textbf{First proof:} Let $|\alpha|=1$. If $f \in \Frloc$, then there is an atomic representation of $f$ by Theorem \ref{AtomicReprrloc} with
\begin{align*}
 f=\sum_{j=0}^{\infty} \sum_{r=1}^{N_j} \lambda_{j,r} a_{j,r}
\end{align*}
and $\lambda \in \fO$. Here $a_{j,r}$ are $K$-atoms (with a suitably large $K$) located at $x_{j}^r \in \Z^{\Om}$. So we find an atomic decomposition for $D^{\alpha}f$ through
\begin{align*}
 D^{\alpha} f=\sum_{j=0}^{\infty} \sum_{r=1}^{N_j} \lambda_{j,r} 2^{j} 2^{-j} D^{\alpha} a_{j,r}.
\end{align*}
By the Definition \ref{AtomicReprrloc} of atoms we observe that $2^{-j} D^{\alpha} a_{j,r}$ are $(K-1)$-atoms located at $x_{j}^r$. The coefficients in front of the atoms are $\lambda_{j,r}':=2^{j}\lambda_{j,r}$. Hence by the atomic decomposition Theorem \ref{AtomicReprrloc} we have
\begin{align*}
 \|D^{\alpha}f|F_{p,q}^{s-1,\rloc}(\Om)\|\lesssim \|2^{j}\lambda_{j,r}|f_{p,q}^{s-1}(\Z^{\Om})\| \sim 
 \|\lambda_{j,r}|\fO\| \lesssim \|f|\Frloc\|.
\end{align*}
Here we needed $s-1>\sigma_{p,q}$, hence in general we need $s-r>\sigma_{p,q}$.

\textbf{Second proof:} Let $|\alpha|=1$. If $f \in \Frloc$, then $\varrho_{j,r} f \in \FR$ for $j \in \N_0, r\in \{1,\ldots,M_j\}$ and 
\begin{align*}
 D^{\alpha} (\varrho_{j,r} f)=(D^{\alpha} \varrho_{j,r}) \cdot f+\varrho_{j,r}\cdot D^{\alpha}f  \in \FR[s-1].
\end{align*}
By triangle inequality and classical differentation properties of $\FR$ we get
\begin{align*}
 \|\varrho_{j,r} D^{\alpha} f|\FR[s-1]\| \lesssim \|(D^{\alpha} \varrho_{j,r}) \cdot f|\FR[s-1]\| + \|\varrho_{j,r} f|\FR\|.
\end{align*}
To prove the proposition, it suffices to estimate the $p$-sum of the first terms on the RHS by $\|f|\Frloc\|$. It holds
\begin{align*}
 (D^{\alpha} \varrho_{j,r}) \cdot f =  (D^{\alpha} \varrho_{j,r}) \cdot \sum_{|j-j'|\leq c}\sum_{r'} (\varrho_{j',r'} f),
\end{align*}
where $c$ is a constant independent of $j$ and $r$. Furthermore, the number of summands in the sum over $r'$ can also be estimated by a constant independent of $j$ and $r$. This follows from the construction of the resolution of unity and the Whitney decomposition, see \eqref{resun}.   

Now we make use of the homogeneity property, see Proposition \ref{homogen}, and our theorem on pointwise multipliers, see Theorem \ref{pointwise}. Let $\rho > \max(s,\sigma_{p,q}-s)$. We get
\begin{align*}
 \|(D^{\alpha} &\varrho_{j,r}) \cdot f|\FR[s-1]\|  \\
&\lesssim  \sum_{|j-j'|\leq c}\sum_{r'} \|(D^{\alpha} \varrho_{j,r}) \cdot (\varrho_{j',r'} f)|\FR[s-1]\| \\
& \sim 2^{j(s-1-\frac{n}{p})} \sum_{|j-j'|\leq c}\sum_{r'} \|(D^{\alpha} \varrho_{j,r})(2^{-j}\cdot) \cdot (\varrho_{j',r'} f)(2^{-j}\cdot)|\FR[s-1]\| \\
& \sim 2^{j(s-\frac{n}{p})} \sum_{|j-j'|\leq c}\sum_{r'} \|D^{\alpha} \left(\varrho_{j,r}(2^{-j}\cdot)\right) \cdot (\varrho_{j',r'} f)(2^{-j}\cdot)|\FR[s-1]\| \\
& \lesssim 2^{j(s-\frac{n}{p})} \|D^{\alpha} \left(\varrho_{j,r}(2^{-j}\cdot)\right)|\hold[\rho]\| \cdot \sum_{|j-j'|\leq c}\sum_{r'} \|(\varrho_{j',r'} f)(2^{-j}\cdot)|\FR[s-1]\| \\
& \lesssim 2^{j(s-\frac{n}{p})} \sum_{|j-j'|\leq c}\sum_{r'} \|(\varrho_{j',r'} f)(2^{-j}\cdot)|\FR[s-1]\| \\
& \lesssim 2^{j(s-\frac{n}{p})} \sum_{|j-j'|\leq c}\sum_{r'} \|(\varrho_{j',r'} f)(2^{-j}\cdot)|\FR[s]\| \\
& \sim \sum_{|j-j'|\leq c}\sum_{r'} \|\varrho_{j',r'} f|\FR[s]\|,
\end{align*}
where the constants do not depend on $r$ or $j$, using property \eqref{resun} of the resolution of unity. Summing up over $j$ and $r$ gives
\begin{align*}
 \sum_{j=0}^{\infty}\sum_{r=1}^{M_j}\|(D^{\alpha} \varrho_{j,r}) \cdot f|\FR[s-1]\|^p \leq c \sum_{j=0}^{\infty}\sum_{r=1}^{M_j}\|\varrho_{j,r} \cdot f|\FR[s]\|^p=\|f|\Frloc\|^p.
\end{align*}
This is what we wanted to prove.

\end{proof}
\begin{Remark}
 For $\Om=\R^n$ there is the converse inequality
\begin{align*}
 \|f|\FR\| \leq c \sum_{|\alpha|\leq r} \|D^{\alpha} f|\FR[s-r]\|. 
\end{align*}
Such an inequality cannot hold for $\Frloc$ on arbitrary domains $\Om$: Consider a $C^{\infty}$-domain $\Om$. Let $1\leq p< \infty, 1\leq q <\infty$ and $0<s<\frac{1}{p}$. Then by \cite[Section 5.24]{Tri01} and \cite[Proposition 3.10]{Tri08} we have
\begin{align}
\label{rlocsmall}
 \FO=\Ft=\Frloc.
\end{align}
On the other hand $s+1>\frac{1}{p}$. Then $\tr f|_{\partial \Om} = 0$ for $f \in F_{p,q}^{s+1,\rloc}(\Om)=\tilde{F\,}\!_{p,q}^{s+1}(\Om)$. Hence if $f \in F_{p,q}^{s+1,\rloc}(\Om)$, then $f+c \notin F_{p,q}^{s+1,\rloc}(\Om)$ for a constant $c\neq 0$. 

But $D^{\alpha} f \in \Frloc$ if, and only if, $D^{\alpha} (f+c) \in \Frloc$ for $|\alpha|\leq 1$: For $|\alpha|=1$ this is obvious and for $|\alpha|=0$ this follows from \eqref{rlocsmall} considering the definition of $\FO$.  

\end{Remark}

\begin{Proposition}
\label{rlocequi}
 Let $\Om$ be an arbitrary domain in $\R^n$ with $\Om\neq \R^n$, let 
\begin{align*}
0<p<\infty, 0< q < \infty, s>\sigma_{p,q}.
 \end{align*}
Then $f \in \Frloc$ if, and only if,
\begin{align*}
 \|f|\FO\| + \| \delta^{-s}(\cdot)f|L_p(\Om)\| < \infty 
\end{align*}
(equivalent norms).
\end{Proposition}
\begin{proof}

\textit{First step:} Let $f \in \Frloc$. Then there is a wavelet characterization of $f$ by Theorem \ref{rlocwavelet}. This results in an atomic decomposition of $f \in \FR$. Hence by Theorem \ref{AtomicRepr} we have
\begin{align*}
 \|f|\FO\| \lesssim \|f|\Frloc\|.
\end{align*}
Furthermore, let $\varrho=\{\varrho_{j,r}\}$ be the resolution of unity adapted to the Whitney cubes $Q_{j,r}^1$ as introduced at the beginning of Section \ref{section:refined}. Since $s>\sigma_{p,q}\geq \sigma_p$ we have
\begin{align*}
 \FR \hookrightarrow L_p(\R^n) \cap L_{\max(1,p)}(\R^n).
\end{align*}
We use the homogeneity property from Proposition \ref{homogen} to get
\begin{align*}
\|\delta^{-s} \varrho_{j,r} f|L_p(\R^n)\|&\sim 2^{js} \|\varrho_{j,r} f\|L_p(\R^n)\| \\
&\sim 2^{j(s-\frac{n}{p})} \|\left(\varrho_{j,r} f\right)(2^{-j}\cdot)\|L_p(\R^n)\| \\
&\lesssim 2^{j(s-\frac{n}{p})}  \|\left(\varrho_{j,r} f\right)(2^{-j}\cdot)\|\FR\| \\	
&\sim \|\varrho_{j,r} f\|\FR\|  
\end{align*}
with constants independent of $j$ and $r$. Here we used
\begin{align*}
 d(x) \sim 2^{-j} \text{ for } x \in supp \ \varrho_{j,r}
\end{align*}
for $j \in \N$ as well as $r \in \{1,\ldots,M_j\}$ and
\begin{align*}
 d(x) \gtrsim 1 \text{ for } x \in supp \ \varrho_{0,r}
\end{align*}
for $r \in \{1,\ldots,M_j\}.$
 
Thus we arrive at
\begin{align*}
 \|\delta^{-s}f|L_p(\R^n)\| &\sim \left(\sum_{j=0}^{\infty}\sum_{r=1}^{M_j}\| \delta^{-s} \varrho_{j,r} f|L_p(\R^n)\|^p\right)^{\frac{1}{p}} \\
&\lesssim \left(\sum_{j=0}^{\infty}\sum_{r=1}^{M_j}\| \varrho_{j,r} f|\FR\|^p\right)^{\frac{1}{p}} \\
&=\|f|\Frloc\|.
\end{align*}

\textit{Second step:} Let $f \in \FO$ with $\delta^{-s}(\cdot)f \in L_p(\Om)$. Then $f \in L_v(\Om)$ with $v$ as in \eqref{Lv}. Hence we can decompose it by Theorem 2.36 in \cite{Tri08} into 
\begin{align*}
f=\sum_{j=0}^{\infty}\sum_{r=1}^{N_j} \lambda_r^j(f) 2^{-\frac{jn}{2}}\Phi_r^j
\end{align*}
with $\lambda_r^j(f) \in f_{v,2}^{0}(\Z^{\Om})$.
We split $f$ into
\begin{align*}
 f=f_1+f_2
\end{align*}
where $f_1$ collects the boundary wavelets (without moment conditions) with  
\begin{align}
\label{f1dist}
\dist(supp\ \Phi_r^{j,1},\Gamma) \sim 2^{-j}
\end{align}
and $f_2$ collects the interior wavelets (with moment conditions) with
\begin{align*}
\dist(supp\ \Phi_r^{j,2},\Gamma) \gtrsim 2^{-j},
\end{align*}
so
\begin{align*}
f_2=\sum_{j=0}^{\infty}\sum_{r=1}^{N_j} \lambda_r^{j,2}(f) 2^{-\frac{jn}{2}}\Phi_r^{j,2}.
\end{align*}
The wavelets $\Phi_r^{j,2}$ fulfil the appropriate derivative and moment conditions \eqref{Atom2} and \eqref{Atom3}. Furthermore, $\Phi_r^{j,2}$ and $\Phi_{r'}^{j',1}$ are orthogonal to each other.
Hence by the local mean Theorem 1.15 from \cite{Tri08} used for the orthogonal wavelets $\Phi_r^{j,2}$ we get
\begin{align*}
 \|\lambda_r^{j,2}(f)|\fO\|&=2^{jn/2} \| (f,\Phi_{r}^{j,2})|\fO\|=2^{jn/2} \| (\tilde{f},\Phi_{r}^{j,2})|\fO\| \\
&\lesssim \|\tilde{f}|\FR\|,
\end{align*}
where $\tilde{f}$ is an arbitrary extension of $f$ from $\Om$ to $\R^n$ (the values outside of $\Om$ do not matter for $(f,\Phi_{r}^{j,2})$). Taking the infimum over the $\FR$-norms of the extensions of $f$, we get the first desired result
\begin{align*}
 \|\lambda_r^{j,2}(f)|\fO\| \lesssim \|f|\FO\|.
\end{align*}
Hence $f_2 \in \Frloc$ by the wavelet Theorem \ref{rlocwavelet} for $\Frloc$. By the first step this shows
\begin{align*}
 \|\delta^{-s}f_2|L_p(\Om)\| \lesssim \|f|\FO\|.  
\end{align*}
Using triangle inequality this leads to
\begin{align}
\label{Lpd}
 \|\delta^{-s}f_1|L_p(\Om)\| \lesssim \|f|\FO\| + \|\delta^{-s}f|L_p(\Om)\|.
\end{align}
Furthermore, $\|2^{js} \lambda_r^{j,1}|f_{p,q}^{0}(\Z^{\Om})\|$ is independent of $q$ since boundary wavelets do not overlap too much - there is a constant $C>0$ such that for all $x \in \Om$ not more than $C$ boundary wavelets are supported at $x$. This argument was also used in the proof of Theorem 2.28 in \cite{Tri08} with a reference to \cite[Remark 2.25]{Tri08}. We have 
\begin{align*}
  \|\lambda_r^{j,1}|f_{p,q}^{s}(\Z^{\Om})\|
&\sim \|2^{js} \lambda_r^{j,1}|f_{p,q}^{0}(\Z^{\Om})\| \\
&\sim \|2^{js} \lambda_r^{j,1}|f_{p,p}^{0}(\Z^{\Om})\| \\
  &\sim \|\delta^{-s}f_1|L_p(\Om)\|
  \end{align*}
by direct calculation of the $L_p(\Om)$-norm and \eqref{f1dist}. Now, using \eqref{Lpd} we have shown
\begin{align*}
  \|\lambda_r^{j,1}|f_{p,q}^{s}(\Z^{\Om})\| \lesssim \|f|\FO\| + \|\delta^{-s}f|L_p(\Om)\|,
\end{align*}
which proves that also $f_1 \in \Frloc$ by the wavelet Theorem \ref{rlocwavelet}. Hence $f=f_1+f_2 \in \Frloc$.
\end{proof}

For special bounded domains we have Proposition 3.10 from \cite{Tri08} in contrast to the observations at the beginning of this section. 
\begin{Proposition}
\label{specialrloc}
 Let $\Om$ be an $E$-thick domain in $\R^n$. Let
\begin{align*}
 0<p<\infty, 0<	q < \infty, s>\sigma_{p,q}.  
\end{align*}
Then
\begin{align*}
 \Frloc=\Ft.
\end{align*}
\end{Proposition}

\section{Reinforced function spaces: Traces and wavelet-friendly extension oper\-a\-tors}
\label{traces}
As stated earlier Proposition \ref{decomptri} cannot hold when $r=s-\frac{n-l}{p} \in \N$. The aim of the following sections is to find a substitute, where we replace $Q_l$ by $\R^n \setminus \R^l$ and $Q_l^n$ by $\R^n$ for convenience. We have to care about traces and wavelet-friendly extension operators now from the point of our newly introduced function spaces $\Frinf[\hyp]$ instead of $\FR$.

As before, let
\begin{align*}
\N_l^n=\left\{ \alpha=(\alpha_1,\ldots,\alpha_n) \in \N_0^n: \alpha_1=\ldots=\alpha_l=0\right\}
\end{align*}
and $x=(y,z) \in \R^l \times \R^{n-l}$.

Let $\tr_l$ be the trace operator 
\begin{align*}
 \tr_l: f(x) \mapsto f(y,0), \text{ for } f \in \AR,
\end{align*}
on $\R^l$ (if it exists) and
\begin{align*}
  \tr_{l}^r: f \mapsto \left\{\tr_l D^{\alpha} f: \alpha \in \N_l^n, |\alpha|\leq r \right\}
\end{align*}
which is the composite map of all traces of derivatives of $f$ onto $\R^l$ with order not bigger than $r$ and derivatives only perpendicular to $\R^l$. For further informations on traces and the historical background see \cite[Section 5.11]{Tri08} or {\cite[Section 4.4]{Tri92}.

\begin{Proposition}[Traces]
\label{Tra}
 Let $l \in \N_0$, $n \in \N$ with $l<n$ and $r \in \N_0$. Let $1 \leq p<\infty, 0<q<\infty$ and
\begin{align*}
s>r + \frac{n-l}{p}.
\end{align*}
Then
\begin{align*}
   \tr_l^r:& \Frinf[\R^n] \rightarrow \prod_{\underset{|\alpha|\leq r}{\alpha \in \N_l^n}} F_{p,p}^{s-\frac{n-l}{p}-|\alpha|}(\R^l).
\end{align*}
\end{Proposition}
\begin{Proof}
 This follows from
\begin{align*}
 \Frinf[\R^n] \hookrightarrow \FR, 
\end{align*}
Proposition 6.17 in \cite{Tri08} and $F_{p,p}^{s}(\R^l)=B_{p,p}^{s}(\R^l)$. The replacement of $Q_l$ by $\R_l$ is immaterial. For the proof one uses atomic decompositions of $\FR$.
\end{Proof}
We now take over the introduction of the wavelet-friendly extension operator after Remark 6.18 of \cite{Tri08}, adopted to our situation - we replace the sets $Q_l$ by $\R^l$ and use wavelet bases of $L_2(\R^l)$ instead of $L_2(Q_l)$. The following observations are crucial to what follows. In a special case this was considered in Step 3 of the proof of Theorem 6.46 in \cite{Tri08} and we will take over the idea.

Let $u \in \N$ and let
\begin{align*}
 \{ \Phi_m^j: j \in \N_0, m \in \Z^l \} \subset C^u(\R^l)
\end{align*}
be a $u$-wavelet basis in $L_2(\R^l)$ according to Theorem \ref{orthbases} which is an interior u-Riesz basis for $\FO[\R^l]$ by Proposition \ref{waveletre}. Hence, every $g \in \FO[\R^l]$ can be represented as a wavelet expansion, namely
\begin{align*}
 g= \sum_{j=0}^{\infty}\sum_{m \in \Z^l} \lambda_m^j(g) 2^{-\frac{jl}{2}}\Phi_m^j
\end{align*}
with
\begin{align*}
  \lambda_m^j(g)= 2^{jl/2} \int_{\R^l} g(y) \Phi_m^j(y) \ dy
\end{align*}
and an isomorphic map
\begin{align*}
 g \mapsto \{\lambda_m^j(g)\} 
\end{align*}
of $\FO[\R^l]$ onto $\fO[\Z^l]$. Let 
\begin{align*}
 \chi^* \in D(\R), \quad supp \ \chi^* \subset \{z \in \R: |z|\leq 2 \}, \quad \chi^*(z)=1 \text{ if } |z|\leq 1
\end{align*}
and
\begin{align}
 \label{chi}
\chi(z)=\chi^*(z_1) \cdot \ldots \cdot \chi^*(z_{n-l})
\end{align}
for all $z \in \R^{n-l}$. Furthermore, it is possible to choose $\chi$ such that it fulfils as many moment conditions
\begin{align}
\label{momenter}
 \int_{\R^{n-l}} \chi(z) z^{\beta} \ dz = 0 \text{ if } |\beta|\leq L 
\end{align}
as we want. Now we define $n$-dimensional functions (extensions) by
\begin{align*}
 \Phi_m^{j,\alpha}(x) = 2^{j|\alpha|} z^{\alpha} \chi(2^j z)\,  2^{(n-l)j/2} \, \Phi_m^j(y) \quad \text{for } \alpha \in \N_l^n.
\end{align*}
It is easy to see that
\begin{align}
\label{extid}
 \tr_l D^{\alpha}  \Phi_m^{j,\alpha}(y) = 2^{j|\alpha|} \cdot \alpha! \cdot 2^{(n-l)j/2} \Phi_m^j(y)= c_{j,\alpha} \Phi_m^j(y)
\end{align}
and
\begin{align}
\label{extid2}
 \tr_l D^{\alpha}  \Phi_m^{j,\beta}(y) = 0 \quad \text{ for } \alpha \neq \beta \in \N_{l,0}^n
\end{align}
for $y \in \R^l$ having in mind the factor $z^{\alpha}$. This is the crucial property giving the possibility to construct the extension operator by
\begin{align}
\label{Ext}
\begin{split}
 g&= \Ext_l^{r,u} \{g_{\alpha}: \alpha \in \N_l^n, |\alpha|\leq r \} (x) \\
  &= \sum_{\underset{|\alpha|\leq r}{\alpha \in \N_l^n}} \sum_{j=0}^{\infty}\sum_{m \in \Z^l} \frac{1}{\alpha!} \lambda_m^j(g_{\alpha}) 2^{-j|\alpha|} 2^{-\frac{jn}{2}}\cdot\Phi_m^{j,\alpha}(x)
\end{split}
\end{align}
In the following Proposition, we will consider extensions from $\R^l$ to $\R^n$ and consider Hardy inequalities (reinforce properties) at a boundary $\R^{l_1}$ with $l\leq l_1<n$. This is always to be understood as
\begin{align*}
 \R^l&=\{(x_1,\ldots,x_n) \in \R^n:x_{l+1}=\ldots=x_{n}=0\} \\
\subset \R^{l_1}&=\{(x_1,\ldots,x_n) \in \R^n:x_{l_1+1}=\ldots=x_{n}=0\}.
\end{align*}

\begin{Proposition}[Wavelet-friendly extension operators]
\label{extwavelet}
Let $n \in \N, l, r \in \N_0$ and $l<n$. Let $\{\Phi_r^j\}$ be an orthonormal u-wavelet basis in $L_2(\R^l)$. Let $\chi$ be as in \eqref{chi} with $L \in \N_0$ sufficiently large in dependence of $q$. Let $\Ext_l^{r,u}$ be given by \eqref{Ext} and
\begin{align*}
 1\leq p < \infty, 1\leq q\leq \infty \text{ and } u>s>r+\frac{n-l}{p} 
\end{align*}
Then
\begin{align*}
 \Ext_l^{r,u}: \{ g_{\alpha}: \alpha \in \N_l^n, |\alpha|\leq r \} \mapsto g
\end{align*}
is an extension operator  
\begin{align*}
 \Ext_l^{r,u} : \prod_{\underset{|\alpha|\leq r}{\alpha \in \N_l^n}} F_{p,p}^{s-\frac{n-l}{p}-|\alpha|}(\R^l) \hookrightarrow \Frinf[\R^n \setminus \R^{l}]
\end{align*}
with 
\begin{align}
\label{idprop}
 \tr_l^r \circ \Ext_l^{r,u} =id, \text{ identity on } \prod_{\underset{|\alpha|\leq r}{\alpha \in \N_l^n}} F_{p,p}^{s-\frac{n-l}{p}-|\alpha|}(\R^l).
\end{align}
Furthermore, let $l_1 \in \N_0$ with $l \leq l_1 < n$ and 
\begin{align*}
 r_1:=s-\frac{n-l_1}{p} \in \N.
\end{align*}
Then 
\begin{align}
\label{distext}
d_{l_1}^{-\frac{n-l_1}{p}} \cdot  D^{\beta} \Ext_l^{r,u} g \in L_p((\R^n \setminus \R^{l_1})_{\eps}) \text{ for all } \beta \in \N_{l_1}^n \text{ with } |\beta|=r_1,
\end{align}
(with a suitable norm estimate) where $d_{l_1}(x)$ is the distance of $x \in \R^n$ to $\R^{l_1}$. 
\end{Proposition}
\begin{Proof}
 At first, one shows that $\Ext_l^{r,u}$ maps to $\FR$. This is nearly the same situation as in \cite[Theorem 6.19]{Tri08}. The idea is the extension of wavelets on $\R^l$, which are special kind of atoms, to atoms of $\R^n$. If moment conditions are necessary, they can be accomplished by using moment conditions of $\chi$ (see \eqref{momenter}). Then $\eqref{Ext}$ gives an atomic decomposition of $\Ext_l^{r,u} g$ and hence $\Ext_l^{r,u} g$ belongs to $\FR$. 

Furthermore, to show that $\Ext_l^{r,u} g\in \Frinf[\hyp]$ as well as $\eqref{distext}$ and to show the continuity of the operator, it suffices to prove
\begin{align*}
\int_{(\R^n \setminus \R^{l_1})_{\eps}}  |D^{\beta} \Ext_l^{r,u} g|^p  \frac{ dx}{d^{n-l}(x)} \lesssim \sum_{\underset{|\alpha|\leq r}{\alpha \in \N_{l}^n}} \|g_{\alpha}|F_{p,p}^{s-\frac{n-l}{p}-|\alpha|}(\R^l)\|^p \text{ for } \beta \in \N_{l_1}^n, |\beta|=r_1
\end{align*}
for $s-\frac{n-l_1}{p}=r_1 \in \N$ and $l \leq l_1 <n$, since only then this additional property is required for $\Frinf[\hyp]$. 

But this follows from the equivalent characterization for the refined localization spaces in Proposition \ref{rlocequi} if we show that $D^{\beta} \Ext_l^{r,u} g$ belongs to $F_{p,p}^{\frac{n-l_1}{p},\rloc}(\R^n \setminus \R^{l_1}) \text{ for } \beta \in \N_{l_1}^n, |\beta|=r_1$ and 
\begin{align*}
\|D^{\beta} \Ext_l^{r,u} g|F_{p,p}^{\frac{n-l_1}{p},\rloc}(\R^n \setminus \R^{l_1})\| \lesssim \sum_{\underset{|\alpha|\leq r}{\alpha \in \N_{l}^n}} \|g_{\alpha}|F_{p,p}^{s-\frac{n-l}{p}-|\alpha|}(\R^l)\|^p.
\end{align*}
Let $\beta \in \N_{l_1}^n$, $|\beta|=r_1>0$ and $\beta=(0,\beta') \in \N^l \times \N^{n-l}$. Furthermore, $\beta'=(0,\beta'') \in \N^{l_1-l} \times \N^{n-l_1}$ since $\beta \in \N_{l_1}^n$. Let $z=(z_1,z_2) \in \R^{l_1-l} \times \R^{n-l_1}$ with the modification $z=z_2$ if $l_1=l$. 
The function $\chi$ from \eqref{chi} is a product of 1-dimensional functions $\chi^*$ (equal to $1$ near $0$) and hence constant on the segments $\{z \in \R^{n-l}: z_1=z_0,|z_2| \leq 1 \}$ with a fixed $z_0 \in \R^{l_1-l}$. Then
\begin{align*}
 D^{\beta'} \left( z^{\alpha}\chi(2^j z) \right)= 0 \text{ for } |z_2| \leq 2^{-j} \text{ since } |\alpha|\leq r < s-\frac{n-l}{p}\leq s-\frac{n-l_1}{p}=|\beta'|.
\end{align*}
This leads to 
\begin{align}
\label{extd2}
 supp \ D^{\beta} \Phi_m^{j,\alpha}  \subset \R^{l_1} \times \{z \in \R^{n-l_1}: 2^{-j}\leq |z|\leq 2^{-j+1} \}
\end{align}
and together with the support properties of $\chi$ this shows
\begin{align}
\label{extd}
 \dist(x,\R^{l_1}) \sim 2^{-j} \text{ for } x \in  supp \ D^{\beta} \Phi_m	^{j,\alpha}. 
\end{align}
Thus we proved that $2^{-jr_1} D^{\beta} \Phi_m^{j,\alpha}$ for $|\beta|=r_1$ behave like atoms in  $F_{p,p}^{\frac{n-l_1}{p},\rloc}(\R^n \setminus \R^{l_1})$, see Definition \ref{Atomsrloc} and Theorem \ref{AtomicReprrloc} - here we have $\sigma_{p,q}=0$ and $\frac{n-l_1}{p}>0$. 

\begin{center}
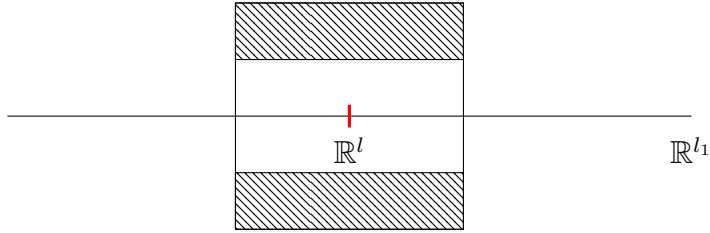

 \begin{tikzpicture}[scale=3]
 \def\u{0.03};
  \draw (-1,0) -- (2,0);
\draw[red,very thick] (0.5,-0.05)--(0.5,0.05);
\draw[very thick] (0.5,-0.15) node {$\R^l$};
\draw (2,-0.15) node {$\R^{l_1}$};
\draw (1,0.5)--(0,0.5)--(0,-0.5,0)--(1,-0.5)--cycle;
\draw[pattern=north west lines] (1,0.25)--(0,0.25)--(0,0.5)--(1,0.5)--cycle;
\draw[pattern=north west lines] (1,-0.25)--(0,-0.25)--(0,-0.5)--(1,-0.5)--cycle;
\end{tikzpicture}
 \captionof{figure}{Support of the derivatives of $\Ext_l^{r,u} g$ of order $r_1$ at $\R^{l_1}$ (shaded)}
   \label{reinfext}
\end{center}

By direct calculation we arrive at 
\begin{align*}
 D^{\beta} \Ext_l^{r,u} g =\sum_{\underset{|\alpha|\leq r}{\alpha \in \N_l^n}} \sum_{j=0}^{\infty}\sum_{m \in \Z^l} \frac{1}{\alpha!} \lambda_m^j(g_{\alpha}) 2^{-j|\alpha|}  2^{jr_1} \cdot 2^{-jr_1} 2^{-\frac{jn}{2}} D^{\beta}\Phi_m^{j,\alpha}(x)
\end{align*}
with ($s=\frac{n-l_1}{p}+r_1$)
\begin{align*}
 \|2^{-j|\alpha|}2^{jr_1}\lambda_m^j(g_{\alpha}) |f_{p,p}^{\frac{n-l_1}{p}}(\Z^{\R^n \setminus \R^{l_1}})\| & \sim  \|2^{-j|\alpha|}\lambda_m^j(g_{\alpha})|f_{p,p}^{s}(\Z^n)\| \\
&\sim \|2^{-j|\alpha|}\lambda_m^j(g_{\alpha})|f_{p,p}^{s-\frac{n-l}{p}}(\Z^l)\| \\
&\lesssim \sum_{\underset{|\alpha|\leq r}{\alpha \in \N_l^n}} \|g_{\alpha}|F_{p,p}^{s-\frac{n-l}{p}-|\alpha|}(\R^l)\|.
\end{align*}
The estimates are a consequence of the known observation that the extension operator maps   
\begin{align*}
\prod_{\underset{|\alpha|\leq r}{\alpha \in \N_l^n}} F_{p,p}^{s-\frac{n-l}{p}-|\alpha|}(\R^l) \rightarrow F_{p,p}^{s}(\R^n).
\end{align*}
This shows 
\begin{align*}
 D^{\beta} \Ext_l^{r,u} g \in F_{p,p}^{\frac{n-l_1}{p},\rloc}(\R^n \setminus \R^{l_1}) \text{ for } \beta \in \N_{l_1}^n, |\beta|=r_1.
\end{align*}
The identity property \eqref{idprop} follows easily from the structure of the extension operator, namely from \eqref{extid} and \eqref{extid2}.

\end{Proof}

\begin{Remark}
 The essential observation of the previous proof is \eqref{extd}. When considering traces and extensions of $\FR$, the $q$-independency of the trace is most remarkable - this can be shown using atomic decompositions and an argument which goes back to Frazier and Jawerth \cite{FrJ90}: One slightly shifts the support of the atoms such that there is only a finite overlap of atoms - resulting in the independency of $q$. 

This is not possible in general for weighted spaces with weights like $d_{l_1}^{-\frac{n-l_1}{p}}$ as we have to consider in our case - but the support condition \eqref{extd2} and its consequence \eqref{extd} enable us to construct a situation very similar to the situation of \cite{FrJ90} for the spaces $\FR$.
\end{Remark}

\section{Decomposition theorems for $\Frinf[\hyp]$ adapted to wavelets}

Our main goal of this section is the proof of the Theorems \ref{Zerleger} and \ref{Zerlegercrit} which are the substitutes for Proposition \ref{decomptri}, which is taken from (6.68) in \cite{Tri08} - from the section called ``A model case''. It can be used later on for the construction of the wavelet bases on cubes, polyhedrons and cellular domains.

A similar observation for the more special $C^{\infty}$-domains is the following proposition, where only traces perpendicular to the boundary $\partial \Om$ are to be considered. Then only the values 
\begin{align*}
  0<s-\frac{1}{p} \notin \N
\end{align*}
 are exceptional.

\begin{Proposition} 
\label{inftydom}
 Let $\Om$ be a bounded $C^{\infty}$-domain in $\R^n$. Let $1\leq p<\infty$, $0<q<\infty$ and 
 \begin{align*}
  0<s-\frac{1}{p} \notin \N.
 \end{align*}
Then
\begin{align*}
 \Ft=\Fo=\{f \in \FO: \tr_{\partial \Om}^{r} f = 0 \}
\end{align*}
\end{Proposition}
\begin{proof}
 These are the observations from \cite[Section 5.24]{Tri01} now in our notation. A proof of these results is given in \cite[Section 2.4.5]{Tri99}. The proof also shows that
 \begin{align*}
  \|f|\Ft\| \sim \|f|\FO\|
 \end{align*}
 for $f \in \Ft$.

\end{proof}
In the following it will be easier to assume $q\geq 1$. For $q$ smaller than $1$ we have to care about the situation for $\Frloc$, since it is only defined for $s>\sigma_{p,q}$. We will give some remarks on these cases in Remark \ref{ZerlegerRem} nevertheless and also try to incorporate the cases $0<q<1$ later.

\subsection{Hardy inequalities using boundary conditions at $\R^l$}

\begin{Lemma}[Hardy inequality]
 \label{HardyZerleger}
Let $n \in \N$, $l \in \N_0$, $l<n$ and $r \in \N$. Let $1\leq p <\infty$, $s>r-1+\frac{n-l}{p}$ and $d(x)$ be the distance of $x=(x',x'') \in \R^l \times \R^{n-l}$ from $\R^l$.
Then there is a constant $c>0$ such that 
\begin{align*}
 \|d^{-s}(\cdot) f|L_p(\R^n)\| \leq c \sum_{\underset{|\alpha|=r}{\alpha \in \N_l^n}} \|d^{-s+r}(\cdot) D^{\alpha}f|L_p(\R^n)\|
\end{align*}
for all $f \in C^r(\R^n)$ with $(D^{\beta}f)(x',0)=0$ for all $x' \in \R^l$ and $\beta \in \N_l^n$ with $|\beta|\leq r-1$.
(In particular, if the left hand side is $\infty$, then also the right hand side.)
\end{Lemma}
\begin{proof}
 This is an $l$-dimensional observation derived from the classical one-dimensional Hardy inequality, first noted in \cite{Har20}. A very good overview on Hardy inequalities is given in \cite{KMP07}.

At first let us prove this lemma for $r=1$: Let $x=(x',x'') \in \R^l \times \R^{n-l}$. We fix $x',x''$ with $x''\neq 0$ and consider the one-dimensional function
\begin{align*}
 g: \R^+ \rightarrow \C, t \mapsto f\left(x',t \cdot \frac{x''}{|x''|}\right).
\end{align*}
Then $g(0)=f(x',0)=0$ and thus
\begin{align*}
 g(t)=g(t)-g(0)&=\int_0^{t} g'(u) \ du\\
&=\int_0^{t} \sum_{j=n-l+1}^n \frac{x_j}{|x''|} \cdot \frac{\partial f(x'',u \cdot \frac{x''}{|x''|})}{\partial x_j} \ du \\
&\leq \int_0^{t} \left|\nabla_{n-l} f\left(x'',u \cdot \frac{x''}{|x''|}\right)\right| \ du 
\end{align*}
by Cauchy's inequality. Now we apply the Hardy inequality for weighted one-dimensional $L_p$-spaces to the function $g$. For this Hardy inequality and a proof we refer to \cite[Theorem 2, p.\ 23]{KMP07} 
\begin{align}
\label{Hardyone}
 \int_0^{\infty} \left(\frac{|f(x',t \cdot \frac{x''}{|x''|})|}{t}\right)^p \cdot t^{\alpha} \ dt \leq c \int_0^{\infty}  \sum_{j=n-l+1}^n \left|\frac{\partial f(x',t \cdot \frac{x''}{|x''|})}{\partial x_j}\right|^p t^{\alpha} \ dt
\end{align}
for $1\leq p < \infty$ and $p>\alpha + 1$.

Now we integrate with respect to $(n-l)$-dimensional spherical coordinates. It holds
\begin{align}
\label{spherecord}
 \int_{x'' \in \R^{n-l}} |h(x'')|^p \ dx'' = \int_{B} \tau(y) \int_0^{\infty}  t^{n-l-1}|h(ty)|^p \ dy \ dt,
\end{align}
where $B:=\{y \in \R^{n-l}: |y|=1\}$ and $\tau$ is a positive function depending only on the angle of $x$, but independent of the absolute value of $x$.

Let $h(x''):=f(x',x'') \cdot |x''|^{-s}$ for $x=(x',x'') \in \R^l \times \R^{n-l}$. Then the inner integral in \eqref{spherecord} can be estimated using \eqref{Hardyone} for every $x'' \in \R^{n-l}$: Having $t=|x''|$ we get
\begin{align*}
 \int_0^{\infty}  t^{n-l-1}|f(x',ty)|^p t^{-sp} \ dt \leq c \int_0^{\infty}  t^{n-l-1} \sum_{j=n-l+1}^n \left|\frac{\partial f(x',ty)}{\partial x_j}\right|^p t^{(-s+1)p} \ dt
\end{align*}
if $p \geq 1$ and $p>1+n-l-1+(-s+1)p$. The second condition is equivalent to $s>\frac{n-l}{p}$. The constant $c$ does not depend on $x'$ or $x''$. Putting together this pointwise estimate and the calculation of the $L_p$-norm in \eqref{spherecord} we arrive at
\begin{align*}
 \int_{x'' \in \R^{n-l}} |f(x',x'')|^p |x''|^{-sp} \ dx'' \lesssim  \int_{x'' \in \R^{n-l}}  \sum_{j=n-l+1}^n \left|\frac{\partial f(x',x'')}{\partial x_j}\right|^p |x''|^{(-s+1)p} \ dx'',
\end{align*}
where the constants are independent of $x' \in \R^l$. Integrating over $x' \in \R^l$ yields our lemma for $r=1$. 

Until now we only used $f(x',0)=0$ and $s>\frac{n-l}{p}$. The general assertion of our lemma for arbitrary $r \in \N$ follows by mathematical induction using the same arguments for the derivatives $D^{\alpha} f$ instead of $f$ itself. Then we need $(D^{\alpha}f)(x',0)=0$ for $|\alpha|\leq r-1$ and $s>r-1+\frac{n-l}{p}$. These observations finish the proof.
\end{proof}

\begin{Remark}
 Lemma \ref{HardyZerleger} also has an easy to proof version for $p=\infty$. One can argue in the same way, but does not need spherical coordinates. There are no restrictions for the exponent of the weight $s \in \R$.  
\end{Remark}

\begin{Remark}
 Let $1\leq p < \infty, 1 \leq q <\infty$. In Lemma \ref{HardyZerleger} we assumed $f \in C^r(\R^n)$ with $\tr_l^{r-1} f=0$. But this lemma also holds true for $f \in \FR$ with $s=r+\frac{n-l}{p}$ and $\tr_l^{r-1} f=0$. Here is a sketch of the arguments: 
 
 Let $\R^+=\{x \in \R: x>0 \}.$ In the proof we used
 \begin{align}
 \label{inteq}
  g(t)=\int_0^t g'(u) \ du 
 \end{align}
for $g \in C^1(\R^+)$ with $g(0)=0$. At first we want to prove that identity \eqref{inteq} also holds true for $g^* \in \FO[\R^+]$ with $s=1+\frac{1}{p}$ and $\tr_{\{0\}} g^*=0$ (only the trace of $g^*$ itself).   

For an extension $g \in \FO[\R]$ of $g^* \in \FO[\R^+]$ we find a sequence of functions $\varphi_j \in \Sc(\R)$ with $g_j:=g * \varphi_j \rightarrow g$ in $\FO[\R]$. Since $s>1$, both $g$ and its (distributional) derivative $g'$ belong to $L_p(\R)$ and hence $g * \varphi_j \rightarrow g$ and $g_j'=g' * \varphi_j \rightarrow g'$ in $L_p(\R)$. By choosing a subsequence we can assume that both sequences converge almost everywhere. Furthermore, we have $s=1+\frac{1}{p}>\frac{1}{p}$ and hence by Proposition \ref{Tra} the trace operator is continuous. This shows
\begin{align*}
 \tr_{\{0\}} g_j \rightarrow \tr_{\{0\}} g=0.
\end{align*}
Now we arrive at
\begin{align*}
  |g(t)- \int_0^t g'(u) \ du| &\leq |g(t)-g_j(t)| + |g_j(t)- \int_0^t g_j'(u) \ du| + |\int_0^t (g_j'(u) -  g'(u)) \ du| \\
  &\leq |g(t)-g_j(t)| + |\tr_{\{0\}} g_j|+ c_t \|g_j'-g'|L_p(\R)\|.
\end{align*}
For almost every $t$ these three terms converge to $0$ and we have shown the identity. 

So, let now $l$ and $n$ be as in Lemma \ref{HardyZerleger} and (as in the proof) at first $r=1$. Then $f \in \FR$ with $s=1+\frac{n-l}{p}$ and $\tr_l f=0$ (only the trace of $f$ itself). In the proof of Lemma \ref{HardyZerleger} we constructed the function
\begin{align*}
 g_{x',x''}: \R^+ \rightarrow \C: t \mapsto f\left(x',t \cdot \frac{x''}{|x''|}\right). 
\end{align*}
But, if $f \in \FR$ with $s=1+\frac{n-l}{p}$, then $h_{x'}(x''):=f(x',x'') \in \FO[\R^{n-l}]$ for almost every $x'$ in $\R^l$ by the Fubini property \ref{Fubini}. Furthermore, using the properties of the trace operator of $F_{p,q}^{1+\frac{n-l}{p}}(\R^{n-l})$ onto one-dimensional lines (see Proposition \ref{Tra}) we get that
\begin{align*}
 g_{x',x''}: \R^+ \rightarrow \C: t \mapsto h_{x'}\left(t \cdot \frac{x''}{|x''|}\right) \in \Fpp{1+\frac{1}{p}}{\R^+}
\end{align*}
for almost all $x' \in \R^{l}$ and moreover $\tr_{\{0\}} g=0$.

Hence we get
\begin{align*}
 g(t)=\int_0^t g'(u) \ du
\end{align*}
almost everywhere.

The rest of the proof of Lemma \ref{HardyZerleger} (for $r=1$) is a matter of $L_p(\R^{n-l})$-integration - as long as $f$ and $D^{\alpha} f$ are regular distributions in $L_p(\R^n)$, there is no further restriction (surely, the right hand side can be infinity). 

For arbitrary $r \in \N$ we made use of an induction argument. Hence we require not only $f \in \FR$ with $s=1+\frac{n-l}{p}$ and $\tr_l f=0$, but the same also for the derivatives $D^{\alpha} f$ of $f$ with $\alpha \in \N_l^n$ upto order $|\alpha|\leq r-1$. But this is satisfied, if we assume $f \in \FR$ with $s=r+\frac{n-l}{p}$ and $\tr_l^r f=0$.

\end{Remark}
Thus we arrive at:
\begin{Corollary}
\label{HardyZerlegerC}
 Let $n \in \N$, $l \in \N_0$, $l<n$ and $r \in \N$. Let $1\leq p <\infty$, $d(x)$ be the distance of $x=(x',x'')\in \R^l \times \R^{n-l}$ from $\R^l$ and let $s=r+\frac{n-l}{p}$.
Then there is a constant $c>0$ such that 
\begin{align*}
 \|d^{-s}(\cdot) f|L_p(\R^n)\| \leq c \sum_{\underset{|\alpha|=r}{\alpha \in \N_l^n}} \|d^{-s+r}(\cdot) D^{\alpha}f|L_p(\R^n)\|
\end{align*}
for all $f \in \FR$ with $\tr_l^{r-1} f = 0$.

(In particular, if the left hand side is $\infty$, then also the right hand side.)

\end{Corollary}
\begin{Remark}
 For the proof of Corollary \ref{HardyZerlegerC} we did not really need $f \in \FR$ with $\tr_l^r f=0$ and $s=r+\frac{n-l}{p}$. It would be enough to assume $s>\max(r,r-1+\frac{n-l}{p})$. Then a counterpart of formula \eqref{inteq} holds, with the same arguments, for $f$ and its derivatives.
\end{Remark}

\subsection{The decomposition theorem for the non-critical cases}

\begin{Theorem}[The non-critical cases]
\label{Zerleger}
 Let $1 \leq p<\infty$ and $1\leq q< \infty$. Let $n \in \N$, $l \in \N_0$ and $l<n$. Let $s>0$,
\begin{align*}
 \quad s-\frac{n-l}{p} \notin \N_0
\end{align*}
and
\begin{align*}
 r= \lfloor{s-\frac{n-l}{p}\rfloor}.
\end{align*}
If $r \in \N_0$, then
\begin{align}
\label{decomp1}
 \Frloc[\hyp]= \left\{f \in \Frinf[\hyp]: \tr_l^{r} f=0\right\}.
\end{align}
If $r=-1$ (hence $s<\frac{n-l}{p}$), then
\begin{align}
\label{decomp2}
 \Frloc[\hyp]=\Frinf[\hyp].
\end{align}
(no trace condition)
\end{Theorem}
\begin{proof}

 \textit{First step:} We show that $\Frloc[\hyp]$ is contained in the RHS of \eqref{decomp1} resp.\ \eqref{decomp2}. At first, if $f \in \Frloc$, then $f$ has a wavelet decomposition by Theorem \ref{rlocwavelet}. Since a wavelet decomposition is a special case of an atomic decomposition, $f$ can also be represented by an atomic decomposition and belongs to $\FR$ with
\begin{align}
\label{rlocest}
 \|f|\FR\| \lesssim \|f|\Frloc[\hyp]\|
\end{align}
by the atomic decomposition Theorem \ref{AtomicRepr}. This shows that $\Frloc[\hyp] \hookrightarrow \Frinf[\hyp]=\FR$ since $s-\frac{n-l}{p} \notin \N_0$.

Furthermore, using \eqref{rlocest} and Remark \ref{rlocDense}, which states that $D(\hyp)$ is dense in $\Frloc[\hyp]$, we find a sequence $\{g_j\}_{j \in \N} \subset D(\hyp)$ with
\begin{align*}
 g_j \rightarrow f \text{ in } \FR
\end{align*}
for every $f \in \Frloc[\hyp]$. Hence by the continuity of the trace operator
\begin{align*}
 0=\tr_l(D^{\alpha}g_j) \rightarrow \tr_l(D^{\alpha}f) \text{ in } F_{p,p}^{s-\frac{n-l}{p}-|\alpha|}(\R^l)
\end{align*}
and by definition
\begin{align*}
\tr_l^r f =0. 
\end{align*}

\textit{Second step:} We show that the RHS is contained in $\Frloc[\hyp]$. For $r=-1$, thus $0<s<\frac{n-l}{p}$, this follows from the Hardy inequalities for the subcritical case, see Proposition \ref{Hardysubcrit}, and the equivalent characterization of $\Frloc[\hyp]$ in Proposition \ref{rlocequi}.  

For the other cases ($r \in \N$, i.\,e.\ $s>\frac{n-l}{p}$) we want to give a proof using a dimension-fixing argument very similar to the proof of Lemma \ref{HardyZerleger}. In the remark after this proof we will present an idea for a second proof using the full generality of Theorem 6.23 and (6.68) in \cite{Tri08}, but with some technical issues making the proof a bit complex.

Let $f \in \Frinf[\hyp]=\FR$ with $\tr_l^r f=0$. Let $x=(x',x'') \in \R^l \times \R^{n-l}$. We fix $x'$ and consider $g_{x'}(x'')=f(x',x'')$ as a function mapping from $\R^{n-l}$. By the Fubini property \ref{Fubini} of $\FR$ we get
\begin{align}
\label{fubint}
 \int_{x'\in \R^l} \|g(x',\cdot)|\FO[\R^{n-l}]\|^p \ dx' \lesssim \|f|\FR\|^p
\end{align}
and at least $g_{x'} \in \FO[\R^{n-l}]$ almost everywhere (with respect to $x' \in \R^l$). Furthermore, by $\tr_l^r f=0$ we get $\tr_{\{x''=0\}} D^{\alpha} g_{x'} = 0$ for $\alpha \in \N_l^n$ with $|\alpha|\leq r$ for all $x'$ with $g_{x'} \in \FO[\R^{n-l}]$.

We now have simplified the situation: We look at a function $g \in \FO[\R^{n-l}]$ with traces at the point $x''=0$ instead of traces at an $l$-dimensional plane. If we show our theorem for this special situation, this means if we find a constant $c>0$ such that
\begin{align}
\label{Hardypoint}
 \|d^{-s}(\cdot) g|L_p((\R^{n-l}\setminus \{0\})_{\eps})\| \lesssim \|g|\FO[\R^{n-l}]\|
\end{align}
for $g$ with $\tr_{\{0\}}^r g = 0$ (with the $(n-l)$-dimensional trace), then by integrating this estimate over $\R^l$ and using \eqref{fubint}, we get the desired inequality
\begin{align*}
 \|d^{-s}(\cdot) f|L_p((\hyp)_{\eps})\| \lesssim \|f|\FR\|.
\end{align*}
So, let's assume $f \in \FO[\R^{n-l}]$ and $\tr_{\{0\}}^r f =0$. In this situation it holds $d(x)=|x|$. Using $(n-l)$-dimensional spherical coordinates similar to \eqref{spherecord} we have
\begin{align*}
 \int_{x \in \R^n} |x|^{-sp} \cdot |f(x)|^p \ dx  = \int_B \tau(y) \int_{0}^{\infty} t^{n-l-1-sp} |f(ty)|^p \ dt \ dy,
\end{align*}
where $B:=\{y \in \R^{n-l}: |y|=1\}$ and $\tau$ is a positive function depending only on the angle $y$ of $x$, but which is independent of the absolute value $t$ of $x$.

Thus it suffices to prove
\begin{align*}
 \int_0^{\infty} t^{-(s-\frac{n-l-1}{p})p} |f(ty)|^p \ dt  \leq c \|f|\FO[\R^{n-l}]\|
\end{align*}
with a constant $c$ independent of $y \in B$ and $f \in \FO[\R^{n-l}]$. 

But again, this can be proven using a very special situation of our theorem, already known: If $f \in \FO[\R^{n-l}]$, then the function
\begin{align*}
 f_y: \R^+ \rightarrow \C: t \mapsto f(ty)  
\end{align*}
for $y \in B$ belongs to $F_{p,p}^{s-\frac{n-l-1}{p}}(\R^+)$ and it holds
\begin{align}
\label{traceinequi}
 \|f_y|F_{p,p}^{s-\frac{n-l-1}{p}}(\R^+)\| \leq c \|f|\FO[\R^{n-l}]\|
\end{align}
with a constant $c$ independent of $y \in B$: For $y=(1,0,\ldots,0)$ this follows from the trace theorem Proposition \ref{Tra}. The independency from $y \in B$ is a consequence of the rotational invariance of $\FR$.

Furthermore, if $\tr_{\{0\}}^r f=0$, then $\tr_{\{0\}}^{r} f_y=0$ by the uniqueness of the trace operator.

Let now $f_y \in F_{p,p}^{s'}(\R^+)$ with $\tr_{\{0\}}^{r} f_y=0$ (all possible traces) and $s'=s-\frac{n-l-1}{p}$. Let $\psi \in D(\R^{+})$ be a non-negative function with $\psi(x)=1$ for $0<x\leq 1$ and $\psi(x)=0$ for $|x|\geq 2$. Then $g_y=\psi \cdot f_y \in F_{p,p}^{s'}([0,2])$ with $\tr_{\{0\}}^{r} g_y=\tr_{\{2\}}^{r} g_y=0$. 

Now we are in a one-dimensional situation. By our assumption it holds
\begin{align*}
 s-\frac{n-l}{p} \notin \N_0 \Rightarrow s'-\frac{1}{p} \notin \N_0.
\end{align*}
and also $s'>\frac{1}{p}$ because we assumed $s>\frac{n-l}{p}$). By Proposition \ref{inftydom} we have $g_y \in \tilde{F\,}\!_{p,p}^{s'}([0,2])$ and by Proposition \ref{specialrloc} thus $g_y \in F_{p,p}^{s',\rloc}([0,2])$ with equivalent norms. Using the equivalent characterization of $F_{p,p}^{s',\rloc}([0,2])$ in Proposition \ref{rlocequi} and $\dist(t,\partial([0,2]))=t$ for $t \in (0,1)$ results in
\begin{align*}
 \int_{0}^{1} t^{-s'p} |g_y(t)|^p \ dt \lesssim \|g_y|F_{p,p}^{s',\rloc}([0,2])\| \sim \|g_y|\tilde{F\,}\!_{p,p}^{s'}([0,2])\| \sim \|g_y|F_{p,p}^{s'}([0,2])\|.
\end{align*}
This together with \eqref{traceinequi} and a pointwise multiplier argument lead to
\begin{align*}
\int_{0}^{\infty} t^{-(s-\frac{n-l-1}{p})p} |f(ty)|^p \ dt &= \int_{0}^{1} t^{-s'p} |f(ty)|^p \ dt + \int_{1}^{\infty} t^{-s'p} |f(ty)|^p \ dt \\
&\lesssim \int_{0}^{1} t^{-s'p} |g_y(t)|^p \ dt + \|f_y|L_p(\R^{+})\|  \\
&\lesssim \|g_y|F_{p,p}^{s'}([0,2])\|+\|f_y|F_{p,p}^{s'}(\R^+)\| \\
&\lesssim \|f_y|F_{p,p}^{s'}(\R^+)\| \\
&\lesssim \|f|\FO[\R^{n-l}]\|.
\end{align*}
This was what we wanted to prove.
\end{proof}

\begin{Remark}
We want to give some ideas for a different proof of the previous Theorem \ref{Zerleger}: It suffices to show the theorem for $f \in \Frinf[\hyp]=\FR$ with $\tr_l^r f =0$ and $supp \ f$ compact. Then we are in the situation of Theorem 6.23 of \cite{Tri08} having $f \in \Ft[c \cdot Q_l^n]$ for a suitable constant $c>0$ with $Q_l^n$ as introduced in Section \ref{Zerlegungsth}. Now we take $2^{n-l}$ functions $\varphi_i \in C^{\infty}(\R^n \setminus \{0\})$ - for every quadrant in $\R^{n-l}$ exactly one. For $i=1,\ldots,2^{n-l}$ the support of $\varphi_i$ should be compact and included in $\R^l \times K_i$ where $K_i$ is a cone with origin at $0$ such that a neighbourhood of the $i$-th quadrant of $\R^{n-l}$ around $0$ is included in $K_i$. Furthermore, 
\begin{align*}
 \sum_{i=1}^{2^{n-l}} \varphi_i(x) =1 \text{ for } x \in B\setminus \{0\}.
\end{align*}
Then $\varphi_i \cdot f \in \Ft[S \times K_i]=\Frloc[S \times K_i]$ where $S$ is a compact set in $\R^l$. We can derive a Hardy inequality at the boundary of $S \times K_i$ by Proposition \ref{rlocequi}. Then it should be possible to follow a Hardy inequality of $\varphi_i \circ f$ at $S\times \{0\} \subset S \times K_i$. Putting all functions $\varphi_i \cdot f$ together we could presumably derive a Hardy inequality at $S \subset \R^l$ which would finish the proof. 
\end{Remark}

\subsection{The decomposition theorem for the critical cases}

\begin{Theorem}[The critical cases]
\label{Zerlegercrit}
 Let $1 \leq p<\infty$ and $1 \leq q< \infty$. Let $n \in \N$, $l \in \N_0$ and $l<n$. Let $s>0$ and
\begin{align*}
 r= s-\frac{n-l}{p} \in \N_0.
\end{align*}
If $r \in \N$, then
\begin{align*}
 \Frloc[\hyp]= \left\{f \in \Frinf[\hyp]: \tr_l^{r-1} f=0\right\}.
\end{align*}
If $r=0$ (hence $s=\frac{n-l}{p}$), then
\begin{align*}
 \Frloc[\hyp]=\Frinf[\hyp].
\end{align*}
(no trace condition)
\end{Theorem}
\begin{proof}
 \textit{First step:} We show, that $\Frloc[\hyp]$ is contained in the RHS. As in the first step of the proof of Theorem \ref{Zerleger}, if $f \in \Frloc[\hyp]$, then $f \in \FR$ and $\tr_l^{r-1} f =0$. 

Furthermore, by Proposition \ref{rlocDiff} it holds $D^{\alpha} f \in F_{p,q}^{\frac{n-l}{p},\rloc}(\hyp)$ for $\alpha \in \N_l^n$ with $|\alpha|=r$. Hence, by Proposition \ref{rlocequi} we have $\delta^{-\frac{n-l}{p}}(\cdot) D^{\alpha}f \in L_p(\R^n)$. Since $\delta(x) = d(x)$ for $d(x) \leq 1$, it follows $f \in \Frinf[\hyp]$ keeping in mind Definition \ref{reinforcedhyp}.

\textit{Second step:} To show, that the RHS is contained in $\Frloc[\hyp]$ with e\-qui\-va\-lence of norms, we use the equivalent characterization of $\Frloc[\hyp]$ from Proposition \ref{rlocequi}. Hence we have to prove that there is a constant $c>0$ such that
\begin{align}
\label{HardyFrinf}
\|d^{-s}(\cdot) f|L_p((\hyp)_{\eps})\| \leq c \|f|\Frinf[\hyp]\|
 \end{align}
for all $f \in \Frinf[\hyp]$ with $\tr_l^{r-1} f=0$. If $r=0$, hence $s=\frac{n-l}{p}$, there is no trace condition - estimate \eqref{HardyFrinf} is a direct consequence of the definition of $\Frinf[\hyp]$.

If $r>0$, then we make use of the Hardy inequality from Corollary \ref{HardyZerlegerC} emerging from Lemma \ref{HardyZerleger}: By definition of $\Frinf[\hyp]$ and $\FR \subset L_p(\R^n)$ we have $d^{-s+r} (\cdot) D^{\alpha} f \in L_p(\R^n)$ for $\alpha \in \N_l^n$ with $|\alpha|=r$. We get
\begin{align*}
 \|d^{-s}(\cdot) f|L_p(\R^n)\| \leq c \sum_{\underset{|\alpha|=r}{\alpha \in \N_l^n}} \|d^{-s+r}(\cdot) D^{\alpha}f|L_p(\R^n)\|
\end{align*}
and so $f$ belongs to $\Frloc[\hyp]= \{f \in \FR: d^{-s}(\cdot) f \in L_p((\hyp)_{\eps})\}$ by Proposition \ref{rlocequi}.
\end{proof}

\begin{Remark}
 One could try to use the idea of the proof of Theorem \ref{Zerlegercrit} also for the non-critical cases from Theorem \ref{Zerleger}.\ By the Hardy inequality for the subcritical case from Proposition \ref{Hardysubcrit} this works perfectly, if $D^{\alpha}f \in \FR[s']$ for $0<s'<\frac{n-l}{p}$ with $\alpha \in \N_l^n$ and $|\alpha|=r+1$. Hence this idea works for $\Frinf[\hyp]$ if $0<s-(r+1)< \frac{n-l}{p}$ with $r=\lfloor s-\frac{n-l}{p} \rfloor$. This covers some cases but not all. 
\end{Remark}

\begin{Remark}
 The simplications in the first proof of Theorem \ref{Zerleger} can also be used in the proof of the critical cases (Theorem \ref{Zerlegercrit}). Hence we can reduce the necessary steps to a one-dimensional (also critical) situation with the trace at a point $x=0$. So it would suffice to prove the Hardy inequality from Corollary \ref{HardyZerlegerC} for dimension $1$ with the trace at a point $x=0$.  
\end{Remark}

\begin{Remark}
 Theorems \ref{Zerleger} and \ref{Zerlegercrit} are the more general substitutes of (6.68) in \cite{Tri08} with the largest possible $r$ as in Remark 6.22 of \cite{Tri08}. In contrast to Proposition 6.21 of \cite{Tri08} and the following observations now $s-\frac{n-l}{p} \in \N_0$ is allowed. Furthermore, when considering domains $\hyp$ instead of $Q_l^n$ as in \cite{Tri08}, the space $\Frloc[\Om]$ is a natural substitute for $\Ft[\Om]$ when $\Om=\hyp$, see the beginning of Section \ref{section:refined} and Proposition \ref{specialrloc}.
\end{Remark}

\begin{Remark}
\label{ZerlegerRem}
 We want to give some remarks on the validity of Theorems \ref{Zerleger} and \ref{Zerlegercrit} if $0<q<1$:

In the non-critical cases investigated in Theorem \ref{Zerleger} the proof only makes use of $s>\sigma_{p,q}$ - then $\Frloc[\hyp]$ is defined, the Fubini property \ref{Fubini} holds and atoms do not need moment conditions.  

But (at least for the second step) on can also assume $q\geq 1$ first and then incorporate $q<1$ by using 
\begin{align}
 \label{monoto}
F_{p,q_1}^{s}(\R^n) 
\hookrightarrow 
F_{p,q_2}^{s}(\R^n)
\end{align}
for $q_1\leq q_2$.

In the critical cases investigated in Theorem \ref{Zerlegercrit} we used $D^{\alpha} f \in F_{p,q}^{\frac{n-l}{p},\rloc}(\hyp)$ for $\alpha \in \N_l^n$ with $|\alpha|=r$. But then naturally we have to assume $s-r=\frac{n-l}{p}> \sigma_{p,q}=\sigma_{p,q}$ by the parameters in the definition of $\Frloc$. 

But one can circumvent the direct use of $F_{p,q}^{\frac{n-l}{p},\rloc}(\hyp)$ such that it suffices to assume $s>\sigma_{p,q}$: If $f \in \Frloc[\hyp]$ with $s-r=\frac{n-l}{p}$, then by Theorem \ref{rlocwavelet} we have a wavelet decomposition of $f$ with a certain structure at the boundary $\R^l$. Differentiation of this decomposition leads to a certain atomic decomposition of $D^{\alpha} f$ with $\alpha \in \N_l^n$ and $|\alpha|=r$, see the first proof of Proposition \ref{rlocDiff}. But using the structure of the atomic decomposition and arguments as in the proof of Corollary \ref{rlocequi} this gives 
\begin{align*}
 \|\delta^{-\frac{n-l}{p}}(\cdot)D^{\alpha} f|L_p(\hyp)\|  \lesssim \|f|\Frloc[\hyp]\|.
\end{align*}
Having in mind $\delta(x)=\min(1,d(x))$ and $D^{\alpha} f\in L_p(\R^n)$ we arrive at the desired replacement.

For the arguments in the second step it suffices to assume $s=r+\frac{n-l}{p}>\sigma_{p,q}$. We are in the same situation as in the second step of the proof for the non-critical cases - maybe incorporating $q<1$ later as suggested by \eqref{monoto}.

Putting everything together, we can extend Theorems \ref{Zerleger} and \ref{Zerlegercrit} to $q<1$ with the additional assumption $s>\sigma_{p,q}$.  
\end{Remark}

\subsection{Corollaries of the decomposition theorems}

\begin{Remark}
 It is well known, that $D(\R^n)$ is dense in $\Frinf[\hyp]=\FR$ for $1\leq p < \infty$, $0<q <\infty$ and $s-\frac{n-l}{p} \notin \N_0$. We want to proof a substitute of this observation for the spaces $\Frinf[\hyp] \subsetneq \FR$, e.\,g.\ when $s-\frac{n-l}{p} \in \N_0$. This will be done in Proposition \ref{frinfdense}.
\end{Remark}

\begin{Corollary}
 \label{frinfdense_tr}
Let $1\leq p < \infty$, $1\leq q <\infty$, let $s=r+\frac{n-l}{p}$ with $r \in \N_0$. Then for every $f \in \Frinf[\hyp]$ with $\tr_l^{r-1} f=0$ there is a sequence $\{g_j\}_{j \in \N} \subset D(\R^n) \cap \Frinf[\hyp]$ with $\tr_l^{r-1} g_j=0$ such that
\begin{align*}
g_j \rightarrow f \text{ in }\Frinf[\hyp]. 
\end{align*}
\end{Corollary}
\begin{proof}
 By Theorem \ref{Zerlegercrit} we have
 \begin{align*}
  \Frloc[\hyp]=\{f \in \Frinf[\hyp]: \tr_l^{r-1} f=0 \}.
 \end{align*}
But by Remark \ref{rlocDense} the set $D(\hyp)$ is dense in $\Frloc[\hyp]$. 
\end{proof}

\begin{Proposition}
\label{frinfdense}
Let $1\leq p < \infty$, $1\leq q <\infty$ and let $s=r+\frac{n-l}{p}$ with $r \in \N_0$. Then for every $f \in \Frinf[\hyp]$ there is a sequence $\{g_j\}_{j \in \N} \subset D(\R^n) \cap \Frinf[\hyp]$ such that 
\begin{align*}
g_j \rightarrow f \text{ in }\Frinf[\hyp]. 
\end{align*}
\end{Proposition}
\begin{Proof}
We decompose $f \in \Frinf[\hyp]$ into two parts $f_1,f_2 \in \Frinf[\hyp]$ such that $\tr_l^{r-1} f_1 = 0$ and $f_2$ has an easy structure:
\begin{align}
\label{deco}
 f= f_1+f_2:=\left( f - (\Ext_l^{r-1,u} \circ\tr_l^{r-1}) f \right) + (\Ext_l^{r-1,u} \circ\tr_l^{r-1}) f
\end{align}
with $u \in \N$ and $u>s$. The desired properties follow from the observations on the extension operator, see Proposition \ref{extwavelet}. By the previous Corollary \ref{frinfdense_tr} the first summand $f_1$ can be approximated by a sequence $\{g_j\}_{j \in \N} \subset D(\R^n) \cap \Frinf[\hyp]$.

Hence we only have to take care about $f_2$. But now we are in a special situation. We constructed $f_2$ as a sum of wavelet building blocks
\begin{align}
\label{extwaveletblocks}
  \Phi_{m}^{j,\alpha}(x) = 2^{j|\alpha|} z^{\alpha} \chi(2^j z)\,  2^{(n-l)j/2} \, \Phi_m^j(y) \quad \text{for } \alpha \in \N_l^n, |\alpha|\leq r-1
\end{align}
converging in $\Frinf[\hyp]$ for $p,q<\infty$, see the proof of Proposition \ref{extwavelet}. Furthermore, by the construction in \eqref{extwaveletblocks} we have $\Phi_{m}^{j,\alpha} \in C^u(\R^n) \cap \Frinf[\hyp]$ and hence $f_2$ can be approximated by finite linear combinations (in $m$ and $j$) of $\Phi_m^{j,\alpha}$. These finite linear combinations have compact support in $\R^n$.

To finish the proof, we have to approximate the building blocks $\Phi_m^{j,\alpha} \in C^u(\R^n) \cap \Frinf[\hyp]$ by functions $D(\R^n) \cap \Frinf[\hyp]$. But this can be done using the tensorproduct structure of $\Phi_m^{j,\alpha}$: We approximate $\Phi_m^j$ in $D(\R^l)$ and multiply this approximation by the other factors from \eqref{extwaveletblocks} - these belong to $D(\R^{n-l})$. 
\end{Proof}

Now we want to give an replacement for the usual monotonic embedding of $F$-spaces. Negative results were presentend in Remark \ref{Frinfsmaller} for the spaces $F_{p,q}^{r+\frac{n-l}{p}+\sigma}(\R^n)$ and $F_{p,q}^{r+\frac{n-l}{p},\rinf}(\hyp)$.

\begin{Corollary}
 \label{embed_rinf} 
 Let $1\leq p<\infty$, $1\leq q<\infty$. Let $r \in \N_0$ and $0<\sigma<1$. Then $f \in F_{p,q}^{r+\frac{n-l}{p}+\sigma,\rinf}(\hyp)=F_{p,q}^{r+\frac{n-l}{p}+\sigma}(\R^n)$ belongs to $ F_{p,q}^{r+\frac{n-l}{p},\rinf}(\hyp)$ 
 if $\tr_l D^{\alpha}f= 0$ for all $\alpha \in \N_l^n$ with $|\alpha|=r$.
\end{Corollary}
\begin{proof}
 As before in \eqref{deco} we decompose $f$ into $f_1+f_2$. Using $\tr_l D^{\alpha}f= 0$ we get
 \begin{align*}
   f= f_1+f_2:&=\left( f - (\Ext_l^{r-1,u} \circ\tr_l^{r-1}) f \right) + (\Ext_l^{r-1,u} \circ\tr_l^{r-1}) f \\
    &=\left( f - (\Ext_l^{r,u} \circ\tr_l^{r}) f \right) + (\Ext_l^{r,u} \circ\tr_l^{r}) f.
 \end{align*}
 Now 
 \begin{align*}
 f_1 \in \left\{f \in F_{p,q}^{r+\frac{n-l}{p}+\sigma}(\R^n): \tr_l^{r} f=0 \right\} 
 \end{align*} and 
\begin{align*}
 f_2 \in (\Ext_l^{r,u} \circ\tr_l^{r}) \left(F_{p,q}^{r+\frac{n-l}{p}}(\R^n)\right). 
\end{align*}
Hence 
\begin{align*}
f_1 \in F_{p,q}^{r+\frac{n-l}{p}+\sigma,\rloc}(\hyp) \subset F_{p,q}^{r+\frac{n-l}{p},\rloc}(\hyp) 
\end{align*}
by Theorem \ref{Zerleger} and $f_2 \in F_{p,q}^{r+\frac{n-l}{p},\rinf}(\hyp)$ by the construction of the wavelet-friendly extension operator in Proposition \ref{extwavelet}. By the usual arguments (see also Remark \ref{ZerlegerRem} for the case $0<q<1$) we have $d^{-s}(\cdot) D^{\alpha} f_1 \in L_p(\R^n)$ for all $\alpha \in \N_l^n$ with $|\alpha|=r$. Hence 
\begin{align*}
 f= f_1+f_2 \in F_{p,q}^{r+\frac{n-l}{p},\rinf}(\hyp).
\end{align*}

\end{proof}
\begin{Remark}
 For $f \in F_{p,q}^{r+\frac{n-l}{p}+\sigma}(\R^n) \cap C^u(\R^n)$ with $u>r+\frac{n-l}{p}+\sigma$ there is also the counterpart of the last theorem: If $f \in F_{p,q}^{r+\frac{n-l}{p},\rinf}(\hyp)$, then $\tr_l D^{\alpha}f= 0$ for all $\alpha \in \N_l^n$ with $|\alpha|=r$. 
 
 This can be seen as follows: If $\tr_l D^{\alpha} f \neq 0$ for some $\alpha \in \N_l^n$ with $|\alpha|=r$, then by continuity we have $|D^{\alpha} f(x',x'')| \geq c>0$ for a small area of $(x',x'') \in \R^n$. But then
 \begin{align*}
  \|d^{-\frac{n-l}{p}}(\cdot) D^{\alpha} f|L_p(\R^n)\| = \infty
 \end{align*}
by direct calculation.
\end{Remark}

In Proposition \ref{rlocDiff} and the subsequent remark we looked at the behaviour of derivatives of $f \in \Frloc$. Now we can prove a converse assertion of these observations - first for the non-critical cases and then for the critical cases: 
\begin{Corollary}
\label{Diffrloc}
 Let $1\leq p <\infty$, $0<q<\infty$, $k \in \N$ and $s-k>\sigma_{p,q}$. Let 
 \begin{align*}
 s-\frac{n-l}{p} \notin \N_0
 \end{align*}
and
\begin{align*}
 r=\lfloor  s-\frac{n-l}{p} \rfloor.
\end{align*}
 Then $f \in \Frloc[\hyp]$ if and only if
\begin{align*}
 D^{\alpha} f \in F_{p,q}^{s-k,\rloc}(\hyp) \text{ for all } |\alpha|\leq k
\end{align*}
and
\begin{align*}
\tr_l D^{\beta} f = 0  
 \text{ for all } \beta \in \N_l^n \text{ with } r-k<|\beta|<k.
 \end{align*}
\end{Corollary}
\begin{proof}
If $f \in \Frloc[\hyp]$, then by Proposition \ref{rlocDiff} we have $D^{\alpha} f \in F_{p,q}^{s-k,\rloc}(\hyp)$ and moreover $\tr_l^r f =0$. 
 
 If $D^{\alpha} f \in F_{p,q}^{s-k,\rloc}(\hyp)$ for $|\alpha|\leq k$, then $D^{\alpha} f \in \FR[s-k]$ and hence $f \in \FR$. Furthermore,  $D^{\alpha} f \in F_{p,q}^{s-k,\rloc}(\hyp)$ for $\alpha \in \N_l^n$ with $|\alpha|=k$. Hence $\tr_l^{r-k} D^{\alpha}f=0$ by Theorem \ref{Zerleger}. Furthermore, we have $\tr_l^{r-k} f=0$. Together with our assumption $\tr_l D^{\beta} f = 0 \text{ for } \beta \in \N_l^n \text{ with } r-k< |\beta| <k$ we get $\tr_l^r f = 0$. By Theorem \ref{Zerleger} this shows $f \in \Frloc[\hyp]$.
\end{proof}

\begin{Corollary}
 Let $1\leq p <\infty$, $0<q<\infty$, $k \in \N$ and $s-k>\sigma_{p,q}$. Let 
 \begin{align*}
  r=s-\frac{n-l}{p} \in \N_0 \text{ and } s-k\geq \frac{n-l}{p}.
 \end{align*} 
 Then $f \in \Frloc[\hyp]$ if and only if
\begin{align*}
 D^{\alpha} f \in F_{p,q}^{s-k,\rloc}(\hyp) \text{ for all } |\alpha|\leq k
\end{align*}
and
\begin{align*}
\tr_l D^{\beta} f = 0  
 \text{ for all } \beta \in \N_l^n \text{ with } r-k\leq |\beta| < k.
 \end{align*}
\end{Corollary} 
\begin{proof}
The proof is nearly the same as the proof of Corollary \ref{Diffrloc}. Instead of showing $f \in \FR$ in the second step we additionally have to show $d^{-\frac{n-l}{p}}(\cdot) D^{\alpha} f \in L_p(\R^n)$ for $\alpha \in \N_l^n$ with $|\alpha|=r$. But this follows from $D^{\alpha} f \in F_{p,q}^{s-k,\rloc}(\hyp)$ for $\alpha \in \N_l^n$ with $|\alpha|=k$ in the same way as in the first step of the proof of Theorem \ref{Zerlegercrit}. Here we needed $s-k\geq \frac{n-l}{p}$.
\end{proof}

%% file: Wavelets_Wuerfel.tex
\label{Wavelets}
\section{Reinforced function spaces on cubes}
Let $Q$ be the open unit cube in $\R^n$ for $n \in \N$. We now adopt the strategy from Sections 6.1.5 to 6.1.7 from \cite{Tri08} to find a wavelet basis for the spaces $\Frinf[Q]$ which at first have to be introduced and will be equal to $\FO[Q]$ as long as we are not in a critical case, i.\,e.\ as long as
\begin{align*}
 s>0, s-\frac{k}{p} \notin \N_0 \text{ for } k=1,\ldots, n .
\end{align*}
Then we will decompose the spaces $\Frinf[Q]$ in a similar way as in Theorem 6.28 of \cite{Tri08} where the spaces $\FO[Q]$ were decomposed for the non-critical values. Now we will find similar results for $\Frinf[Q]$ for the non-critical and critical values which therefore include the results from Theorem 6.28 and 6.30 of \cite{Tri08} as special cases.

The boundary $\Gamma=\partial Q$ of $Q$ can be represented as
\begin{align*}
 \Gamma=\bigcup_{l=0}^{n-1} \Gamma_{l} \text{ with } \Gamma_{l} \cap \Gamma_{l'}= \emptyset \text{ for } l\neq l',
\end{align*}
where $\Gamma_{l}=\bigcup_{j=0}^{n_{l}} \Gamma_{l,j}$ consists of all $l$-dimensional (open) faces $\Gamma_{l,j}$ of $Q$, which are $l$-dimensional disjoint cubes. By construction it holds
\begin{align*}
 \overline{\Gamma}_l=\bigcup_{j=0}^{l} \Gamma_j
\end{align*}
for the closure $ \overline{\Gamma}_l$ of $\Gamma_l$ in $\R^n$.
Let
\begin{align*}
 \N_{l,j}^n=\left\{ \alpha=(\alpha_1,\ldots,\alpha_n) \in \N_0^n: \alpha \text{ perpendicular to } \Gamma_{l,j} \right\}
\end{align*}
be the multi-indices (of length $n$) where only directions perpendicular to $\Gamma_{l,j}$ are considered.

Furthermore, let
\begin{align*}
 d_{l,j}(x)=\dist(x,\Gamma_{l,j}) \text{ and } Q_{l,j,\eps}:=\left\{x \in \R^n: d_{l,j}(x)<\eps\right\}.
\end{align*} 

\subsection{Definition of reinforced spaces on cubes}
In Section \ref{Reinf} we introduced reinforced spaces $\Frinf$ for $\Om=\hyp$ with $0\leq l < n $. Roughly speaking, the spaces $\Frinf[\hyp]$ are emerging from the function spaces $\FR$ only that we had to add a additional decay property at the boundary $\R^l$ if we were in a critical case, i.\,e.\ if
\begin{align*}
 s-\frac{n-l}{p} \in \N_0.
\end{align*}
Now, when we are looking at the cube $Q$, we have to deal with faces $\Gamma_l$ of dimension $l \in \{0,\ldots,n-1\}$. The idea is now to introduce the function space $\Frinf[Q]$ as restriction of the function space $\FR$ where we additionally assume a decay property for every face $\Gamma_l$ of dimension $l$ where the value $l$ is critical, i.\,e.\ where
\begin{align*}
 s-\frac{n-l}{p} \in \N_0.
\end{align*}
The other way around, this means that if 
\begin{align*}
 s-\frac{k}{p} \notin \N_0
\end{align*}
for every $k \in\{1,\ldots,n\}$, then we will have $\Frinf[Q]=\FO[Q]$.

\begin{Definition}
Let $n \in \N$ and $l\in \N_0$ with $l<n$. Let $1\leq p < \infty$ and 
\begin{align*}
s-\frac{n-l}{p}=r \in \N_0.
\end{align*}
Then $f \in \FR$ is said to fulfil the reinforce property $R_l^{r,p}$ if, and only if, 
\begin{align*}
 d_{l,j}^{-\frac{n-l}{p}} \cdot D^{\alpha} f \in L_p(Q_{l,j,\eps}) \text{ for all } \alpha \in \N_{l,j}^n, |\alpha|=r \text{ and } j=1,\ldots,n_{l}.
\end{align*}
\end{Definition}
\begin{Remark}
 Roughly speaking, this means that if we are in a critical case for dimension $l$, then we have to take care of the decay at all faces $\Gamma_{l,j}$ of dimension $l$ and for all derivatives of order $r$ with directions perpendicular to $\Gamma_{l,j}$. 
\end{Remark}

\begin{center}
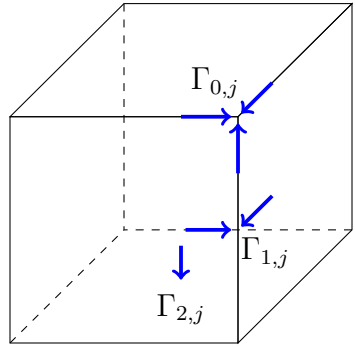

 \begin{tikzpicture}[scale=3]
 \def\u{0.03};
  \draw (0,0) -- (0,1) -- (1,1) -- (1,0) -- cycle;
 \draw (0,1) -- (1,1) -- (1.5,1.5) -- (0.5,1.5) -- cycle;
 \draw (1,0) -- (1.5,0.5) -- (1.5,1.5) -- (1,1) -- cycle;
 \draw[dashed] (0,0)--(0.5,0.5);
 \draw[dashed] (0.5,1.5)--(0.5,0.5);
\draw[dashed] (1.5,0.5)--(0.5,0.5);

\draw[color=blue,line width=0.05cm] [->] (0.75,0.43) to (0.75,0.25+\u);

\draw[color=blue,line width=0.05cm] [->] (0.77,0.5) to (1-\u,0.5);
\draw[color=blue,line width=0.05cm] [->] (1.15,0.65) to (1+\u/2,0.5+\u/2);

\draw[color=blue,line width=0.05cm] [->] (1,0.75) to (1,1-\u);
\draw[color=blue,line width=0.05cm] [->] (0.75,1) to (1-\u,1);
\draw[color=blue,line width=0.05cm] [->] (1.15,1.15) to (1+\u/2,1+\u/2);

\draw (0.9,1.15) node  {$\Gamma_{0,j}$};
\draw (0.75,0.15) node  {$\Gamma_{2,j}$};
\draw (1.12,0.4) node  {$\Gamma_{1,j}$};
\end{tikzpicture}
 \captionof{figure}{Directions of the boundary conditions at the faces of $Q$}
   \label{reinfdirect}
\end{center}

Now we are able to define the reinforced function spaces $\Frinf[Q]$ for the cube $Q$ which we will characterize by a wavelet basis later. 

\begin{Definition}
\label{reinforcedQ}
Let $Q$ be the unit cube and $\Gamma=\partial Q$ its boundary. Let $1\leq p < \infty$, $0<q<\infty$ and $s \in \R$. Then
\begin{multline*}
\Frinf[\R^n \setminus \Gamma]:=\\\{f \in \FR: f \text{ fulfilfs } R_l^{r^l,p} \text{ for all } l\in \{0,\ldots,n-1\} \text{ with } r^l=s-\frac{n-l}{p} \in \N_0 \}.
\end{multline*}
A possible norm $\|f|\Frinf[\R^n\setminus \Gamma]\|$ is given as the sum of $\|f|\FR\|$ and the weighted $L_p$-norms for $D^{\alpha}f$ appearing in the definition of $R_l^{r^l,p}$ for those $l$ where $s-\frac{n-l}{p} \in \N_0$. 

Furthermore, the space $\Frinf[Q]$ is defined by restriction of $\Frinf[\R^n \setminus \Gamma]$ to $Q$, i.\,e.\
\begin{align*}
 \Frinf[Q]:=\{f \in D'(Q): \exists \, g \text{ with } g|Q=f \}
\end{align*}
with the usual quotient space norm
\begin{align*}
 \|f|\Frinf[Q]\|:= \inf_g \|g|\Frinf[\R^n \setminus \Gamma]\|
\end{align*}
where the infimum is taken over all $g$ with $g|Q=f$. 
\end{Definition}
\begin{Remark}
 The definition can be understood as follows: An $f$ belongs to the space $\Frinf[\R^n \setminus \Gamma]$ if it belongs to $\FR$ and additionally for every critical value $l$ we add a reinforce property, this means we assume a decay property of the perpendicular derivatives at the faces of dimension $l$.  
 
 In the definition of $\Frinf[Q]$ we can replace $D'(Q)$ by $\FO[Q]$.
\end{Remark}

\section{The decomposition of $\Frinf[Q]$ into function spaces having wavelet bases}
In \cite[Theorem 6.28]{Tri08} Triebel proved a crucial decomposition of $\FO[Q]$ which paved the way to construct a wavelet basis for $\FO[Q]$ in Theorem 6.30 since every part of the decomposition has a wavelet basis. But then one has to exclude the critical values, i.\,e.\
\begin{align*}
 s-\frac{n-l}{p} \in \N_0 \text{ for } l \in \{0,\ldots,n-1\}. 
\end{align*}
Now we want to find a similar decomposition for the spaces $\Frinf[Q]$ which also incorporates the critical values, but includes Theorem 6.28 as well, since $\Frinf[Q]=\FO[Q]$ for the non-critical values.

\subsection{Traces and wavelet-friendly extension operators at the boundaries of the cube}
We adopt the notation from Section \ref{traces} and \cite[Section 6.1.5]{Tri08}: Let $n \in \N$ and $l\in \N_0$ with $l<n$. For $j=0,\ldots,n_l$ let now $\tr_{l,j}$ be given as the restriction 
\begin{align*}
 \tr_{\Gamma_{l,j}}: f \mapsto f|\Gamma_{l,j}, f \in \FR
\end{align*}
of $f$ to $\Gamma_{l,j}$ (if existing). As before, we collect all traces of derivatives upto order $r$ which are perpendicular to $\Gamma_{l,j}$ and denote this operator by
\begin{align*}
 \tr_{\Gamma_{l,j}}^r: f \mapsto \{\tr_{l,j} D^{\alpha} f: \alpha \in \N_{l,j}^n, |\alpha|\leq r \}.
\end{align*}
Finally we collect all these traces on the $l$-dimensional faces $\Gamma_{l,j}$ together for the composite map
\begin{align*}
 \tr_{\Gamma_{l}}^r: f \mapsto \{\tr_{l,j}^r f : j=0\ldots,n_l\}.
\end{align*}

Furthermore, in analogy to Section \ref{traces} we can construct a wavelet-friendly extension operator, but now for $\Gamma_{l,j}$ instead of $\R^l$, for a fixed $l \in \{0,\ldots,n-1\}$ and all $j=0,\ldots,n_l$. 

Let $1\leq p <\infty$ and $\sigma>0$. We take a function $g \in F_{p,p}^{\sigma,\rloc}(\Gamma_{l,j})=\tilde{F\,}\!_{p,p}^{\sigma}(\Gamma_{l,j})$ (now instead of $F_{p,p}^{\sigma}(\R^l)$) which has by Theorem \ref{rlocwavelet} a decomposition into wavelets adapted to $\Gamma_{l,j}$. This is in principal the same for every $\Gamma_{l,j}$ with fixed $l$. For convenience we will give the construction for the first 
\begin{align*}
\Gamma_{l,0}=\{x \in \R^n: 0<x_m<1 \text{ for } 1\leq m\leq l \text{ and } x_m=0 \text{ for } l<m\leq n \}.
\end{align*}
The extended functions for different $\Gamma_{l,j}$ (fixed $l$) will not interfere.

  Let $u \in \N$ and 
\begin{align*}
 \{ \Phi_m^j: j \in \N_0, m=1,\ldots,N_j \} \subset C^u(\R^l)
\end{align*}
be an interior u-Riesz basis for $F_{p,p}^{\sigma,\rloc}(\Gamma_{l,0})$ by Theorem \ref{rlocwavelet}. Hence, every $g \in F_{p,p}^{\sigma,\rloc}(\Gamma_{l,0})$ can be represented as a wavelet expansion, namely
\begin{align*}
 g= \sum_{j=0}^{\infty}\sum_{m=1}^{N_j} \lambda_m^j(g) 2^{-\frac{jl}{2}}\Phi_m^j
\end{align*}
with
\begin{align*}
  \lambda_m^j(g)= 2^{jl/2} \int_{\R^l} g(y) \Phi_m^j(y) \ dy
\end{align*}
and an isomorphic map
\begin{align*}
 g \mapsto \{\lambda_m^j(g)\} 
\end{align*}
of $F_{p,p}^{\sigma,\rloc}(\Gamma_{l,0})$ onto $f_{p,p}^{\sigma}(\Z^{\Gamma_{l,0}})$ where $f_{p,p}^{\sigma}(\Z^{\Gamma_{l,0}})$ is the interior sequence space introduced in Definition \ref{sequence}. Let 
\begin{align*}
 \chi^* \in D(\R), \quad supp \ \chi^* \subset \left\{z \in \R: |z|\leq \frac{1}{4} \right\}, \quad \chi^*(z)=1 \text{ if } |z|\leq \frac{1}{8}
\end{align*}
and
\begin{align*}
\chi(z)=\chi^*(z_1) \cdot \ldots \cdot \chi^*(z_{n-l}).
\end{align*}
The support of $\chi$ is small enough such that extension operators of opposite faces will not interfere with each other. It is possible to choose $\chi$ such that it fulfils as many moment conditions
\begin{align*}
 \int_{\R^{n-l}} \chi(z) z^{\beta} \ dz = 0 \text{ if } |\beta|\leq L 
\end{align*}
as we want. Now we define $n$-dimensional functions (extensions) by
\begin{align*}
 \Phi_m^{j,\alpha}(x) = 2^{j|\alpha|} z^{\alpha} \chi(2^j z)\,  2^{(n-l)j/2} \, \Phi_m^j(y) \quad \text{for } \alpha \in \N_{l,0}^n.
\end{align*}
with $x=(y,z) \in \R^l \times \R^{n-l}$. It is easy to see that
\begin{align}
\label{extid3}
 \tr_{\Gamma_{l,0}} D^{\alpha}  \Phi_m^{j,\alpha}(y) = 2^{j|\alpha|} \cdot \alpha! \cdot 2^{(n-l)j/2} \Phi_m^j(y)= c_{j,\alpha} \Phi_m^j(y)
\end{align}
and
\begin{align}
\label{extid4}
 \tr_{\Gamma_{l,0}} D^{\alpha}  \Phi_m^{j,\beta}(y) = 0 \quad \text{ for } \alpha \neq \beta \in \N_{l,0}^n	 
\end{align}
for $y \in \R^l$ having in mind the factor $z^{\alpha}$. This is the crucial property giving the possibility to construct the extension operator by
\begin{align*}
 g_0^l&= \Ext_{\Gamma_{l,0}}^{r,u} \{g_{\alpha}: \alpha \in \N_{l,0}^n, |\alpha|\leq r \} \\
  &= \sum_{|\alpha|\leq r} \sum_{j=0}^{\infty}\sum_{m=1}^{N_l} \frac{1}{\alpha!} \lambda_m^j(g_{\alpha}) 2^{-j|\alpha|} 2^{-\frac{jn}{2}}\cdot\Phi_m^{j,\alpha}.
\end{align*}
Furthermore, by construction
\begin{align*}
 supp \,g_0^l \subset [0,1]^l \times \left\{z=(z_1,\ldots,z_{n-l} \in \R^{n-l}:|z_i|\leq \frac{1}{2},\ i=1,\ldots,n-l\right\}
\end{align*}
and traces on different faces of dimension $l$ do not interfere, i.\,e.\
\begin{align*}
 \tr_{\Gamma_{l,j}} D^{\alpha} g_0^l = 0 \text{ for } j=1,\ldots, n_l \text{ and } \alpha \in \N_{l,j}^n.
\end{align*}
This is a result of the specific construction of the extension operator out of the wavelet basis - see the picture below illustrating the maximal support of the functions $ \Phi_m^{j,\alpha}$ arising in the construction of the extension operator.
\begin{center}
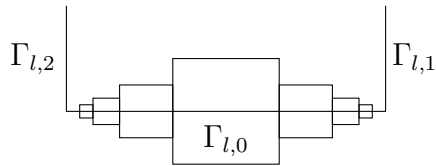

 \begin{tikzpicture}[scale=0.35]
    \draw (0,4)--(0,0) -- (12,0) -- (12,4);
    \draw (4,-2) -- (4,2) -- (8,2) -- (8,-2) --cycle;
    \draw (2,-1) -- (2,1) -- (4,1) -- (4,-1) --cycle;  
    \draw (8,-1) -- (8,1) -- (10,1) -- (10,-1) --cycle; 
    \draw (1,-0.5) -- (1,0.5) -- (2,0.5) -- (2,-0.5) --cycle; 
    \draw (10,-0.5) -- (10,0.5) -- (11,0.5) -- (11,-0.5) --cycle; 
    \draw (0.5,-0.25) -- (0.5,0.25) -- (1,0.25) -- (1,-0.25) --cycle; 
    \draw (11,-0.25) -- (11,0.25) -- (11.5,0.25) -- (11.5,-0.25) --cycle;
    \draw (6,-1.1) node  {$\Gamma_{l,0}$};
    \draw (13.1,2) node  {$\Gamma_{l,1}$};
    \draw (-1.25,2) node  {$\Gamma_{l,2}$};
 \end{tikzpicture}
 \captionof{figure}{Support of extended wavelets near $\Gamma_{l,0}$}
   \label{waveleter}
\end{center}
Now, after constructing the extension operator $\Ext_{\Gamma_{l,0}}^{r,u}$ for the face $\Gamma_{l,0}$ and the observation that different traces do not interfere, we can construct similar extension operators $g_j^l=\Ext_{\Gamma_{l,j}}^{r,u}\{g_{j,\alpha}: \alpha \in \N_{l,j}^n, |\alpha|\leq r \}$  for the faces $\Gamma_{l,j}$. Then we define the extension operator of the whole $l$-dimensional face $\Gamma_l$ as
\begin{align*}
 \Ext_{\Gamma_{l}}^{r,u}\{g_{j,\alpha}: j\in \{0,\ldots,n_l\}, \alpha \in \N_{l,j}^n, |\alpha|\leq r \}:= \sum_{j=0}^{n_l} g_j^l
\end{align*}
with $g_j^l=\Ext_{\Gamma_{l,j}}^{r,u}\{g_{j,\alpha}: \alpha \in \N_{l,j}^n, |\alpha|\leq r \}$.  Using \eqref{extid3} and \eqref{extid4}, the observation that different extension operators do not interfere and the definition of $\tr_{\Gamma_{l}}^r$ we arrive at
\begin{align*}
 \tr_{\Gamma_{l}}^r \circ \Ext_{\Gamma_{l}}^{r,u}=\text{id} 
\end{align*}
on a product space which will be described in the next section.

\subsection{Trace and extension theorems at the boundaries of the cube}
Now we are in the situation to prove a crucial proposition on traces and extension operators of the spaces $\Frloc[\R^n \setminus \overline{\Gamma}_l]$. As in \cite[Theorem 6.28]{Tri08}, one has to observe the behaviour of traces and extension operators on resp.\ of faces of different dimensions $\Gamma_{l_1}$ and $\Gamma_{l_2}$. This is the counterpart of the observation in the proof of Theorem 6.28 in \cite{Tri08} that the trace space of the right hand side of (6.120) is $\tilde{B\,}\!_{p,q}^{s-\frac{n-l}{p}}(\Gamma_l)$. 

The set $\N_l^n$ shall be the collection of all multi-indices (of length $n$) where only directions perpendicular to $\Gamma_{l}$ are considered. This obviously depends on the special face $\Gamma_{l,j}$ we are considering but because of the total decoupling of traces and extension operators of different faces of fixed dimension $l$ we can consider every face of dimension $l$ separately and it is always clear what is meant by $\N_l^n$ at a given boundary point. 

The idea of the proposition is the following: If we have an element $f$ of the function space $\FR$ which has no boundary values at the faces $\Gamma_{l-1}$ of the cube of dimension $l-1$, then the traces at the faces $\Gamma_l$ of the cubes of dimension $l$ vanish at their boundary (which is $\Gamma_{l-1}$) , i.\,e.\ the element $f$ belongs to a refined localization space on $\Gamma_l$.	

\begin{Proposition}
\label{Frlocdecomp}
Let $n \in \N$, $l,r\in \N_0$ with $0<l<n$. Let $1\leq p< \infty$, $0<q<\infty$, $s>\sigma_{p,q}$, $s>\frac{n-l}{p}$, $u>s$ and 
\begin{align*}
r<s-\frac{n-l}{p}. 
\end{align*}
Then
\begin{align*}
 \tr_{\Gamma_l}^r: \Frloc[\R^n \setminus \overline{\Gamma}_{l-1}] \mapsto \prod_{\underset{|\alpha|\leq r}{\alpha \in \N_{l}^n}} F_{p,p}^{s-\frac{n-l}{p}-|\alpha|,\rloc}(\Gamma_{l})
\end{align*}
and
\begin{align*}
 \Ext_{\Gamma_{l}}^{r,u}: \prod_{\underset{|\alpha|\leq r}{\alpha \in \N_{l}^n}} F_{p,p}^{s-\frac{n-l}{p}-|\alpha|,\rloc}(\Gamma_{l}) \mapsto \Frloc[\R^n \setminus \overline{\Gamma}_{l-1}]
\end{align*}
with 
\begin{align*}
 \tr_{\Gamma_l}^r \circ 
 \Ext_{\Gamma_{l}}^{r,u}= \text{id }, \text{ identity on }  \prod_{\underset{|\alpha|\leq r}{\alpha \in \N_{l}^n}} F_{p,p}^{s-\frac{n-l}{p}-|\alpha|,\rloc}(\Gamma_{l})
\end{align*}
\end{Proposition}
\begin{proof}
The proof is a consequence of the wavelet decompositions of the spaces $\Frloc[\R^n \setminus \overline{\Gamma}_{l-1}]$ and $F_{p,p}^{\sigma,\rloc}(\Gamma_{l})=\tilde{F\,}\!_{p,p}^{\sigma}(\Gamma_l)$. 
Let $f \in \Frloc[\R^n \setminus \overline{\Gamma}_{l}]$. By Theorem \ref{rlocwavelet} we find a wavelet decomposition
\begin{align*}
  f=\sum_{j=0}^{\infty}\sum_{m=1}^{N_j} \lambda_m^j(f) 2^{-\frac{jn}{2}}\Phi_m^j \text{ on } \R^n
\end{align*}
with $\lambda(f) \in \fO[\Z^{\R^n \setminus \overline{\Gamma}_{l-1}}]$. The support of the wavelets near a face of dimension $l-1$, e.\,g.\ $\Gamma_{l-1,0}$, is illustrated in the picture below (looking from one side of the cube):

\begin{center}
 \begin{tikzpicture}[scale=0.35] 
    \draw (4,-2) -- (4,2) -- (8,2) -- (8,-2) --cycle;
    \draw (2,-1) -- (2,1) -- (4,1) -- (4,-1) --cycle;  
    \draw (8,-1) -- (8,1) -- (10,1) -- (10,-1) --cycle; 
    \draw (1,-0.5) -- (1,0.5) -- (2,0.5) -- (2,-0.5) --cycle; 
    \draw (10,-0.5) -- (10,0.5) -- (11,0.5) -- (11,-0.5) --cycle; 
    \draw (0.5,-0.25) -- (0.5,0.25) -- (1,0.25) -- (1,-0.25) --cycle; 
    \draw (11,-0.25) -- (11,0.25) -- (11.5,0.25) -- (11.5,-0.25) --cycle;
    \draw (0,4)--(0,0) -- (12,0) -- (12,4);
    \draw (-4,-2) -- (-4,2) -- (-8,2) -- (-8,-2) --cycle;
    \draw (-2,-1) -- (-2,1) -- (-4,1) -- (-4,-1) --cycle;  
    \draw (-8,-1) -- (-8,1) -- (-10,1) -- (-10,-1) --cycle; 
    \draw (-1,-0.5) -- (-1,0.5) -- (-2,0.5) -- (-2,-0.5) --cycle; 
    \draw (-10,-0.5) -- (-10,0.5) -- (-11,0.5) -- (-11,-0.5) --cycle; 
    \draw (-0.5,-0.25) -- (-0.5,0.25) -- (-1,0.25) -- (-1,-0.25) --cycle; 
    \draw (-11,-0.25) -- (-11,0.25) -- (-11.5,0.25) -- (-11.5,-0.25) --cycle;
    
    \draw (-2,-4) -- (2,-4) -- (2,-8) -- (-2,-8) --cycle;
    \draw (-1,-2) -- (1,-2) -- (1,-4) -- (-1,-4) --cycle;  
    \draw (-1,-8) -- (1,-8) -- (1,-10) -- (-1,-10) --cycle; 
    \draw (-0.5,-1) -- (0.5,-1) -- (0.5,-2) -- (-0.5,-2) --cycle; 
    \draw (-0.5,-10) -- (0.5,-10) -- (0.5,-11) -- (-0.5,-11) --cycle; 
    \draw (-0.25,-0.5) -- (0.25,-0.5) -- (0.25,-1) -- (-0.25,-1) --cycle; 
    \draw (-0.25,-11) -- (0.25,-11) -- (0.25,-11.5) -- (-0.25,-11.5) --cycle; 
    
    \draw (-2,4) -- (2,4) -- (2,8) -- (-2,8) --cycle;
    \draw (-1,2) -- (1,2) -- (1,4) -- (-1,4) --cycle;  
    \draw (-1,8) -- (1,8) -- (1,10) -- (-1,10) --cycle; 
    \draw (-0.5,1) -- (0.5,1) -- (0.5,2) -- (-0.5,2) --cycle; 
    \draw (-0.5,10) -- (0.5,10) -- (0.5,11) -- (-0.5,11) --cycle; 
    \draw (-0.25,0.5) -- (0.25,0.5) -- (0.25,1) -- (-0.25,1) --cycle; 
    \draw (-0.25,11) -- (0.25,11) -- (0.25,11.5) -- (-0.25,11.5) --cycle;

    \draw[dashed] (0,12)--(0,0) -- (12,0) -- (12,12)--cycle;
    
    \draw (6,-1.1) node  {$\Gamma_{l,0}$};
    \draw (13.1,6) node  {$\Gamma_{l,1}$};
    \draw (-1,6) node  {$\Gamma_{l,2}$};
    \draw (6,13.1) node  {$\Gamma_{l,3}$};
    
 \end{tikzpicture}
   \captionof{figure}{Support of wavelets of $\Frloc[\R^n \setminus \overline{\Gamma}_{l-1}]$ near $\Gamma_{l,0}$}
   \label{waveleter2}
\end{center}
If we look at the trace operator $\tr_{\Gamma_l}^r f$, hence at the restriction of the wavelet decompositions to $\Gamma_{l}$ (more exactly to $\Gamma_{l,j}$), then we see that the remaining functions have the same supports as the wavelets of the wavelet decompositions of $F_{p,p}^{\sigma,\rloc}(\Gamma_{l})$, see Theorem \ref{rlocwavelet}. Hence we have an atomic decomposition of $\tr_{\Gamma_l} D^{\alpha} f$ in $F_{p,p}^{s-\frac{n-l}{p}-|\alpha|,\rloc}(\Gamma_{l})$ as in Theorem \ref{AtomicReprrloc}. 

The independency of $q$ of the trace space for $\Frloc[\R^n \setminus \overline{\Gamma}_{l-1}]$ as well as the exponent $s-\frac{n-l}{p}-|\alpha|$ are well known observations, see \cite[Theorem 5.14, Proposition 6.17]{Tri08} and the foregoing remarks. Using atomic decomposition arguments, the independency of $q$ of the trace space for $F$-spaces goes back to Frazier and Ja\-werth \cite{FrJ90}, especially resulting from Corollary 5.6 there. Since this is a matter of sequence spaces $f_{p,p}^{\sigma}(\Z^{\Gamma_l})$, one can transfer 
the observations made for $\FR$ to $\Frloc[\R^n \setminus \overline{\Gamma}_{l-1}]$ directly.

Putting this together, we obtain an atomic decomposition (as in Theorem \ref{AtomicReprrloc}) of $\tr_{\Gamma_l} D^{\alpha}f$ in $F_{p,p}^{s-\frac{n-l}{p}-|\alpha|,\rloc}(\Gamma_{l})$ with
\begin{align*}
 \|\lambda_{\tr_{\Gamma_l} D^{\alpha}f}|f_{p,p}^{s-\frac{n-l}{p}-|\alpha|}(\Z^{\Gamma_l})\| \lesssim  \|\lambda_f|\fO[\Z^{\R^n \setminus \overline{\Gamma}_{l-1}}]\| \sim \|f|\Frloc[\R^n \setminus \overline{\Gamma}_{l-1}]\|.
\end{align*}
This shows the first part of the proposition.

For the second part we take another look at Figures \ref{waveleter} and \ref{waveleter2}. If we have a wavelet decomposition of functions $g_{\alpha,j} \in F_{p,p}^{s-\frac{n-l}{p}-\alpha,\rloc}(\Gamma_{l})$ for $\alpha \in \N_{l,j}^n$ with $|\alpha|\leq r$, then the support of the wavelet building blocks of the extension operator $\Ext_{\Gamma_{l,j}}^{r,u}\{g_{\alpha}: \alpha \in \N_{l,j}^n, |\alpha|\leq r \}$ has the same structure as the wavelet decomposition of $\Frloc[\R^n \setminus \overline{\Gamma}_{l-1}]$. Hence the extension operator $\Ext_{\Gamma_{l,j}}^{r,u}\{g_{\alpha}: \alpha \in \N_{l,j}^n, |\alpha|\leq r \}$ has an atomic decomposition (as in Theorem \ref{AtomicReprrloc}) in $\Frloc[\R^n \setminus \overline{\Gamma}_{l-1}]$ for every $0<q<\infty$. The independency of $q$ and the subsequent norm estimate follow as in the first step (by the arguments of \cite{FrJ90}).

\noindent
Finally, property
\vspace{-0.2cm}
\begin{align*}
 \tr_{\Gamma_l}^r \circ 
 \Ext_{\Gamma_{l}}^{r,u}= \text{id}, \text{ identity on }  \prod_{\underset{|\alpha|\leq r}{\alpha \in \N_{l}^n}} F_{p,p}^{s-\frac{n-l}{p}-|\alpha|,\rloc}(\Gamma_{l})
\end{align*}
\vspace{-0.4cm}

\noindent
follows from the construction of the extension operator and the non-interference at different faces of fixed dimension $l$.
\end{proof}

We need another lemma which helps us to enable the decomposition of $\Frinf[Q]$ - if $f \in \Frinf[Q]$ needs to fulfil reinforce properties $R_l^{r,p}$, we need to make sure that the reinforce properties $R_l^{r,p}$ are preserved by the trace-extension operator procedure we will use later on.

\begin{Lemma}
\label{reinfprop}
 Let  $n \in \N$, $l,l_1,r\in \N_0$ with $l\leq l_1<n$. Let $1\leq p< \infty$,   
\begin{align*}
 u>s>r+\frac{n-l}{p} \text{ and } r_1=s-\frac{n-l_1}{p} \in \N_0.
\end{align*}
Let 
\begin{align*}
 g_{j,\alpha} \in F_{p,p}^{s-\frac{n-l}{p}-|\alpha|,\rloc}(\Gamma_{l,j}) \text{ for } j\in \{0,\ldots,n_l\}, \alpha \in \N_{l,j}^n, |\alpha|\leq r .
\end{align*}
Then 
\begin{align*}
g&=\Ext_{\Gamma_{l}}^{r,u}\{g_{j,\alpha}: j\in \{0,\ldots,n_l\}, \alpha \in \N_{l,j}^n, |\alpha|\leq r \} \\ &
\end{align*}
\vspace{-1.1cm}

\noindent
fulfils the reinforced property $R_{l_1}^{r_1,p}$.

Furthermore, there is a suitable norm estimate of the weighted $L_p$-norms appearing in the definition of $R_{l_1}^{r_1,p}$ by the sum of $\|g_{\alpha,j}|F_{p,p}^{s-\frac{n-l}{p}-|\alpha|,\rloc}(\Gamma_{l,j})\|$).
\end{Lemma}
\begin{proof}
 In Proposition \ref{extwavelet} we showed this property with $\R^l$ instead of $\Gamma_l$ and $\R^{l_1}$ instead of $\Gamma_{l_1}$ (the latter is to be considered in the definition of $R_{l_1}^{r_1,p}$). We look at the different components $\Gamma_{l,j}$ of $\Gamma_{l}$ and $\Gamma_{l_1,j'}$ of $\Gamma_{l_1}$. At first, we can assume that $\Gamma_{l,j}$ and $\Gamma_{l_1,j'}$ are adjacent to each other, otherwise the support of $\Ext_{\Gamma_{l,j}}^{r,u} \{g_{j,\alpha}: \alpha \in \N_{l,j}^n, |\alpha|\leq r \}$ has positive distance to $\Gamma_{l_1,j'}$ and reinforce properties $\Gamma_{l_1,j'}$ are trivial.
 
 In a first step we consider the case, where $\Gamma_{l,j}$ is adjacent to $\Gamma_{l_1,j'}$ and $l_1>l$. This means that $\Gamma_{l,j}$ is part of the boundary of $\Gamma_{l_1,j'}$. Then we can argue in the same way as in the proof of Proposition \ref{extwavelet}: The support of the derivatives of order $r_1$ of the wavelet blocks looks like in Figure \ref{reinfext}. Obviously, the distance of $x \in \R^n$ from $\Gamma_{l_1,j'}$ is not smaller than the distance of $x \in \R^n$ from the $l_1$-dimensional plane which includes $\Gamma_{l_1,j'}$ (which corresponds to $\R^{l_1}$). 
 
 \begin{center}
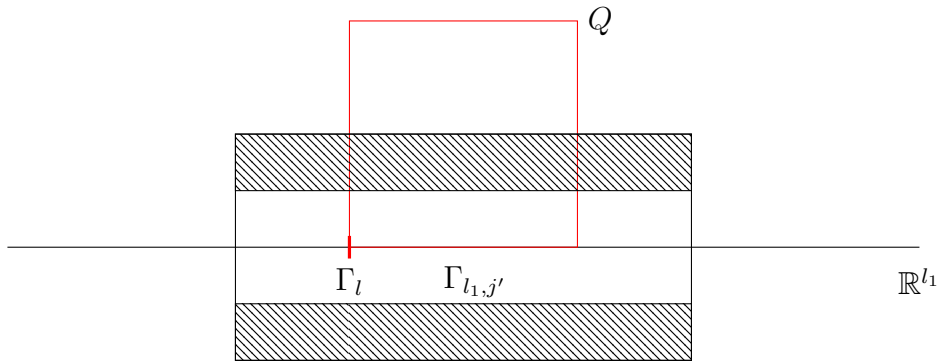

 \begin{tikzpicture}[scale=3]
 \def\u{0.03};
  \draw (-1,0) -- (3,0);
  \draw[red] (0.5,0) -- (1.5,0) -- (1.5,1) -- (0.5,1) -- cycle;
  \draw (1.6,1) node {$Q$};
\draw[red,very thick] (0.5,-0.05)--(0.5,0.05);
\draw[very thick] (0.5,-0.15) node {$\Gamma_l$};
\draw (3,-0.15) node {$\R^{l_1}$};
\draw (1.05,-0.15) node {$\Gamma_{l_1,j'}$};
\draw (2,0.5)--(0,0.5)--(0,-0.5,0)--(2,-0.5)--cycle;
\draw[pattern=north west lines] (2,0.25)--(0,0.25)--(0,0.5)--(2,0.5)--cycle;
\draw[pattern=north west lines] (2,-0.25)--(0,-0.25)--(0,-0.5)--(2,-0.5)--cycle;
\end{tikzpicture}
 \captionof{figure}{Support of the derivatives of order $r_1$ of the extension operator on $\Gamma_{l}$ at the face $\Gamma_{l_1}$}
\end{center}
 Hence we can argue exactly as in Proposition \ref{extwavelet} and show that 
 \begin{align*}
   D^{\beta} \Ext_{\Gamma_{l,j}}^{r,u} g \in F_{p,p}^{\frac{n-l_1}{p},\rloc}(\R^n \setminus \overline{\Gamma}_{l_1,j'}) \text{ for } \beta \in \N_{l_1,j'}^n \text{ with }|\beta|=r_1.
 \end{align*}
 By the equivalent characterization for $F_{p,p}^{\frac{n-l_1}{p},\rloc}(\R^n \setminus \overline{\Gamma}_{l_1,j'})$ from Proposition \ref{rlocequi} this shows that $\Ext_{l,j}^{r,u} g$ fulfils $R_{l_1}^{r_1,p}$.
 
 In the second step we consider the case $l_1=l$. If $j\neq j'$, then we are in a situation as in Figure \ref{waveleter2}, for example take a look at $\Gamma_{l,0}$ and $\Gamma_{l,2}$ there. The extension operator is constructed on $\Gamma_{l,0}$. In this situation we have 
 \begin{align*}
  \dist(supp \ \Phi_m^{j,\alpha}, \Gamma_{l,2}) \gtrsim 2^{-j}.
 \end{align*}
This shows $\Ext_{\Gamma_{l,0}}^{r,u} g \in F_{p,p}^{s,\rloc}(\R^n \setminus \overline{\Gamma}_{l,2})$ and hence by Proposition \ref{rlocDiff} about the derivatives of elements of refined localization spaces we have
\begin{align*}
D^{\beta} \Ext_{\Gamma_{l,0}}^{r,u} g \in F_{p,p}^{s-|\beta|,\rloc}(\R^n \setminus \overline{\Gamma}_{l,2})=F_{p,p}^{\frac{n-l_1}{p},\rloc}(\R^n \setminus \overline{\Gamma}_{l,2}) \text{ for } \beta \in \N_{l,2}^n \text{ with }|\beta|=r_1.
 \end{align*}
As in the first step this shows that $\Ext_{\Gamma_{l,0}}^{r,u} g$ fulfils $R_{l}^{r_1,p}$, this time at $\Gamma_{l,2}$.
 
 If $l=l_1$ and $j=j'$, then we can argue as in the first step ($l_1>l$). This is covered by the arguments in Proposition \ref{extwavelet}. 
 \end{proof}

Before we are ready to prove the decomposition theorem for the cube - the generalization of Theorem 6.28 in \cite{Tri08}, we need a counterpart of the observations in Theorems \ref{Zerleger} and \ref{Zerlegercrit}, now going up dimension by dimension. For convenience we distinguish between the critical and the non-critical cases.

\begin{Lemma}[The non-critical dimensional decomposition]
\label{ZerlegerCube}
 Let $n \in \N$, $l \in \N_0$ with $1\le	 l\leq n-1$. Let $1\leq p < \infty$ and $1\leq q <\infty$. Let $s>0$, 
 \begin{align*}
  s-\frac{n-l}{p} \notin \N_0 \text{ and } r=\lfloor s-\frac{n-l}{p} \rfloor.
 \end{align*}
Then 
\begin{align*}
 \Frloc[\R^n \setminus \overline{\Gamma}_l] = \{f \in  \Frloc[\R^n \setminus \overline{\Gamma}_{l-1}]: \tr_{\Gamma_l}^r f=0 \}
\end{align*}
(no trace for $r=-1$, i.\,e.\ $s<\frac{n-l}{p}$).
\end{Lemma}
\begin{Proof}
 We will use the observations from the proof of Theorem \ref{Zerleger}. Clearly, the left-hand side belongs to the right-hand side by the same arguments as there. 
 
 To show the converse, it suffices to show Hardy inequalities at the boundaries $\Gamma_{l,j}$, see Proposition \ref{rlocequi}. Let $f \in  \Frloc[\R^n \setminus \overline{\Gamma}_{l-1}]$ with $\tr_{\Gamma_l}^r f=0$. We decompose $\R^n$ into two segments as in the following figure:
 
\begin{center}
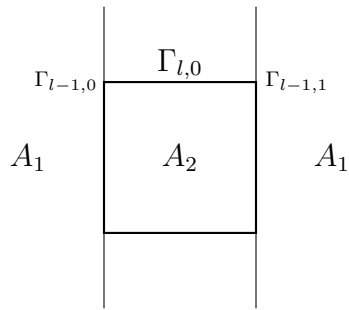

 \begin{tikzpicture}[scale=2] 
    \draw[thick] (0,0) -- (0,1) -- (1,1) -- (1,0) --cycle;
    \draw (0,-0.5) -- (0,1.5);  
    \draw (1,-0.5) -- (1,1.5);
    
    \draw (0.5,1.12) node  {$\Gamma_{l,0}$};
    \draw (-0.25,1) node[font=\footnotesize] {$\Gamma_{l-1,0}$};
    \draw (1.27,1) node[font=\footnotesize]  {$\Gamma_{l-1,1}$};
    \draw (-0.5,0.5) node  {$A_1$};
    \draw (0.5,0.5) node {$A_2$};
    \draw (1.5,0.5) node  {$A_1$};
    
 \end{tikzpicture}
   \captionof{figure}{The situation for the Hardy inequalities at $\Gamma_{l,j}$}
   \label{decompr}
\end{center}
So $A_1$ consists of all points $x \in \R^n$ with $\dist(x,\Gamma_{l,j})=\dist(x,\Gamma_{l-1,j'})$ for some $\Gamma_{l-1,j'}$ which is a boundary face of $\Gamma_{l,j}$. 

Let $d_{l,j}$ be the distance of $x$ to the boundary $\Gamma_{l,j}$ and $Q_{l,j,\eps}=\left\{x \in \R^n: d_{l,j}(x)<\eps\right\}$. Since we assumed $f \in  \Frloc[\R^n \setminus \overline{\Gamma}_{l-1}]$, we have Hardy inequalities at the boundary $\overline{\Gamma}_{l-1,j}$ by Proposition \ref{rlocequi}. This gives
\begin{align*}
 \|d^{-s}_{l,j}(\cdot)f|L_p(Q_{l,j,\eps} \cap A_1)\| \leq \sum_{k=0}^{n_{l-1}} \|d^{-s}_{l-1,k}(\cdot)f|L_p(Q_{l-1,j,\eps}\|.
\end{align*}
Thus we only have to care about $x \in A_2$. But then we are totally in the situation of Theorem $\ref{Zerleger}$ - the only difference is that we are now looking at a stripe (cylinder) constructed on $\Gamma_{l,j}$ instead of $\R^n$ considered over the plane $\R^l$. But this difference makes no problems: In the proof of Theorem \ref{Zerleger} we fixed $x' \in \R^l$ and took care about $x=(x',x'') \in \R^l \times \R^{n-l}$ - using mainly Fubini's Theorem \ref{Fubini} for $\FR$, for details see \eqref{Hardypoint}. We arrived at estimates for every fixed $x'\in \R^l$ and only at the end we integrated over $x' \in \R^l$. So we can transfer the arguments directly to the situation of a stripe instead of the whole $\R^n$ and we derive a Hardy inequality also for $x \in A_2$. 
\end{Proof}

\begin{Remark}
 The observations of the previous lemma also hold for $l=0$ in a adapted way. Then $\Gamma_{-1}:=\emptyset$ and so $\Frloc[\R^n \setminus \overline{\Gamma}_{-1}]=\FR$. Here we are exactly in the same situation as in Theorem \ref{Zerleger}, now considering all the corner points of $Q$ instead of only one point $x=0$ as there. This makes no problems since the Hardy inequalities are local observations - we only look at a neighbourhood of $\Gamma_{l,j}$ of $x \in \R^n$ with $\dist(x,\Gamma_{l,j})<\eps$ for a small $\eps>0$, see Remark \ref{indeps}.
\end{Remark}

\begin{Lemma}[The critical dimensional decomposition]
\label{ZerlegercritCube}
Let $n\in \N$, $l \in \N_0$ with $1\leq l\leq n-1$. Let $1\leq p <\infty$, $1\leq q <\infty$, $s>0$ and
\begin{align*}
 r=s-\frac{n-l}{p} \in \N_0.
\end{align*}
Then
\begin{align*}
 \Frloc[\R^n \setminus \overline{\Gamma}_l] = \{f \in  \Frloc[\R^n \setminus \overline{\Gamma}_{l-1}]: \tr_{\Gamma_l}^{r-1} f=0 \text{ and } f \text{ fulfils } R_l^{r,p}\}.
\end{align*}
(no trace for $r=0$, i.\,e.\ $s=\frac{n-l}{p}$).
\end{Lemma}
\begin{Proof}
 The proof is nearly the same as the proof of the previous lemma for the non-critical case, we only have to take care about the reinforce property $R_l^{r,p}$. At first, if $f \in 
 \Frloc[\R^n \setminus \overline{\Gamma}_l]$, then $f \in  \Frloc[\R^n \setminus \overline{\Gamma}_{l-1}]$ and $\tr_{\Gamma_l}^{r-1} f=0$ as before. Furthermore, by Proposition \ref{rlocDiff} we have
 \begin{align*}
  D^{\beta}f \in F_{p,q}^{\frac{n-l}{p},\rloc}(\R^n \setminus \overline{\Gamma}_{l}) \text{ for } \beta \in \N_{l}^n \text{ with }|\beta|=r.
 \end{align*}
Hence $f$ fulfils the reinforce property $R_l^{r,p}$ by Proposition $\ref{rlocequi}$.

To show that the right-hand side is contained in $\Frloc[\R^n \setminus \overline{\Gamma}_l]$, we have to show a Hardy inequality at $\Gamma_l$. We can argue as in the proof of the previous lemma: If $x \in A_1$ from Figure \ref{decompr}, then the Hardy inequality follows as before. 

For $x \in A_2$ we now use the arguments from Theorem \ref{Zerlegercrit} instead of Theorem \ref{Zerleger}. Since the arguments for proving the underlying Hardy inequality in Lemma \ref{HardyZerleger} were local (we fixed $x' \in \R^l$), we can transfer them from $\R^n$ (with the basement $\R^l$) to the stripe $A_2$ (with the basement $\Gamma_{l,j}$). Here we need the reinforce property $R_l^{r,p}$ at the components $\Gamma_{l,j}$ of $\Gamma_l$ as in Theorem \ref{Zerlegercrit}. 
\end{Proof}
\begin{Remark}
 The observations of the previous lemma also hold for $l=0$ in an adapted way. Then $\Gamma_{-1}:=\emptyset$ and so $\{f \in \Frloc[\R^n \setminus \Gamma_{-1}]: f \text{ fulfils } R_0^{r,p}\}=\{f \in \FR: f \text{ fulfils } R_0^{r,p}\}$. Here we are exactly in the same situation as in Theorem \ref{Zerlegercrit}, now considering all the corner points of $Q$ instead of only one point $x=0$ as there. This makes no problems since the Hardy inequalities are local observations.   
\end{Remark}

\subsection{The decomposition theorem for $\Frinf[Q]$}
Our goal is now the decomposition of $\Frinf[Q]$ in a similar way as in Theorem 6.28 of \cite{Tri08}, more exactly (6.117), where $\FO[Q]$ was decomposed in the non-critical cases. With the following theorem we will incorporate the non-critical values (where $\Frinf[Q]=\FO[Q]$) as well as the critical values. Hence the theorem will be a generalization of Theorem 6.28 in \cite{Tri08}.

We proved Proposition \ref{Frlocdecomp} for $r \in \N_0$ with $r<s-\frac{n-l}{p}$ and for every dimension $l$ separately. Now we will choose $r$ as large as possible for every dimension $l$ which is suggested by Theorem \ref{Zerleger} and Theorem \ref{Zerlegercrit}: Depending on the smoothness parameter $s$ and the parameter $p$, it is possible that we don't need to care about traces of small dimension (as they will not exist), see \cite[Section 6.1.5]{Tri08}. 

Let $s>0$. If $s>\frac{n}{p}$, then we define the starting dimension as $l_0=0$. Otherwise we define the starting dimension to be the value $l_0 \in \{1,\ldots,n\}$ such that
\begin{align}
\label{l0def}
 0< s-\frac{n-l_0}{p}\leq \frac{1}{p}
\end{align}
In particular, if $0<s\leq \frac{1}{p}$, then $l_0=n$. In this situation there are no traces at the boundary to be considered. 
Furthermore, for all dimensions $l$ bigger or equal to the starting dimension, i.\,e.\ $l_0\leq l\leq n-1$ we choose
\begin{align}
 \label{rldef}
r^l=\begin{cases} s-\frac{n-l}{p}-1&, s-\frac{n-l}{p} \in \N\\
\lfloor s-\frac{n-l}{p} \rfloor&, \text{otherwise} 
\end{cases}.
\end{align}
By definition of $l_0$ we always have $r^l\geq 0$. In the following theorem $r^l$ will stand for the maximal order of derivatives of traces on faces of dimension $l$, compare with Lemma \ref{ZerlegerCube} and Lemma \ref{ZerlegercritCube} - for convenience now with $r^l$ instead of $r^l-1$ in the critical cases. We will consider
\begin{align*}
 \tr_{\Gamma_l}^r: f \mapsto \{\tr_l D^{\alpha} f: \alpha \in \N_{l}^n, |\alpha|\leq r^l \}
\end{align*}
and the suitable extension operator $\Ext_{\Gamma_{l}}^{r^l,u}$ for all $l\in \N_0$ with $l_0\leq l \leq n-1$. This is the main idea to connect the results from Lemma \ref{ZerlegerCube}, Lemma \ref{ZerlegercritCube} and Proposition \ref{Frlocdecomp} to get a decomposition of $\Frinf[Q]$, using the decomposition idea of \cite[Theorem 6.28]{Tri08}.

\begin{Remark}
 We need to give some remarks about the situation if $l=0$. Then $\Gamma_l$ consists of isolated points. The spaces $
 F_{p,q}^{\sigma}(\Gamma_{0})$ and the reinforce property $R_0^{r,p}$ are well-defined. The space $F_{p,q}^{\sigma,\rloc}(\Gamma_{0})$ can be associated with the functional values at the points of $\Gamma_0$ - this space has a trivial wavelet basis consisting of functions equal to $1$ in one point and equal to $0$ in the other points of $\Gamma_0$. In the following it makes no problems just to define (or assume)
 \begin{align*}
 F_{p,q}^{\sigma}(\Gamma_{0})=\tilde{F\,}\!_{p,q}^{\sigma}(\Gamma_{0})=F_{p,q}^{\sigma,\rloc}(\Gamma_{0}).
\end{align*}
For more details for the case $l=0$ and the construction of the extension operator see \cite[Section 5.2.3]{Tri08}.
\end{Remark}

\begin{Theorem}
\label{Zerlegerwav}
Let $1\leq p <\infty$, $1\leq q<\infty$ and $0<s<u \in \N$. Let $n \in \N$, $l_0 \in \N_0$ with $0\leq l_0\leq n$ defined as in \eqref{l0def} and $r^l$ for $l_0\leq l\leq n-1$ defined as in \eqref{rldef}.\ Then it holds
\begin{align}
\label{ZerlegerFrinf}
\Frinf[Q]=\Frloc[Q] \times \prod_{l=l_0}^{n-1} \Ext_{\Gamma_{l}}^{r^{l},u} \prod_{\underset{|\alpha|\leq r^{l}}{\alpha \in \N_{l}^n}} F_{p,p}^{s-\frac{n-l}{p}-|\alpha|,\rloc}(\Gamma_{l}).
\end{align}
(complemented subspaces)
\end{Theorem}
\begin{Proof}
 The proof follows the steps in the proof of Theorem 6.28 in \cite{Tri08}. We use an induction argument on the dimension $l$. Let $\tilde{f} \in \Frinf[Q]$ and let $f \in \Frinf[\R^n \setminus \Gamma]$ be a continuation of $\tilde{f}$ onto $\R^n$, hence $f|{Q}=\tilde{f}$ (see Definition \ref{reinforcedQ}).
 
 First we start with dimension $l_0$. Let $l_0\geq 1$ (hence $s\leq \frac{n}{p}$). By definition of $l_0$ it holds
 \begin{align*}
 0<s-\frac{n-l_0}{p}\leq \frac{1}{p} \Rightarrow  \frac{n-l_0}{p}<s\leq \frac{n-(l_0-1)}{p}<\frac{n-(l_0-2)}{p}<\ldots < \frac{n}{p}. 
\end{align*}
If $s<\frac{n-(l_0-1)}{p}$, then by repeated application of Lemma \ref{ZerlegerCube} with no trace conditions (since $r^{k}=-1$ for $k=0,\ldots,l_0-1$) we have
\begin{align*}
 \FR= \Frloc[\R^n \setminus \overline{\Gamma}_0]=\ldots=\Frloc[\R^n \setminus \overline{\Gamma}_{l_0-1}]. 
\end{align*}
Here we are completely in the non-critical situation.

If $s=\frac{n-(l_0-1)}{p}$, then we repeatedly apply Lemma \ref{ZerlegerCube} and once Lemma \ref{ZerlegercritCube} for the critical case 
\begin{align*}
 r^{l_0-1}=0=s-\frac{n-(l_0-1)}{p}
\end{align*}
to get
\begin{align*}
 \{f\in \FR: f \text{ fulfils } R_{l_0-1}^{0,p}\}&=\ldots= \{f\in \Frloc[\R^n \setminus\overline{\Gamma}_{l_0-2}]: f \text{ fulfils } R_{l_0-1}^{0,p}\}\\
 &=\Frloc[\R^n \setminus \overline{\Gamma}_{l_0-1}].
\end{align*}
Here are also no trace conditions necessary.
Packing both cases together, we obtain
\begin{align*}
 \Frinf[\R^n \setminus \Gamma] \subset \Frloc[\R^n \setminus \overline{\Gamma}_{l_0-1}]
\end{align*}
since it is possible that $f \in \Frinf[\R^n \setminus \Gamma]$ may fulfil more reinforce properties $R_l^{r,p}$ for larger dimensions $l$.
Trivially, this inclusion is also true for $l_0=0$ with $\FR$ instead of $\Frloc[\R^n \setminus \overline{\Gamma}_{l_0-1}]$.

The crucial idea is now the decomposition
\begin{align*}
 f=f_1+f_2=\left(f- \left(\Ext_{\Gamma_{l_0}}^{r^{l_0},u} \circ \tr_{\Gamma_{l_0}}^{r^{l_0}}\right) f\right)+\left( \Ext_{\Gamma_{l_0}}^{r^{l_0},u} \circ \tr_{\Gamma_{l_0}}^{r^{l_0}}\right)  f.
\end{align*}
By construction of the extension operator we have 
\begin{align*} 
 \tr_{\Gamma_{l_0}}^{r^{l_0}} f_1= 0.
\end{align*}
Furthermore, if $f$ fulfils a reinforce property $R_{l}^{r^{l}+1,p}$ (occuring in the definition of $\Frinf[\R^n \setminus \Gamma]$) for some $l \in \N$ with $l_0\leq l<n$, then also $f_1=f-f_2$ fulfils this reinforce property $R_{l}^{r^{l}+1,p}$ by Lemma \ref{reinfprop} - here we now have to consider $r^l+1$ instead of $r^l$ having in mind the definition in \eqref{rldef}. 

In particular we have $f_1 \in \Frinf[\R^n \setminus \Gamma] \subset \Frloc[\R^n \setminus \overline{\Gamma}_{l_0-1}]$ with $\tr_{\Gamma_{l_0}}^{r^{l_0}} f_1= 0$. Hence we are in the situation of the dimensional decomposition Lemma \ref{ZerlegerCube} resp.\ for the critical cases Lemma \ref{ZerlegercritCube}. This shows
\begin{align*}
 f_1 \in \Frloc[\R^n \setminus \overline{\Gamma}_{l_0}].
\end{align*}
By construction it holds 
\begin{align*}
f_2 \in 
\Ext_{\Gamma_{l_0}}^{r^{l_0},u} \prod_{\underset{|\alpha|\leq r^{l_0}}{\alpha \in \N_{l_0}^n}} F_{p,p}^{s-\frac{n-l_0}{p}-|\alpha|,\rloc}(\Gamma_{l_0}). 
\end{align*}
Putting these two observations together we arrive at
\begin{align}
\label{decomprinf}
 \Frinf[\R^n \setminus \Gamma] \hookrightarrow \Frloc[\R^n \setminus \overline{\Gamma}_{l_0}] \times \Ext_{\Gamma_{l_0}}^{r^{l_0},u} \prod_{\underset{|\alpha|\leq r^{l_0}}{\alpha \in \N_{l_0}^n}} F_{p,p}^{s-\frac{n-l_0}{p}-|\alpha|,\rloc}(\Gamma_{l_0}).
\end{align}
The operator $P_{l_0}^{r^{l_0},u}=\Ext_{\Gamma_{l_0}}^{r^{l_0},u} \circ \tr_{\Gamma_{l_0}}^{r^{l_0}}$ is a linear, bounded projection ($P^2=P$) of
\begin{align*}
 \Frloc[\R^n \setminus \overline{\Gamma}_{l_0-1}] \text{ onto } \Ext_{\Gamma_{l_0}}^{r^{l_0},u} \prod_{\underset{|\alpha|\leq r^{l_0}}{\alpha \in \N_{l_0}^n}} F_{p,p}^{s-\frac{n-l_0}{p}-|\alpha|,\rloc}(\Gamma_{l_0})
\end{align*}
with $P_{l_0}^{r^{l_0},u}f=0$ for $f \in \Frloc[\R^n \setminus \overline{\Gamma}_{l_0}]$. Hence the spaces on the right-hand side of \eqref{decomprinf} are complemented to each other, i.\,e.\ the intersection only consists of $f\equiv 0$.

Now we argue by induction: As in the first step, where we decomposed $f \in \Frloc[\R^n \setminus \overline{\Gamma}_{l_0-1}]$, we now decompose $f_1 \in \Frloc[\R^n \setminus \overline{\Gamma}_{l_0}]$ in the same way using the trace operator $\tr_{\Gamma_{l_0+1}}^{r^{l_0+1}}$ and the wavelet-friendly extension operator $\Ext_{\Gamma_{l_0+1}}^{r^{l_0+1},u}$. The reinforce properties  $R_{l}^{r^{l}+1,p}$ of order $l_0+1\leq l < n$ remain true by Lemma \ref{reinfprop}. We have
\begin{align}
\label{Zerleg2}
 f_1=f_{11}+f_{12}=\left(f_1- \left(\Ext_{\Gamma_{l_0+1}}^{r^{l_0+1},u} \circ \tr_{\Gamma_{l_0+1}}^{r^{l_0+1}}\right) f_1\right)+\left(\Ext_{\Gamma_{l_0+1}}^{r^{l_0+1},u} \circ \tr_{\Gamma_{l_0+1}}^{r^{l_0+1}}\right) f_1
\end{align}
with $f_{11} \in \Frloc[\R^n \setminus \overline{\Gamma}_{l_0}]$, $f_{11}$ fulfils reinforce properties $R_{l}^{r^{l}+1,p}$ of order $l_0+1\leq l < n$ if $f$ did and $f_{11}$ fulfils reinforce properties $R_{l}^{r^{l}+1,p}$ of order $0\leq l\leq l_0$ if we are in a critical case. The last observation follows as in the first step of the proof of Theorem \ref{Zerlegercrit} as a property of the space $\Frloc[\R^n \setminus \overline{\Gamma}_{l_0}]$. Furthermore, we have $\tr_{\Gamma_{l_0+1}}^{r^{l_0+1}} f_{11} = 0$. 

Lemmata \ref{ZerlegerCube} resp.\ \ref{ZerlegercritCube} show that $f_{11} \in \Frloc[\R^n \setminus \overline{\Gamma}_{l_0+1}]$ and hence
\begin{align*}
  \Frinf[\R^n \setminus \Gamma] \hookrightarrow \Frloc[\R^n \setminus \overline{\Gamma}_{l_0+1}] \times \prod_{l=l_0}^{l_0+1} \Ext_{\Gamma_{l}}^{r^{l},u} \prod_{\underset{|\alpha|\leq r^{l}}{\alpha \in \N_{l}^n}} F_{p,p}^{s-\frac{n-l}{p}-|\alpha|,\rloc}(\Gamma_{l}).
\end{align*}
As before the the spaces on the right-hand side are complemented subspaces since $P_{l_0+1}^{r^{l_0+1},u}=\Ext_{\Gamma_{l_0+1}}^{r^{l_0+1},u} \circ \tr_{\Gamma_{l_0+1}}^{r^{l_0+1}}$ is a bounded projection.

Going on in this way, in the end we arrive at dimension $l=n-1$ and, using the trivial observation $\overline{\Gamma}_{n-1}=\Gamma$, we end up at
\begin{align*}
  \Frinf[\R^n \setminus \Gamma] \hookrightarrow \Frloc[\R^n \setminus \Gamma] \times \prod_{l=l_0}^{n-1} \Ext_{\Gamma_{l}}^{r^{l},u} \prod_{\underset{|\alpha|\leq r^{l}}{\alpha \in \N_{l}^n}} F_{p,p}^{s-\frac{n-l}{p}-|\alpha|,\rloc}(\Gamma_{l})
\end{align*}
with complemented subspaces on the right-hand side.

We showed that $f \in \Frinf[\R^n \setminus \Gamma]$ belongs to the right-hand side. If we restrict $f$ and its decomposition to $Q$ (which is $\tilde{f}$), we get
\begin{align*}
 \tilde{f} \in \Frloc[Q] \times \prod_{l=l_0}^{n-1} \Ext_{\Gamma_{l}}^{r^{l},u} \prod_{\underset{|\alpha|\leq r^{l}}{\alpha \in \N_{l}^n}} F_{p,p}^{s-\frac{n-l}{p}-|\alpha|,\rloc}(\Gamma_{l}).
\end{align*}
This follows from the fact that the wavelet decomposition of $\Frloc[\R^n \setminus \Gamma]$ is divided into two totally decoupled decomposition parts - one part outside and one part inside the cube $Q$. So we have shown
\begin{align*}
 \Frinf[Q] \hookrightarrow \Frloc[Q] \times \prod_{l=l_0}^{n-1} \Ext_{\Gamma_{l}}^{r^{l},u} \prod_{\underset{|\alpha|\leq r^{l}}{\alpha \in \N_{l}^n}} F_{p,p}^{s-\frac{n-l}{p}-|\alpha|,\rloc}(\Gamma_{l})
\end{align*}
with complemented subspaces on the right-hand side. 

At the end of the proof we want to show that our constructed spaces on the right-hand side belong to $\Frinf[Q]$. We have $\Frloc[Q] \subset \Frinf[Q]$: On the one hand side we always have $\Frloc[Q] \subset \FO[Q]$. On the other hand side $f \in \Frloc[Q]$ always fulfils reinforce properties $R_l^{r,p}$ for every dimension $0\leq l<n$ for which they are necessary, for an argument see the first step of the proof of Theorem \ref{Zerlegercrit}. 

Furthermore, as we decomposed $f=f_1+f_2$ we showed that $f_1$ (and hence also $f_2$) fulfil reinforce properties $R_{l}^{r^l+1,p}$ for all $l \in \N_0$ with $0\leq l\leq n-1$ for which $f$ did. Since by Proposition \ref{Frlocdecomp} the restriction of the extension operators $\Ext_{\Gamma_{l}}^{r^{l},u}$ to the cube $Q$ always maps into at least $\FO[Q]$, we have $f_1,f_2 \in \Frinf[Q]$. If we now go on by induction as in the construction before, we can show that all elements of the decomposition of $f$ belong to $\Frinf[Q]$. These observations finish the proof and we get
\begin{align*}
 \Frinf[Q] = \Frloc[Q] \times \prod_{l=l_0}^{n-1} \Ext_{\Gamma_{l}}^{r^{l},u} \prod_{\underset{|\alpha|\leq r^{l}}{\alpha \in \N_{l}^n}} F_{p,p}^{s-\frac{n-l}{p}-|\alpha|,\rloc}(\Gamma_{l})
\end{align*}
with complemented subspaces on the right-hand side. 
\end{Proof}
\begin{Remark}
 We slightly modified the definition of the wavelet-friendly extension operator in comparison to \cite[Theorem 6.28]{Tri08}. Triebel took the extension operators $\Ext_{\Gamma_{l}}^{r^{l},u}$ for fixed dimension $l \in \{l_0,\ldots,n-1\}$ and constructed one extension operator $\Ext_{\Gamma}^{\overline{r},u}$ altogether. But in Remark 6.27 and Theorem 6.28 in \cite{Tri08} there were some inaccuracies emerging from the interplay of the traces and extension operators at different dimensions $l$. 
 
 Now we want to give a construction of an extension operator $\Ext_{\Gamma}^{\overline{r},u}$ for all dimensions $l \in \{l_0,\ldots,n-1\}$ at once - here $r=\{r^{l_0},\ldots,r^{n-1}\}$: Let
 \begin{align*}
  \{g_{l,j,\alpha}: l_0\leq l \leq n-1, j \in \{0,\ldots,n_l\}, \alpha \in \N_{l,j}^n, |\alpha|\leq r \} \in  \prod_{l=l_0}^{n-1} \prod_{\underset{|\alpha|\leq r^{l}}{\alpha \in \N_{l}^n}} F_{p,p}^{s-\frac{n-l}{p}-|\alpha|,\rloc}(\Gamma_{l})   
 \end{align*}
 be given. We want to have the typical identity property of the extension operator, namely
 \begin{multline*}
  \left(\tr_{\Gamma_l}^{r^l} \circ \Ext_{\Gamma}^{\overline{r},u}\right)  \{g_{l,j,\alpha}: l_0\leq l \leq n-1, j \in \{0,\ldots,n_l\}, \alpha \in \N_{l,j}^n, |\alpha|\leq r \} \\ 
  =   \{g_{l,j,\alpha}: j \in \{0,\ldots,n_l\}, \alpha \in \N_{l,j}^n, |\alpha|\leq r \}
 \end{multline*}
for every $l \in \{l_0,\ldots,n-1\}$.
 
 At first, by the construction of the extension operator $\Ext_{\Gamma_{l}}^{r^{l},u}$ and its mapping properties - see Proposition \ref{Frlocdecomp} - we have
 \begin{align*}
  \left(\tr_{\Gamma_{l}}^{r^{l}} \circ \Ext_{\Gamma_{l_1}}^{r^{l_1},u}\right)  \{g_{l_1,j,\alpha}: j \in \{0,\ldots,n_{l}\}, \alpha \in \N_{l_1,j}^n, |\alpha|\leq r \} = 0
 \end{align*}
for $l_0\leq l<l_1\leq n-1$. This means that higher dimensional extensions do not give traces at boundaries $\Gamma_l$ of smaller dimension. But extension operators of smaller dimension will influence the traces at boundaries of higher dimension. 

The construction of the all-dimensional extension operator is adapted by the decomposition in the proof of Theorem \ref{Zerlegerwav} and goes as follows: We start with the extension at the boundary $\Gamma_{l_0}$ of lowest dimension
\begin{align*}
 f_{l_0}= \Ext_{\Gamma_{l_0}}^{r^{l_0},u} \{g_{l_0,j,\alpha}: j \in \{0,\ldots,n_{l_0}\}, \alpha \in \N_{l_0,j}^n, |\alpha|\leq r \} \in \Frinf[Q].
\end{align*}
The observation $f_{l_0} \in \Frinf[Q]$ follows from Lemma \ref{reinfprop} and the assumption $g_{l_0,j,\alpha} \in F_{p,p}^{s-\frac{n-l_0}{p}-|\alpha|,\rloc}(\Gamma_{l_0})$ - one can use the same arguments as in the proof of Theorem \ref{Zerlegerwav}, namely directly after \eqref{Zerleg2}. We have
\begin{align*}
 \tr_{\Gamma_{l_0}}^{r^{l_0}} f_{l_0}= \{g_{l_0,j,\alpha}: j \in \{0,\ldots,n_{l_0}\}, \alpha \in \N_{l_0,j}^n, |\alpha|\leq r \}. 
\end{align*}
In the second step we consider the extensions at the boundary $\Gamma_{l_0+1}$. But now we have to be careful about the influence of $f_{l_0}$ at $\Gamma_{l_0+1}$. We construct
\begin{align*}
 f_{l_0+1}&= \Ext_{\Gamma_{l_0+1}}^{r^{l_0+1},u} \{g_{l_0+1,j,\alpha}: j \in \{0,\ldots,n_{l_0+1}\}, \alpha \in \N_{l_0+1,j}^n, |\alpha|\leq r \} \\
   & - \left( \Ext_{\Gamma_{l_0+1}}^{r^{l_0+1},u} \circ  \tr_{\Gamma_{l_0+1}}^{r^{l_0+1}}\right) f_{l_0}.
\end{align*}
Since the extension operator $\Ext_{\Gamma_{l_0+1}}^{r^{l_0+1},u}$ at the boundary $\Gamma_{l_0+1}$ has no influence on the trace $\tr_{\Gamma_{l_0}}^{r^{l_0}}$ at the boundary $\Gamma_{l_0}$ as mentioned before, we get
\begin{align*}
 \tr_{\Gamma_{l_0}}^{r^{l_0}} \left( f_{l_0} + f_{l_0+1} \right) =  \tr_{\Gamma_{l_0}}^{r^{l_0}} f_0= \{g_{l_0,j,\alpha}: j \in \{0,\ldots,n_{l_0}\}, \alpha \in \N_{l_0,j}^n, |\alpha|\leq r \}
\end{align*}
and
\begin{align*}
 \tr_{\Gamma_{l_0+1}}^{r^{l_0+1}} \left( f_{l_0} + f_{l_1} \right) &= \{g_{l_0+1,j,\alpha}: j \in \{0,\ldots,n_{l_0+1}\}, \alpha \in \N_{l_0+1,j}^n, |\alpha|\leq r \} \\
 &- \left( \tr_{\Gamma_{l_0+1}}^{r^{l_0+1}} \circ \Ext_{\Gamma_{l_0+1}}^{r^{l_0+1},u} \circ  \tr_{\Gamma_{l_0+1}}^{r^{l_0+1}}\right) f_{l_0} \\
 &+ \tr_{\Gamma_{l_0+1}}^{r^{l_0+1}}  f_{l_0} \\
 &= \{g_{l_0+1,j,\alpha}: j \in \{0,\ldots,n_{l_0+1}\}, \alpha \in \N_{l_0+1,j}^n, |\alpha|\leq r \}.
\end{align*}
The rest of the construction can be done by inductively. At the end of the construction we get a function $f=f_{l_0}+\ldots + f_{n-1} \in \Frinf[Q]$ with 
\begin{align*}
 \tr_{\Gamma_{l}}^{r^{l}} f =  \tr_{\Gamma_{l}}^{r^{l}} \left(f_{l_0}+\ldots+f_{l} \right)
  = \{g_{l,j,\alpha}: j \in \{0,\ldots,n_{l}\}, \alpha \in \N_{l,j}^n, |\alpha|\leq r \}.
\end{align*}
Hence we have constructed a (linear and bounded) extension operator $\Ext_{\Gamma}^{\overline{r},u}$ mapping from
\begin{align*}
 \prod_{l=l_0}^{n-1} \prod_{\underset{|\alpha|\leq r^{l}}{\alpha \in \N_{l}^n}} F_{p,p}^{s-\frac{n-l}{p}-|\alpha|,\rloc}(\Gamma_{l}) \text{ 
into } \Frinf[Q]
\end{align*} 
with 
\begin{align*}
 (\tr_{\Gamma_{l_0}}^{r^{l_0}},\ldots,\tr_{\Gamma_{n-1}}^{r^{n-1}}) \circ \Ext_{\Gamma}^{\overline{r},u}= \text{id}, \text{ identity on }  \prod_{l=l_0}^{n-1} \prod_{\underset{|\alpha|\leq r^{l}}{\alpha \in \N_{l}^n}} F_{p,p}^{s-\frac{n-l}{p}-|\alpha|,\rloc}(\Gamma_{l}).   
\end{align*}
Looking directly into the construction in the proof of Theorem \ref{Zerlegerwav}, we can reformulate this theorem with the newly constructed extension operator:
\end{Remark}
\begin{Theorem}
 Let $1\leq p <\infty$, $1\leq q<\infty$ and $0<s<u \in \N$. Let $n \in \N$, $l_0 \in \N_0$ with $0\leq l_0\leq n$ defined as in \eqref{l0def}, $r^l$ for $l_0\leq l\leq n-1$ defined as in \eqref{rldef} and $\overline{r}=\{r^{l_0},\ldots,r^{n-1}\}$.  Then it holds
\begin{align*}
\Frinf[Q]=\Frloc[Q] \times \Ext_{\Gamma}^{\overline{r},u} \prod_{l=l_0}^{n-1}  \prod_{\underset{|\alpha|\leq r^{l}}{\alpha \in \N_{l}^n}} F_{p,p}^{s-\frac{n-l}{p}-|\alpha|,\rloc}(\Gamma_{l}).
\end{align*}
(complemented subspaces)
\end{Theorem}
\begin{Proof}
 The proof is nearly the same as the proof of Theorem \ref{Zerlegerwav}, even simpler. But now we directly start with the decomposition
 \begin{align*}
  f=f_1+f_2 = \left(f - \Ext_{\Gamma}^{\overline{r},u} \circ (\tr_{\Gamma_{l_0}}^{r^{l_0}},\ldots,\tr_{\Gamma_{n-1}}^{r^{n-1}}) \right) + \Ext_{\Gamma}^{\overline{r},u} \circ (\tr_{\Gamma_{l_0}}^{r^{l_0}},\ldots,\tr_{\Gamma_{n-1}}^{r^{n-1}}).
 \end{align*}
 By construction this is the final decomposition and it is exactly the same decomposition of $f$ as in Theorem \ref{Zerlegerwav}: We have $\tr_{\Gamma_{l}}^{r^{l}} f_1 =0$ for all $l \in \{l_0,\ldots,n-1\}$. Hence, arguing as there (inductively) we have $f_1 \in \Frloc[Q]$ and $f_2$ originates from the trace spaces $F_{p,p}^{s-\frac{n-l}{p}-|\alpha|,\rloc}(\Gamma_{l})$. Furthermore, $f_2$ and hence also $f_1$ belong to $\Frinf[Q]$. These two observations finish the proof.
\end{Proof}

\section{Riesz bases for $\Frinf[Q]$}
\subsection{The main theorem for Riesz bases on $\Frinf[Q]$}
Now we are in the same situation as in \cite[Section 6.1.6]{Tri08}. We have decomposed $\Frinf[Q]$ in Theorem \ref{Zerlegerwav} into spaces having orthonormal $u$-wavelet bases (by Definition \ref{u-basis}) which are at the same time $u$-Riesz bases by Definition \ref{u-Riesz}:
\begin{align*}
\Frinf[Q]=\Frloc[Q] \times \prod_{l=l_0}^{n-1} \Ext_{\Gamma_{l}}^{r^{l},u} \prod_{\underset{|\alpha|\leq r^{l}}{\alpha \in \N_{l}^n}} F_{p,p}^{s-\frac{n-l}{p}-|\alpha|,\rloc}(\Gamma_{l}).
\end{align*}
The existence of a wavelet basis for the spaces on the right-hand side follows from Theorem \ref{rlocwavelet} where we investigated the spaces $\Frloc$ for arbitrary domains $\Om \subset \R^n$. The wavelet-friendly extension operator in the decomposition transfers the wavelet blocks of the spaces $F_{p,p}^{\sigma,\rloc}(\Gamma_{l})$ on the boundary of dimension $l$ to functions on the cube $Q$ which behave like atoms and are totally covered by the Definition \ref{u-waveletr} of oscillating $u$-wavelet systems. This shows that we can construct an oscillating $u$-Riesz basis (by Definition \ref{u-Riesz}) for $\Frinf[Q]$ which is a $u$-wavelet system on $Q$.

\begin{Theorem}
\label{waveletdecomp}
 Let $Q$ be the unit cube in $\R^n$ for $n\geq 2$. Let 
 \begin{align*}
  s>0, \quad 1\leq p <\infty \quad \text{and} \quad 1\leq q<\infty.
 \end{align*}
 Then $\Frinf[Q]$ has an oscillating $u$-Riesz basis for any $u \in \N_0$ with $u>s$. The related sequence space is $\fO[\overline{Q}]$ introduced in Definition \ref{extsequence}. 
\end{Theorem}
\begin{proof}
 The proof is the same as the proof of Theorem 6.30 in \cite{Tri08} now with the spaces $\Frinf[Q]$ instead of $\AO[Q]$. The essential idea is decomposition \eqref{ZerlegerFrinf} - the spaces on the right-hand side admit interior $u$-Riesz bases. The wavelet bases on the boundary spaces $F_{p,p}^{\sigma,\rloc}(\Gamma_{l})$ are extended to boundary components of a $u$-wavelet system by the wavelet-friendly extension operator. Hence we now need to use the sequence space $\fO[\overline{\Om}]$ incorporating the boundary values.
\end{proof}

\subsection{Further remarks and generalizations}
\begin{Remark}
 In the decomposition Theorem \ref{Zerlegerwav} and the $u$-Riesz basis Theorem \ref{waveletdecomp} we assumed $1\leq q< \infty$. The exclusion of $q=\infty$ is natural. But we can include the values $q<1$ under additional conditions. The first ideas in this direction are collected in Remark \ref{ZerlegerRem}. In the proof of Theorem \ref{Zerlegerwav} we essentially used Proposition \ref{Frlocdecomp}, Lemma \ref{reinfprop} and Lemma \ref{ZerlegerCube} resp.\ Lemma \ref{ZerlegercritCube}. Lemma $\ref{reinfprop}$ does not depend on $q$ since the reinforce properties $R_l^{r,p}$ do not depend on $q$. Taking a short look into the proof of Proposition \ref{Frlocdecomp} it also holds true if $0<q<1$ with the additional property $s>\sigma_{p,q}$. Lemma \ref{reinfprop} and Lemma \ref{ZerlegerCube} resp.\ Lemma \ref{ZerlegercritCube} mainly depend on the ideas of Theorem \ref{Zerleger} resp.\ Theorem \ref{Zerlegercrit}. In Remark \ref{ZerlegerRem} we discussed that these two theorems are also valid for $0<q<1$ with 
$s>\
\sigma_{p,q}$. 
 
 Hence also Theorem \ref{Zerlegerwav} and Theorem \ref{waveletdecomp} are valid for $0<q<1$ if $s>\sigma_{p,q}$: Every space on the right-hand side of the decomposition of $\Frinf[Q]$ in \eqref{ZerlegerFrinf} has a wavelet basis by Theorem \ref{rlocwavelet} if $0<q<\infty$ and $s>\sigma_{p,q}$.
\end{Remark}
\begin{Corollary}
 Let $Q$ be the unit cube in $\R^n$ for $n\geq 2$. Let 
 \begin{align*}
  1\leq p <\infty, \quad 0<q<\infty \quad \text{and} \quad s>\sigma_{p,q}.
 \end{align*}
 Then $\Frinf[Q]$ has an oscillating $u$-Riesz basis for any $u \in \N_0$ with $u>s$. The related sequence space is $\fO[\overline{Q}]$ introduced in Definition \ref{extsequence}. 
\end{Corollary}

\begin{Remark}
 As in \cite[Corollary 6.31]{Tri08} one can extend the construction of a $u$-Riesz basis from spaces $\Frinf[Q]$ on a cube $Q$ to $\Frinf[P]$ on a polyhedron. The intersecting angle of two different faces must not be zero. The definition of $\Frinf[P]$ can be given analogously to the definition of $\Frinf[Q]$.
\end{Remark}

\begin{Remark}
We have derived a decomposition of $\Frinf[Q]$ into spaces which have wavelet bases in Theorem \ref{Zerlegerwav} - in the same way as in \cite[Theorem 6.28]{Tri08}. In Definition \ref{reinforcedQ} we defined $\Frinf[Q]$ a bit unnaturely - first we asked for reinforce properties $R_l^{r^l,p}$ for a continuation of $f$ onto the whole $\R^n$ such that reinforce properties were assumed from both sides of the boundary $\Gamma_l$ - from outside as well as from inside the cube. Hence we cannot decide whether $f \in \Frinf[Q]$ just by the intrinsic situation, i.\,e.\ knowing the behaviour of $f$ inside the cube $Q$.

But there is a natural way to ask for reinforce properties only inside the cube $Q$ without using the restriction of a continuation onto $\R^n$. We will give a modified definition of a reinforced space and will 	explain the relation of both.
\end{Remark}
\begin{Definition}
Let $Q$ be the unit cube and $\Gamma=\partial Q$ its boundary. As introduced at the beginning of this chapter let
\begin{align*}
 d_{l,j}(x)=\dist(x,\Gamma_{l,j}) \text{ and } Q_{l,j,\eps}:=\left\{x \in \R^n: d_{l,j}(x)<\eps\right\}.
\end{align*} 
We define
 \begin{align*}
  Q_{l,j,\eps}^* = Q_{l,j,\eps} \cap Q.
 \end{align*}
Let $n \in \N$ and $l\in \N_0$ with $l<n$. Let $1\leq p < \infty$, $0<q<\infty$ and 
\begin{align*}
s-\frac{n-l}{p}=r \in \N_0.
\end{align*}
Then $f \in \FR$ is said to fulfil the interior reinforce property $R_l^{r,p,*}$ if, and only if, 
\begin{align*}
 d_{l,j}^{-\frac{n-l}{p}} \cdot D^{\alpha} f \in L_p(Q_{l,j,\eps}^*) \text{ for all } \alpha \in \N_{l,j}^n, |\alpha|=r \text{ and } j=1,\ldots,n_{l}.
\end{align*}
 Furthermore, let $s \in \R$. Then 
\begin{multline*}
\Frinf[Q]^*:=\\\{f \in \FO[Q]: f \text{ fulfilfs } R_l^{r^l,p,*} \text{ for all } l\in \{0,\ldots,n-1\} \text{ with } r^l=s-\frac{n-l}{p} \in \N_0 \}.
\end{multline*}
\end{Definition}
\begin{Remark}
\label{NewDef}
 We now introduced $\Frinf[Q]^*$ assuming reinforce properties which are totally intrinsic (interior). This means that we only look for the decay of $f \in \Frinf[Q]^*$ at the boundary $\Gamma_l$ from the inside of the cube $Q$. 
 
 If $f \in \Frinf[Q]$, then there is by definition a continuation $\tilde{f}$ of $f$ which belongs to $\Frinf[\R^n \setminus \Gamma]$. Then $\tilde{f}$ fulfils reinforce properties at $\Gamma_l$ (if necessary) and surely also its restriction $f$ fulfils interior reinforce properties at $\Gamma_l$. Hence we have
 \begin{align*}
  \Frinf[Q] \hookrightarrow \Frinf[Q]^*.
 \end{align*}
Thus the following statement would be nice to have.
\end{Remark}
\begin{Conjecture}
\label{Conj}
 Let $Q$ be the unit cube in $\R^n$ for $n\geq 2$. Let 
 \begin{align*}
  1\leq p <\infty, \quad 0<q<\infty \quad \text{and} \quad s>\sigma_{p,q}.
 \end{align*}
 Then 
 \begin{align*}
  \Frinf[Q] = \Frinf[Q]^*.
 \end{align*}
\end{Conjecture}
\begin{Remark}
 One idea to proof this conjecture is to show that $\Frinf[Q]^*$ admits the same decomposition as $\Frinf[Q]$ from Theorem \ref{Zerlegerwav}, refined using the all-dimensional extension operator getting 
 \begin{align*}
\Frinf[Q]=\Frloc[Q] \times \Ext_{\Gamma}^{\overline{r},u} \prod_{l=l_0}^{n-1}  \prod_{\underset{|\alpha|\leq r^{l}}{\alpha \in \N_{l}^n}} F_{p,p}^{s-\frac{n-l}{p}-|\alpha|,\rloc}(\Gamma_{l})
\end{align*}
 in \eqref{ZerlegerFrinf}. 
 
 A second possibility would be the following: Construct or use an existing (linear, bounded) extension operator 
 \begin{align*} 
  \Ext_Q: \FO[Q] \rightarrow \FR
 \end{align*}
with $\Ext_Q f|Q=f$ and show that $\Ext_Q f$ fulfils $R_l^{r^l,p}$ if $f$ fulfils $R_l^{r^l,p,*}$ (with a suitable norm estimate). This would proof the conjecture.
 
\end{Remark}

\section{An example - the reinforced function space $W_2^{1,\rinf}(Q)$}
\label{Wrreinf}
Let $Q$ be the unit cube in dimension $n\geq 2$. The most promiment example which is an exceptional space on the cube $Q$ in our sense is the Sobolev space $W_2^{1}(Q)$. By classical observations it holds
\begin{align*}
 F_{2,2}^1(Q)=W_2^{1}(Q)=\{f \in L_p(Q): \|f|W_2^1(Q)\| <\infty \}
\end{align*}
with
\begin{align*}
 \|f|W_2^1(Q)\| = \sum_{|\alpha|\leq 1} \|D^{\alpha} f |L_2(Q)\|.
\end{align*}
Here we have (in our notation)
\begin{align*}
 s-\frac{n-(n-2)}{p}= 1-\frac{2}{2}=0
\end{align*}
and
\begin{align*}
 s-\frac{n-k}{p} \notin \N_0 \text{ for } k \in \{1,3,4,\ldots,n\}.
\end{align*}
For these exceptional spaces one cannot use the method of \cite[Section 6.15]{Tri08} where one decomposes the space $\FO[Q]$ into spaces with interior wavelet basis and spaces emerging from the traces of $f$ at the boundaries $\Gamma_l$. This is observed more deeply in \cite[Remark 5.50]{Tri08} and goes back to observations in \cite{Gri85} and \cite{Gri92} at least for dimension $n=2$: Let $\Gamma=\partial \Om= I_1 \cup I_2 \cup I_3 \cup I_4$, the four sides of the cube. Then the trace space $\tr_{\Gamma} W_2^1(Q)$ is the collection of all tuples $g=(g_1,g_2,g_3,g_4)$ with
\begin{align*}
 g_{l} \in H^{\frac{1}{2}}(I_{l}), \quad l=1,2,3,4
\end{align*}
and 
\begin{align*}
 \int_0^{1/2} \frac{|g_1(t)-g_2(t)|^2}{t} \ dt <\infty, \text{ etc.\ }
\end{align*} 
Hence the traces at different faces of the same dimension interfere with each other which makes the construction procedure in \cite[Section 6.15]{Tri08} and also in our Theorem \ref{Zerlegerwav} impossible - there is no total decoupling of the traces. Until now there seem to be no construction given for a $u$-Riesz basis for $W_2^{1}(Q)$, at least using Definition \ref{u-Riesz} of a $u$-Riesz basis.

In \cite[Section 6.2.4]{Tri08} Triebel suggested the modification of the Triebel-Lizorkin spaces by desired Hardy inequalities at the boundaries. This was the starting point of the observations in this chapter. As an example Triebel observed the behaviour of the space $W_2^{1,\rinf}(Q):=F_{2,2}^{1,\rinf}(Q)$ for dimension $n=2$. This is now a special case of Theorem \ref{Zerlegerwav}:
\begin{Theorem}
  Let $Q$ be the unit cube in $\R^2$. Then $W_2^{1,\rinf}(Q)$ has an oscillating $u$-Riesz basis for any $u \in \N_0$ with $u>1$.
\end{Theorem}
But the definition of $W_2^{1,\rinf}(Q)$ is not intrinsic - one has to care about the reinforce properties of a function extended from $Q$ to $\R^n$. The definition
\begin{align*}
 W_2^{1,\rinf}(Q)^*:=\left\{f \in W_2^{1}(Q): \int_Q |f(x)|^2 \frac{ dx }{d_0(x)^2}<\infty \right\}
\end{align*}
would be more natural. Here $d_0(x)$ is the distance of $x$ to the corner points, i.\,e.\ $d_0(x)=\dist(x,\Gamma_0)$. The reinforce property is now demanded of $f$ itself, not of its derivatives since $r^{l_0}=0$. 

The following would be desirable:
\textit{ Let $Q$ be the unit cube in $\R^2$. Then $W_2^{1,\rinf}(Q)^*$ has an oscillating $u$-Riesz basis for any $u \in \N_0$ with $u>1$.}

But this is not covered by Theorem \ref{waveletdecomp}. The definition of $W_2^{1,\rinf}(Q)^*$ slightly differs from our definition of reinforced function spaces in Definition \ref{reinforcedQ} and coincide with the Definition of $\Frinf[Q]^*$ for $s=1,p=2,q=2$ in the previous Remark \ref{NewDef}. To use Theorem \ref{waveletdecomp} we would have to show
 \begin{align*}
  W_2^{1,\rinf}(Q)^* = F_{2,2}^{1,\rinf}(Q)
 \end{align*}
 which is a special case of Conjecture \ref{Conj}. Looking into the definition of both spaces we have to show the following: Let $f \in W_2^{1,\rinf}(Q)^*$, i.\,e.\ $f$ fulfils reinforce properties at $\Gamma_0$. Then there is a continuation $\tilde{f}$ of $f$ onto the whole $\R^2$ such that 
\begin{align*}
 \|\tilde{f} | F_{2,2}^{1,\rinf}(\R^n \setminus \Gamma)\| \sim \|f|  W_2^{1,\rinf}(Q) \|. 
\end{align*}
This problem in connection with function spaces is usually called the extension problem, see \cite[Chapter 4]{Tri08} and \cite[Section 1.11.5]{Tri06}. There is an extension operator $\Ext_Q$ for $W_2^{1}(Q)$ by direct construnction, see \cite[Chapter VI]{Ste70}. Hence we have
\begin{align*}
 \|\Ext_Q f |W_2^{1}(\R^n)\| \lesssim \|f|W_2^1(Q)\| \text{ and } \Ext_Q f|Q=f.
\end{align*}
The question is: Does $\Ext_Q f$ fulfil reinforce properties $R_l^{r,p}$ at $\Gamma_l$ from inside and outside the cube $Q$ if $f$ fulfils reinforce properties $R_l^{r,p}$ at $\Gamma_l$ from the inside of $Q$? This problem remains open.

%% file: Open_Problems.tex
\label{Open}

In the following we want to discuss some of the observations made in the previous chapter and talk about some problems which remain unsolved. The first unsolved problems were discussed in Remark \ref{NewDef} and Section \ref{Wrreinf}. This chapter is some kind of summary of the things one could look at in the future.

\section{Necessity of reinforced properties}

Let $n \in \N$ with $n\geq 2$ and $Q$ be the unit cube in $\R^n$. So far we constructed $u$-Riesz bases for the spaces $\Frinf[Q]$. For the non-exceptional values these spaces coincide with the usual spaces $\FO[Q]$ but in the critical cases, i.\,e.\ when $s-\frac{k}{p} \in \N_0$ for $k \in \{1,\ldots,n\}$ these spaces are defined as subsets of $\FO[Q]$. 

The first natural question we have to ask is whether these additional reinforce properties $R_l^{r,p}$ are really necessary. At first this means we have to show that there exists a function $f \in \FO[Q]$ such that $f \notin \Frinf[Q]$. But this can be obtained in the same way as the results in Remark \ref{Frinfsmaller}. For instance, we trivially have
\begin{align*}
 f \equiv 1 \in \FO[Q]
\end{align*}
for all $s>0$, $1\leq p <\infty$ and $0<q<\infty$. On the other hand
\begin{align*}
 f \notin \Frinf[Q] \text{ for } s=\frac{1}{p},\ldots, \frac{n}{p}.
\end{align*}
This can be proven in the same way as the observations in Remark \ref{Frinfsmaller}. Here we only need to consider one of the faces, say for convenience
\begin{align*}
\Gamma_{n-1,0}=\{(x_1,\ldots,x_n)\in\R^n: 0\leq x_m \leq 1 \text{ for } 1\leq m \leq n-1 \text{ and } x_n=0\}.
\end{align*}
 By taking $f(x)=x_n^r$ with $r \in \N_0$ where $x=(x_1,\ldots,x_n)$ we can prove
\begin{align*}
 F_{p,q}^{r+\frac{k}{p},\rinf}(Q) \subsetneq F_{p,q}^{r+\frac{k}{p}}(Q)
\end{align*}
for every $k \in \{1,\ldots,n\}$. 85We showed that in every exceptional case the space $\Frinf[Q]$ is a proper subset of $\FO[Q]$.

The second question is whether all reinforced properties are always necessary - if say $s-\frac{n-l_1}{p}=r^{l_1} \in \N_0$ and $s-\frac{n-l_2}{p}=r^{l_2} \in \N_0$, are both reinforced properties $R_{l_1}^{r^{l_1},p}$ and $R_{l_2}^{r^{l_2},p}$ necessary to require in the definition of $\Frinf[Q]$? The answer is probably yes but a proof is rather tricky. The problem is that we now have to deal with more than one face of dimension $l_2$ such that there are a lot of different directions of derivatives to consider. It is not easy to show that the counterexamples from Remark \ref{Frinfsmaller} are only counterexamples for explicitly one reinforced property $R_{l_1}^{r^{l_1},p}$ and fulfil the others.

\section{The incorporation of Haar wavelets and of Besov spaces}
Let $1\leq p<\infty$, $1\leq q<\infty$ and
\begin{align*}
 0<s< \min\left(\frac{1}{p},\frac{1}{q}\right).
\end{align*}
In \cite[Theorem 2.26]{Tri10} Triebel showed that the Haar wavelet basis of $L_p(Q)$ is an interior $0$-Riesz basis for $\FO[Q]$ with the related sequence space $\fO[Q]$. It is possible to incorporate the Haar wavelet in Theorem \ref{waveletdecomp} for $1\leq p<\infty$, $1\leq q<\infty$ and $s< \min(\frac{1}{p},\frac{1}{q})$. Since $s<\frac{1}{p}$ there are no values on the boundary and hence also our constructed $u$-Riesz basis is interior. 

For $s\leq 0$ the spaces $\Frinf[Q]=\FO[Q]$ have interior $u$-Riesz basis, even for $0<p<\infty$ and $0<q<\infty$, with $u$ sufficiently large in dependence of $s$, $p$ and $q$. This follows from \cite[Theorem 5.43]{Tri08}. Hence together with Theorem \ref{waveletdecomp} we can construct $u$-Riesz basis for $\Frinf[Q]$ for all $s \in \R$, at least assuming $1\leq p<\infty$ and $1\leq q < \infty$.

Another question is whether a similar theorem for the cube $Q$ incorporating the exceptional values can be proven for the Besov spaces - the question is how to reinforce Besov spaces. In Theorem 6.28 and Theorem 6.30  of \cite{Tri08} Triebel proved the wavelet decomposition for the $B$- and $F$-spaces with the same exceptional values $s-\frac{k}{p} \in \N_0$. There are sharp Hardy inequalities also for the critical Besov spaces, see \cite[Theorem 16.2]{Tri01}, which depend on $q$ in contrast to the Triebel-Lizorkin spaces.

Furthermore, one cannot define the refined localization spaces for $B$-spaces as in Definition \ref{FrlocDef} for the $F$-spaces whenever $p \neq q$ - if $p=q$, then $\BR=\FR$ and there are no problems. But instead of $\Frloc[\Om]$ one can also use $\Ft[\Om]$. There is a counterpart of the crucial equivalent Characterization \ref{rlocequi} for the $B$-spaces, see \cite[Section 5.12]{Tri01}, at least for bounded $C^{\infty}$-domains. One gets for $1<p<\infty$, $1\leq q \leq \infty$ and $s>0$ 
\begin{align*}
 \|f|\Bt\| \sim \|f| \BO \| + \left(\int_0^{\infty} t^{-sq} \left(\int_{\Om^t} |f(x)|^p \ dx \right)^{\frac{q}{p}} \ \frac{dt}{t} \right)^{\frac{1}{q}}
\end{align*}
with
\begin{align*}
 \Om^t=\{x \in \Om: d(x) <t \}.
\end{align*}
The situation for $B$-spaces is far away from being satisfactory solved. There is a lot of work to do to arrive at a similar theorem as Theorem \ref{waveletdecomp}. Until now, it is not even clear to the author how the reinforced properties for the $B$-spaces would look like. 

\section{The situation for general domains - the domain problem}
\subsection{Known results}
The typical idea to construct a wavelet basis for function spaces on a general domain $\Om$ is to decompose the domain into simpler standard domains - this is usually called the domain problem. The starting point for such decompositions where the papers of Ciesielski and Figiel \cite{CiF83A}, \cite{CiF83B} and \cite{Cie84} dealing with spline bases for spaces of differentiable functions as well as classical Sobolev and Besov spaces on compact $C^{\infty}$ manifolds. Similar approaches and extensions were given in \cite{Dah97}, \cite{DaS99}, \cite{Dah01}, \cite{Coh03}, \cite{HaS04}, \cite{JoK07} and \cite{FoG08}.

A different approach was given in \cite{Tri08} which is the point of departure of this thesis. On the one hand he construced wavelet (Riesz) frames for spaces $\FO[\Om]$ for $C^{\infty}$-domains $\Om$ with the natural exceptional values $s-\frac{1}{p} \in \N_0$ in \cite[Theorem 5.27]{Tri08}. But he was not able to show that there is a Riesz basis for general dimensions and general smoothness parameter $s$ - see Theorem 5.35 for small dimensions.

On the other hand he constructed $u$-Riesz basis for $\FO[\Om]$ where $\Om$ is an $n$-dimensional ball with the exceptional values $s-\frac{k}{p} \notin \N_0$ for $k\in \{1,\ldots,n-1\}$, see \cite[Theorem 5.38]{Tri08}, and for $\FO[\Om]$ where $\Om$ is an $n$-dimensional cellular domain with the exceptional values $s-\frac{k}{p} \notin \N_0$ for $k\in \{1,\ldots,n\}$, see \cite[Theorem 6.30]{Tri08}. The definition of cellular domains can be found in \cite[Definition 5.40]{Tri08}. Roughly speaking, cellular domains are unions of diffeomorphic images of cubes. His construction of $u$-Riesz bases on cellular domains is based on the $u$-Riesz bases for the unit cube $Q$. In the next section we discuss if we can use this procedure also for reinforced function spaces $\Frinf[Q]$ to get $u$-Riesz bases also on cellular domains for the exceptional cases. 

Until now there seems to be not too much known how the shape of the domains influence the exceptional values for the existence of $u$-Riesz basis in the sense of Definition \ref{u-Riesz}. Sufficient conditions exist for $C^{\infty}$-domains, balls and cellular domains but it is nearly totally unclear when they are also necessary. 

\subsection{Extension of reinforced function spaces to cellular domains}
\label{cell}
So far we managed to construct $u$-Riesz basis for reinforced function spaces $\Frinf[Q]$ on the cube $Q$, or more general on an $n$-dimension polyhedron. It is clear that one can transfer $u$-Riesz basis to reinforced function spaces (suitably defined) on diffeomorphic images of a cube resp.\ of a polyhedron. 

Triebel defined cellular domains in \cite[Definition 6.9]{Tri08} using the notation of Lipschitz domains from \cite[Definition 3.4(iii)]{Tri08}:
\begin{Definition}
 A domain $\Om \subset \R^n$ is called cellular if it is a bounded Lipschitz domain which can be represented as
\begin{align*}
 \Om=\left(\bigcup_{l=1}^L \bar{\Om}_{l}\right)^{\circ} \text{ with } \Om_{l}\cap \Om_{l'}=\emptyset \text{ if }  l \neq l',
\end{align*}
such that each $\Om_{l}$ is diffeomorphic to an n-dimensional polyhedron (cube).
\end{Definition}

In \cite[Theorem 6.32]{Tri08} Triebel extended the constructed $u$-Riesz basis for the cube $Q$ and the non-exceptional values to $\FO[\Om]$ where $\Om$ is a cellular domain. Every $C^{\infty}$-domain, so for instance the unit ball $B$ in $\R^n$, is a cellular domain. The idea for the construction of $u$-Riesz bases on cellular domains is simple: Essentially a cellular domain is a union of cubes which have wavelet bases by \cite[Theorem 6.30]{Tri08}. One has to take care about the faces which can belong to more than one cube of the decomposition. But this does not make any problems for the non-exceptional cases.

If we now want to apply the same decomposition to a reinforced function space on a cellular domain and we need to fulfil a reinforce property $R_l^{r,p}$ at a boundary $\Gamma_l$, then it can and will happen that the function $f \in \Frinf[\Om]$ needs to fulfil reinforce properties $R_l^{r,p}$ at a boundary inside the domain which is very unnatural. Even worse when we consider that the decomposition of a cellular domain won't be unique in general. We could define reinforced Triebel Lizorkin spaces $\Frinf[\Om]$ in dependency on the decomposition of $\Om$ and construct $u$-Riesz bases for these space but it would not make too much sense since the reinforced properties $R_l^{r,p}$ at boundaries $\Gamma_l$ inside the domain are unnatural conditions.
\begin{center}
 \includegraphics[scale=0.3]{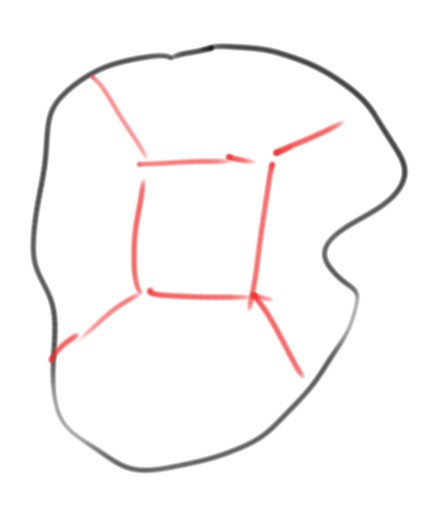}
 \captionof{figure}{A $C^{\infty}$-domain decomposed into diffeomorphic images of a cube}
\end{center}

\subsection{An example - decomposition of the unit ball and unit sphere}
In Theorem 5.37 and Theorem 5.38 Triebel constructed a $u$-wavelet system which is a $u$-Riesz basis for the function space $\FO[B^n]$ where $B^n$ is the unit ball in dimension $n$. Essentially he used the following decomposition of the unit sphere $\partial B^n=S^{n-1}=\{x \in \R^n: |x|=1\}$
\begin{align*}
 S^{n-1}=S^{n-1}_- \cup S^{n-1}_+ \cup \{x=(x_1,\ldots,x_n) \in S^{n-1}: x_n=0 \}.
\end{align*}
with $S^{n-1}_+=\{x \in S^{n-1}: x_n>0\}.$ One decomposes the unit sphere into the northern half, the southern half and the equator. 

Now one can argue by induction: The third set on the right hand side (the equator) is isomorphic to $S^{n-2}$. The unit sphere $S^{1}$ in dimension $n=2$ is a torus and hence $\FO[S^1]$ has a wavelet basis, see also \cite[Theorem 1.37]{Tri08}. This construction produces exceptional values where it cannot be applied - $s-\frac{k}{p} \notin \N_0$ for $k\in \{1,\ldots,n-1\}$. Hence the question is whether we can add reinforce properties $R_l^{r,p}$ to ensure the construction in the exceptional cases. But this is not possible - at least in a desirable way. Otherwise, by the construction, one would require reinforce properties at the equator 
$\{x=(x_1,\ldots,x_n) \in S^{n-1}: x_n=0 \}$ of derivatives perpendicular to the equator - this means inside the domain $B^n$. Furthermore, the choice of the equator is not canonical. Hence we are in the same situation as in the previous Section \ref{cell}. Right now, we have no idea how to introduce natural reinforce properties to ensure wavelet decompositions for a subset of $\FO[B^n]$.

%% file: Erklaerung.tex
    
Hiermit erkläre ich,

\begin{itemize}
 \item dass mir die Promotionsordnung der Fakultät bekannt ist, \\
 \item dass ich die Dissertation selbst angefertigt habe, keine Textabschnitte oder Ergebnisse eines Dritten oder eigenen Prüfungsarbeiten ohne Kennzeichnung übernommen und alle von mir benutzten Hilfsmittel, persönliche Mitteilungen und Quellen in meiner Arbeit angegeben habe, \\
 \item dass ich die Hilfe eines Promotionsberaters nicht in Anspruch genommen habe und dass Dritte weder unmittelbar noch mittelbar geldwerte Leistungen von mir für Arbeiten erhalten haben, die im Zusammenhang mit dem Inhalt der vorgelegten Dissertation stehen, \\
 \item dass ich die Dissertation noch nicht als Prüfungsarbeit für eine staatliche oder andere wissenschaftliche Prüfung eingereicht habe. \\
 \item dass ich die gleiche, eine in wesentlichen Teilen ähnliche bzw.\ eine andere Abhandlung bei keiner anderen Hochschule als Dissertation eingereicht habe.
\end{itemize}

\vspace{5cm}

\begin{center}
 Ort, Datum \hfill Benjamin Scharf
\end{center}

%% file: CV_de_a.tex
\section*{Benjamin Scharf -- Lebenslauf}
\subsubsection*{Kontaktinformationen}
\begin{tabular}{p{.23\textwidth}p{.02\textwidth}p{.7\textwidth}}
Büro	& &  JenTower 18N01, Leutragraben 1, 07740 Jena \\
E-Mail & &     \url{benjamin.scharf@uni-jena.de} \\
Homepage & & \url{http://users.minet.uni-jena.de/~benscha}\\
\end{tabular}

\subsubsection*{Persönliche Informationen}
\begin{tabular}{p{.23\textwidth}p{.02\textwidth}p{.7\textwidth}}
Geboren & & 29. November 1985 in Jena, Deutschland\\
Nationalität & &  Deutsch \\
\end{tabular}

\vspace{-0.3cm}
\subsubsection*{Forschungsinteressen}
\begin{tabular}{p{.23\textwidth}p{.02\textwidth}p{.7\textwidth}}&&Harmonische Analysis, Funktionalanalysis, Funktionenräume, Waveletcharakterisierungen, Hochdimensionale Analysis
\end{tabular}

\vspace{-0.5cm}
\subsubsection*{Ausbildung}
\begin{tabular}{p{.23\textwidth}p{.02\textwidth}p{.7\textwidth}}

09/1992 -- 05/2004  	& & Schule, Abitur am Carl-Zeiss-Gymnasium Jena (Spezialschulteil), Note 1,0 \\

10/2004 -- 03/2009	& & Mathematikstudium mit Nebenfach Informatik an der Friedrich-Schiller-Universität Jena \\

02/2007 -- 03/2009	& & Stipendium der Studienstiftung des Deutschen Volkes \\

03/2009  & &	Diplom Mathematik mit Auszeichnung, Note 1,0 \newline
	        Diplomarbeit: \textit{Atomare Charakterisierungen vektorwertiger Funktionenräume} \newline
		Betreuer: Hans-Jürgen Schmeißer\\
10/2010 &	&Ausgezeichnet mit dem Examenspreis des Dekans \\

08/2009 -- 02/2013  & & Promotionsstudent in der Forschungsgruppe \textit{Funktionenräume} am Mathematischen Institut der Friedrich-Schiller-Universität Jena \newline
		Promotionsstipendiat der Studienstiftung des Deutschen Volkes \\
		
02/2013 & 	& Promotion in Mathematik zum Dr. rer. nat.  \newline	
Promotionsthema: \textit{Wavelets in Funktionenräumen auf zellulären Gebieten und Mannigfaltigkeiten} \newline
		Betreuer: Hans-Jürgen Schmeißer und Hans Triebel \\
\end{tabular}

\subsubsection*{Publikationen}
\begin{itemize}
\item[3] B. Scharf, \textit{Atomic representations in function spaces and applications to pointwise multipliers and diffeomorphisms, a new approach}, Math. Nachr. \textbf{286} (2013), no. 2--3, 283--305.	
\item[2] B. Scharf, H.-J. Schmei\ss er, and W. Sickel, \textit{Traces of vector-valued Sobolev Spaces}, Math. Nachr. \textbf{285} (2012), no. 8--9, 1082--1106.
 \item[1] B. Scharf, \textit{Local means and atoms in vector-valued function spaces}, Jenaer
Schriften zur Mathematik und Informatik (2010), Math/Inf/05/10.
\end{itemize}

\subsubsection*{Vorträge}
\begin{itemize}
 \item[3] \textit{Wavelets for reinforced function spaces on cellular domains} 
  \newline Workshop Applied Coorbit theory, 2012, Wien, Österreich.\newline International Conference Function Spaces X, 2012, Poznan, Polen.
 \newline Mathematisches Seminar der Beijing Normal University, 2011, Peking, China.
 \item[2] \textit{Pointwise multipliers and diffeomorphisms in function spaces} \newline  8th International Conference FSDONA, 2011, Tabarz, Deutschland.
 \newline Mathematisches Seminar des Steklov Mathematical Institute, 2010, Moskau, Russland.	
 \item[1] \textit{Equivalent norms and characterizations for
vector-valued function spaces} \newline Workshop on Smoothness, Approximation, and Function Spaces, 2010, Oppurg, Deutschland.
\end{itemize}
\subsubsection*{Lehre}
\begin{tabular}{p{.23\textwidth}p{.02\textwidth}p{.7\textwidth}}

10/2007 --  07/2008 & & Übungsleiter \textit{Höhere Analysis 1} und \textit{Höhere Analysis 2} bei Prof. Hans-Gerd Leopold \\

10/2008 --  07/2009 & & Übungsleiter \textit{Analysis 1} und \textit{Analysis 2} bei Prof. Hans-J\"urgen Schmei\ss er \\

10/2009 --  07/2010 & & Übungsleiter \textit{Analysis 1} und \textit{Analysis 2} bei Prof. Bernd Carl \\

10/2010 --  07/2011 & & Übungsleiter \textit{Höhere Analysis 1} und \textit{Höhere Analysis 2} bei Prof. Dorothee D. Haroske \\

10/2011 --  07/2012 & & Übungsleiter \textit{Algebra 1} und \textit{Algebra 2} bei Prof. Klaus Haberland \\

10/2012 --  02/2012 & & Übungsleiter \textit{Analysis 3} bei Prof. Hans-J\"urgen Schmei\ss er
\end{tabular}

\subsubsection*{Sonstiges}
\begin{tabular}{p{.23\textwidth}p{.02\textwidth}p{.7\textwidth}}

Sprachenkenntnisse	& & Deutsch, Muttersprache  \newline
	Englisch, sicher in Wort und Schrift \newline
	Latein, Schulkenntnisse 6 Jahre \newline
	Russisch, Grundkenntnisse \\

IT-Kenntnisse & & Mathematische Software (Matlab, Scilab, Gnuplot) \newline
		 Office (Word, Excel, Powerpoint, LaTeX)  \newline
		 Programmiersprachen (Java, Oberon, C++) \newline
		 Linux, Windows \\	
	
Aktivitäten und Interessen	& & 
			   Betreuer und Redakteur der Zeitschrift \textit{Die Wurzel}\newline
Basketball als Trainer, Sportler, Schiedsrichter und Zuschauer \newline
Finanzielle Bildung
 \end{tabular}
\newline

\vspace{0.4cm}
Jena, 08.02.2013